\documentclass{amsart}

\usepackage[a-2b,mathxmp]{pdfx}[2021/12/22]  
\usepackage{colorprofiles}

\usepackage{amssymb}
\let\amslrcorner\lrcorner
\usepackage{amsfonts,color}
\usepackage{mathabx}

\usepackage[utf8]{inputenc}
\usepackage{mathtools}
\usepackage{amsthm,wasysym}
\usepackage[english]{babel}

\usepackage{adjustbox}
\usepackage{eucal}
\usepackage{graphicx}

\usepackage{hyperref}
\usepackage{subcaption}

\usepackage{amsmath}
\usepackage{float}
\usepackage{stmaryrd}
\usepackage{accents}

\usepackage{pdfsync}
\usepackage{tikz}
\usetikzlibrary{shadows}
\usetikzlibrary{intersections}
\usetikzlibrary{shapes.geometric}
\usetikzlibrary{calc, fadings}
\usetikzlibrary{decorations.pathreplacing}
\usetikzlibrary{decorations.pathmorphing}
\usetikzlibrary{decorations.text}
\usetikzlibrary{decorations.markings}
\usetikzlibrary{chains}
\usetikzlibrary{overlay-beamer-styles}
\usetikzlibrary{patterns,backgrounds}
\usetikzlibrary{3d}          
\usetikzlibrary{perspective} 
\usetikzlibrary{cd}

\usepackage{pgfplots}
\usepackage{tikz-3dplot}

\pgfplotsset{compat=1.15}
\usepackage{mathrsfs}
\usetikzlibrary{arrows}
\usetikzlibrary {arrows.meta}

\def\centerarc[#1](#2)(#3:#4:#5)
{ \draw[#1] ($(#2)+({#5*cos(#3)},{#5*sin(#3)})$) arc (#3:#4:#5); }

\usepackage{chngcntr}
\counterwithin{equation}{section}

\newtheorem{theorem}{Theorem}[section]
\newtheorem{lemma}[theorem]{Lemma}
\newtheorem{corollary}[theorem]{Corollary}
\newtheorem{proposition}[theorem]{Proposition}
\theoremstyle{definition}
\newtheorem{definition}[theorem]{Definition}
\theoremstyle{remark}
\newtheorem{remark}[theorem]{Remark}

\allowdisplaybreaks[1]

\newcommand{\VV}{{\mathcal{V}}}
\newcommand{\Vo}{\mathring{\VV}}

\newcommand{\om}{\varOmega}
\renewcommand{\Omega}{\om}
\def\d{\partial}

\newcommand{\vphi}{\varphi}
\newcommand{\veps}{\varepsilon}
\newcommand{\eveps}{\hat{\veps}}
\newcommand{\gveps}{\varepsilon}
\newcommand{\Xm}[1]{\mathfrak{X}({#1})}

\renewcommand{\forall}{\text{ for all }}

\newcommand{\D}{\ensuremath{\mathcal{D}}}
\newcommand{\Sc}{\ensuremath{\mathcal{S}}}
\newcommand{\E}{\ensuremath{\mathscr{E}}}
\newcommand{\Eint}{\ensuremath{\mathring{\mathscr{E}}}}
\newcommand{\Ebnd}{\ensuremath{\mathscr{E}_\partial}}
\newcommand{\Fint}{\ensuremath{\mathring{\mathscr{F}}}}
\newcommand{\Fbnd}{\ensuremath{\mathscr{F}_\partial}}

\newcommand{\dhat}[1]{{\bar{{#1}}}}
\newcommand{\gex}{{\dhat{g}}}  
\newcommand{\gpar}{{g}}  

\newcommand{\gappr}{{g_h}}

\renewcommand{\div}{\mathrm{div}}
\DeclareMathOperator{\curl}{\mathrm{curl}}

\DeclareMathOperator{\inc}{\mathrm{inc}}
\DeclareMathOperator{\INC}{\mathrm{Inc}}

\DeclareMathOperator{\grad}{\mathrm{grad}}
\newcommand{\W}{\mathord{\adjustbox{valign=B,totalheight=.6\baselineskip}{$\bigwedge$}}}
\newcommand{\TT}{\mathcal{T}}
\newcommand{\og}{\omega}
\newcommand{\mt}[1]{[{#1}]}

\newcommand{\TM}[2]{\mathcal{T}_{{#1}}^{{#2}}(\om)}
\newcommand{\XM}{\mathfrak{X}(\om)}
\newcommand{\WM}[1]{\W^{{#1}}(\om)}
\newcommand{\ip}[1]{\langle{#1}\rangle}

\newcommand{\jump}[2][]{\ensuremath{\ifthenelse{\equal{#1}{}}{\llbracket #2 \rrbracket}{\llbracket #2 \rrbracket}_{#1}}}
\newcommand{\jmp}[1]{\ensuremath{\llbracket #1 \rrbracket}}

\newcommand{\vol}[2][]{\omega_{{#2}}(\ifthenelse{\equal{#1}{}}{g}{#1})}
\newcommand{\vo}[1]{\omega_{{#1}}}
\newcommand{\volform}{\omega}
\newcommand{\Rog}{\widetilde{\Riemann\omega}}

\newcommand{\RogU}{\widetilde{\Curvature\omega}}
\newcommand{\RogA}{\widetilde{\Curvature\omega}}
\newcommand{\RoggA}[1]{\widetilde{\Curvature\omega}({#1})}
\newcommand{\x}{\ensuremath{\times}}								

\newcommand{\N}{\ensuremath{\mathbb{N}}}

\newcommand{\R}{\ensuremath{\mathbb{R}}}

\newcommand{\T}{\ensuremath{\mathscr{T}}} 
\newcommand{\F}{\ensuremath{\mathscr{F}}} 
\newcommand{\Teuc}{\tilde{T}}

\newcommand{\Ltwo}[1][]{\ensuremath{L^2\ifthenelse{\equal{#1}{}}{}{(#1)}}}
\newcommand{\Linf}[1][]{\ensuremath{L^{\infty}\ifthenelse{\equal{#1}{}}{}{(#1)}}}

\newcommand{\Winf}[1][]{\ensuremath{W^{1,\infty}\ifthenelse{\equal{#1}{}}{}{(#1)}}}
\newcommand{\Winfh}[1][]{\ensuremath{W_h^{1,\infty}\ifthenelse{\equal{#1}{}}{}{(#1)}}}
\newcommand{\Wtinf}[1][]{\ensuremath{W^{2,\infty}\ifthenelse{\equal{#1}{}}{}{(#1)}}}

\newcommand{\Wsp}[1][]{\ensuremath{W^{s,p}\ifthenelse{\equal{#1}{}}{}{(#1)}}}
\newcommand{\Wsph}[1][]{\ensuremath{W_h^{s,p}\ifthenelse{\equal{#1}{}}{}{(#1)}}}
\newcommand{\Hone}[1][]{\ensuremath{H^1\ifthenelse{\equal{#1}{}}{}{(#1)}}}

\newcommand{\Hsh}[1][]{\ensuremath{H^s_h\ifthenelse{\equal{#1}{}}{}{(#1)}}}
\newcommand{\Htwo}[1][]{\ensuremath{H^2\ifthenelse{\equal{#1}{}}{}{(#1)}}}
\newcommand{\Htwoz}[1][]{\ensuremath{H^2_0\ifthenelse{\equal{#1}{}}{}{(#1)}}}
\newcommand{\Hmtwo}[1][]{\ensuremath{H^{-2}\ifthenelse{\equal{#1}{}}{}{(#1)}}}

\newcommand{\Cone}[1][]{\ensuremath{C^1\ifthenelse{\equal{#1}{}}{}{(#1)}}}

\newcommand{\Cinf}[1][]{\ensuremath{C^{\infty}\ifthenelse{\equal{#1}{}}{}{(#1)}}}

\newcommand{\Riemann}{\mathcal{R}}
\newcommand{\Scalar}{S}
\newcommand{\Einstein}{G}
\newcommand{\Ricci}{\mathrm{Ric}}
\newcommand{\sff}{\mathrm{I\!I}}
\newcommand{\Curvature}{\mathcal{Q}}
\newcommand{\CurvatureOther}{\mathfrak{Q}}
\newcommand{\Gauss}{K}
\newcommand{\GeodCurv}{\kappa}

\newcommand{\nv}{\hat{\nu}}
\newcommand{\cnv}{\hat{\mu}}
\newcommand{\tv}{{\hat{\tau}}}       
\newcommand{\gn}{{{{\nu}}}}  
\newcommand{\gcn}{{{{\mu}}}}  
\newcommand{\gt}{{{{\tau}}}} 

\newcommand{\Regge}[1][]{\ensuremath{\mathrm{Reg}\ifthenelse{\equal{#1}{}}{}{(#1)}}}
\newcommand{\oRegge}[1][]{\ensuremath{\mathring{\mathrm{Reg}}\ifthenelse{\equal{#1}{}}{}{(#1)}}}
\newcommand{\RR}{\Regge}

\newcommand{\idop}{\mathrm{id}}

\newcommand{\pder}[2]{\ensuremath{\partial_{#2}{#1}}}

\DeclareMathOperator{\tro}{\mathrm{tr}}
\newcommand{\tr}[2][]{\ensuremath{\,\mathrm{tr}\ifthenelse{\equal{#1}{}}{}{_{#1}}(#2)}}

\newcommand{\Pol}{\ensuremath{\mathcal{P}}}
\newcommand{\RegInt}[1][]{\mathcal{I}_{h}^{\Regge\ifthenelse{\equal{#1}{}}{}{,#1}}}
\newcommand{\OptInt}[1][]{\mathcal{I}_{h}^{\ifthenelse{\equal{#1}{}}{}{#1}}}

\newcommand{\Eucl}{\delta}

\newcommand{\Proj}{Q}
\newcommand{\Fr}{{F}}

\newcommand{\Dr}{{D}}
\newcommand{\Er}{{E}}

\newcommand{\Fp}{{\top}}
\newcommand{\nablaF}{\nabla^{\Fp}}
\newcommand{\nablaFR}{\nabla^{\Fr}}
\newcommand{\nablans}[2]{(\nabla {{#1}})_{{#2}}}

\newcommand{\divFR}{\div^{\Fr}}

\newcommand{\tnnt}[1]{{({#1})}_{\Fr\gn\gn \Fr}}
\newcommand{\Atnnt}{A_{\Fr\gn\gn \Fr}}   
\newcommand{\Psitnnt}{\Psi_{\Fr\gn\gn \Fr}}   
\newcommand{\Btnnt}{B_{\Fr\gn\gn \Fr}}   
\newcommand{\dotAtnnt}{\dot{A}_{\Fr\gn\gn \Fr}}
\newcommand{\Atest}{\mathcal{A}}
\newcommand{\Atesto}{\mathring{\mathcal{A}}}
\newcommand{\Ao}{\Atesto}
\newcommand{\Utest}{\mathcal{U}}
\newcommand{\Utesth}{\mathcal{U}(\T_h)}
\newcommand{\Utesto}{\mathring{\Utest}}
\newcommand{\Btest}{\mathcal{B}}
\newcommand{\Btesto}{\mathring{\mathcal{B}}}

\newcommand{\mapUA}{\mathbb{A}}

\newcommand{\gdot}{{\dot{g}}}
\newcommand{\Perm}{\mathfrak{S}}

\newcommand{\Lsigma}{L_{\sigma}}
\newcommand{\Ldotg}{L_{\gdot}}
\newcommand{\LsigmaD}{L_{\sigma_D}}

\newcommand{\LdotgF}{L_{\gdot_F}}

\newcommand{\nrm}[1]{\left\vvvert{{#1}}\right\vvvert}
\newcommand{\nrmt}[1]{\left\|\!\left\|{#1}\right\|\!\right\|}

\let\lrcorner\amslrcorner

\title[Generalized Riemann curvature]
{Generalizing Riemann curvature to Regge metrics}

\author[J.~Gopalakrishnan]{Jay~Gopalakrishnan}
\address{Portland State University, PO Box 751, Portland OR 97201, USA }
\email{gjay@pdx.edu}

\author[M.~Neunteufel]{Michael~Neunteufel}
\address{Portland State University, PO Box 751, Portland OR 97201, USA}
\address{Institute of Analysis and Scientific Computing, TU Wien, Wiedner Hauptstra{\ss}e 8--10, 1040 Wien, Austria}
\curraddr{Johann Radon Institute for Computational and Applied Mathematics \mbox{(RICAM)}, Austrian Academy of Sciences, Altenberger Stra{\ss}e 69, 4040 Linz, Austria}
\email{michael.neunteufel@ricam.oeaw.ac.at}
\author[J.~Sch\"oberl]{Joachim~Sch\"oberl}
\address{Institute for Analysis and Scientific Computing, TU Wien, Wiedner Hauptstr. 8-10, 1040 Wien, Austria}
\email{joachim.schoeberl@tuwien.ac.at}
\author[M.~Wardetzky]{Max~Wardetzky}
\address{Institute of Numerical and Applied Mathematics, University of G\"ottingen, Lotzestr. 16-18, 37083 G\"ottingen, Germany}
\email{email: wardetzky@math.uni-goettingen.de}

\begin{document}

\begin{abstract}

  In this paper, we propose a generalization of the Riemann curvature tensor on manifolds (of dimension two or higher) endowed with a Regge metric. Specifically, while all components of the metric tensor are assumed to be smooth within elements of a triangulation of the manifold, they need not be smooth across element interfaces, where only continuity of the tangential components is assumed. While linear derivatives of the metric can be generalized to Schwartz distributions, similarly generalizing the classical Riemann curvature tensor, a nonlinear second-order derivative of the metric, requires more care. We propose a generalization that combines the classical angle defect and jumps in the second fundamental form across element interfaces and 
  {argue its correctness.} Specifically, if a piecewise smooth metric approximates a globally smooth metric, then our generalized Riemann curvature tensor approximates the classical Riemann curvature tensor associated with the latter. Moreover, we show that if the metric approximation converges at some rate in a {mesh-dependent norm equivalent to the $L^2$ norm,} then the curvature approximation converges in the negative Sobolev space {$H^{-2}$, the dual space of $H^2_0$,} at the same rate, under additional assumptions. By appropriate contractions of the generalized Riemann curvature tensor, this work also provides generalizations of scalar curvature, the Ricci curvature tensor, and the Einstein tensor in any dimension.
	\\
	\vspace*{0.25cm}
	
\end{abstract}

\maketitle

\section{Introduction}
\label{sec:intro}

Regge calculus, introduced by Tullio Regge in the 1960s \cite{Regge61}, is now a firmly established technique in numerical relativity. It employs piecewise flat simplicial complexes to discretize the metric tensor of an $N$-dimensional manifold $(N\ge 2)$ via edge-length specifications and discretizes curvatures via angle defects (also called angle deficits). These angle defects, which are concentrated at $(N-2)$-subsimplices, were shown~\cite{Cheeger84} to converge to the scalar curvature in the sense of measures as the triangulation becomes finer. However, high-order approximations of the full Riemann curvature tensor in $N$ dimensions {have} remained elusive so far. This work provides such an extension, and a rigorous proof of convergence in a Sobolev norm (including rates of convergence), confirming its accuracy across all dimensions. The heart of the matter lies beyond utilitarian considerations in approximating curvature numerically. In fact, our central contribution is the identification and the analysis of a generalized notion of Riemannian curvature on triangulated manifolds with isometrically glued simplices, whose intrinsic metric is not necessarily flat.

Regge's edge-length prescription is equivalent to defining a constant (flat) metric tensor within each $N$-simplex (``element'') of the simplicial complex such that their tangential components are continuous across element interfaces~\cite[\S~II.A]{Sorkin75}. By using higher-order polynomial approximations of the metric tensor, rather than just constants, within each $N$-simplex, while maintaining the same tangential inter-element continuity, one can obtain better metric approximations on the simplicial complex. Such piecewise polynomial metrics, possessing inter-element tangential continuity, form our point of departure. {Extending the classical flat model introduced by Regge,} we refer to such metrics as ``Regge metrics''. Since these metrics are smooth within each element, one can view each element, by itself, as a Riemannian manifold with boundary. Imposing tangential continuity of the metric tensor across elements, precisely defined in \S\ref{ssec:tt-continuity} and referred to as ``$tt$-continuity'' throughout this paper, {is tantamount to isometrically gluing adjacent element manifolds.}

Two important questions arise. How can one craft a notion of Riemannian curvature from such a (possibly discontinuous) Regge metric? Will such a generalized notion of Riemannian curvature produce good approximations to classical Riemannian curvature if the Regge metric approximates a globally smooth metric? Let us start by outlining our approach to answering the first question. Obviously, in the interior of each $N$-simplex, classical formulas for the Riemann curvature tensor apply. These curvature contributions from the interior must be added to the angle defects, which provide curvature contributions concentrated at $(N-2)$-subsimplices (``bones''). We show that an additional type of curvature contribution is required, namely curvature contributions concentrated at $(N-1)$-subsimplices (``facets''). These contributions correspond to the jump of the second fundamental form across facets. We prove that this additional contribution suffices in order to obtain a generalization of the Riemann curvature tensor, for which high-order convergence rates can be established, when the Regge metric approximates a given smooth metric.

In two dimensions, the full Riemann curvature tensor, which encodes all intrinsic curvature information, reduces to the classical Gauss curvature. Numerous prior works have proposed and analyzed a generalization of the Gauss curvature. For example, in \cite{BKG21, Gaw20, GNSW2023, Str2020}, the authors considered three types of curvature contributions: an element contribution obtained by classical Gauss curvature formulas, a bone contribution at each vertex of the simplicial complex consisting of the angle deficit, and a facet contribution at each edge consisting of the jump of geodesic curvature. The latter can be transparently motivated by the Gauss--Bonnet theorem.
Nevertheless, since the Gauss--Bonnet theorem does not apply in higher dimensions, the extension of the facet contribution to higher dimensions is not obvious, so we have devoted subsection~\S\ref{ssec:motiv1} to motivate it.

\subsection{Contributions}
One central contribution of our paper is the convergence analysis of the generalized notion of Riemannian curvature. Our analysis relies on the \emph{linearization} of Riemannian curvature and of the other terms in the generalized Riemann curvature we propose. In two dimensions, as shown in~\cite{GNSW2023}, linearization of Gauss curvature leads to the so-called incompatibility operator, a second-order linear differential operator involving two curls, which also arises in elasticity when quantifying violation of St.~Venant's compatibility conditions. We provide {(Theorem~\ref{thm:distr_Riem_evol})} a corresponding linearization of the Riemann curvature tensor in higher dimensions. Similar linearizations have also been considered in the literature of Ricci flows~\cite{Top2006}. Comparing the linearization to the known incompatibility operators in two and three dimensions, we are led to a generalization of the covariant incompatibility operator in arbitrary dimensions (named ``$\INC$'' and defined in~\eqref{eq:1}). {Theorem~\ref{thm:distr_inc_curved} shows how to recover this ``$\INC$'' as an integral part $b(g;\cdot,\cdot)$ of the linearization computed in Theorem~\ref{thm:distr_Riem_evol}. The adjoint of ``$\INC$'', generalized to piecewise smooth metrics (Theorem~\ref{thm:distr_covariant_adjoint_inc}),} turns out to be a crucial ingredient in our analysis. We often refer to our generalized Riemann curvature as a ``distributional Riemann curvature'' (although it is technically not a distribution) because its linearization around a flat metric gives terms that correspond exactly to those of a distributional version of the linear incompatibility operator. Another key idea in our analysis is a variant of the Uhlenbeck trick~\cite{Ham1986, Top2006}. Namely, when a Regge metric approximates a smooth metric, comparing geometric quantities that depend on the metric is made easier if their metric dependence can be transformed into a common metric-independent space. We identify such a space in \S\ref{ssec:metr-indep-test}.

\subsection{A motivating calculation for the facet curvature measure}
\label{ssec:motiv1}

We here motivate why the jump of the second fundamental form across facets arises naturally in the definition of a generalized curvature tensor.  
 For this motivation, we impose the stronger assumption that the metric be continuous across each facet, rather than merely tangentially continuous. Under this stronger assumption, the one-sided Levi-Civita connections may have a jump across facets, but the connections themselves contain no Dirac delta contribution. Their distributional derivatives may nevertheless produce a facet-supported singular contribution to the curvature tensor. More precisely, in local coordinates,
$g,g^{-1}\in W^{1,\infty}_{\mathrm{loc}}$, and hence the Christoffel symbols satisfy $\Gamma\in L^\infty_{\mathrm{loc}}$. Consequently, the Riemann curvature of the piecewise smooth continuous metric is well-defined as a tensor-valued distribution: schematically, in coordinates, $\Riemann=\d\Gamma+\Gamma\Gamma$, where the product $\Gamma\Gamma$ is locally bounded and the derivatives $\d\Gamma$ may produce facet-supported Dirac measures.

In order to make our setup precise, consider a common facet $F$ of two adjacent elements $T_\pm$ in the simplicial complex, with metrics $g_\pm$ that are smooth in $T_{\pm}$ up to $F$ from their respective sides and that are continuous across $F$. Continuity of the metric across $F$ implies that along $F$ there is a \emph{unique} normal line field, and after choosing a coorientation, there is a unique corresponding unit normal field, denoted by $\gn$. As a matter of convention, let $\gn$ point from $T_{-}$ to $T_{+}$. Now consider a point $p \in F$, let  $X \in T_{p}F$ be a unit vector, and let $\gamma:(-\varepsilon,\varepsilon)\to F$ denote the \emph{intrinsic} $(F,g_F)$-geodesic in $F$ with $\gamma(0)=p$ and $\dot\gamma(0)=X$. This geodesic is well-defined since the induced metric $g_{F}$ is well-defined. Along $\gamma$,  consider the \emph{ambient} $g_\pm$-geodesics, denoted by
  $\exp^{g_\pm}_{\gamma(s)}(\pm t\,\gn(\gamma(s)))$, emanating from
  $\gamma(s)$ in direction $\pm \gn (\gamma(s))$.  This provides two ruled surfaces
  \begin{equation}
    \label{eq:normal-fan}
    S_\pm=\big\{
    \exp^{g_\pm}_{\gamma(s)}\big(\pm t\,\gn (\gamma(s))\big):\
    |s|<\varepsilon,\ 0\le t < \varepsilon \big\}
  \end{equation}
  for some $\varepsilon$ small enough, such that $S_\pm$ are smooth embedded
  surfaces in $T_\pm$.
  The two smooth 2-manifolds $S_\pm$ are isometrically glued along their common boundary curve
  $\gamma = S_+\cap S_-$. Let $\kappa_{\pm}$ denote the signed geodesic curvature of $\gamma$  in $S_{\pm}$, computed using the common conormal $\gn$ along $\gamma$.
Then the aforementioned 2D-results~\cite{BKG21, Gaw20,
    GNSW2023, Str2020} applied to $S = S_- \cup S_+$ yield that the \emph{singular part} of the distributional Gauss curvature $K$ of $S$ is equal to the difference of geodesic curvatures along $\gamma$, i.e.,
      \begin{equation*}
    K_{\mathrm{sing}}
    = (\kappa _{+}- \kappa_-)\,\delta_\gamma,
  \end{equation*}
  where  $\delta_\gamma$ denotes the Dirac measure of $\gamma$.  Using that $\kappa_\pm(s) = \sff_\pm\big(\dot\gamma(s),\dot\gamma(s)\big)$, where $\sff_\pm(X,Y) = g_\pm(\gn , \nabla^{g_\pm}_X Y)$ denote the second fundamental forms of $F$, yields
     \begin{equation}
    \label{eq:slice-measure}
    K_{\mathrm{sing}} = \jump{\sff}\left( \dot \gamma, \dot \gamma \right) \delta_\gamma
    ,
  \end{equation}
  where $\jump{\sff}= \sff_{+}-\sff_{-}$ denotes the jump of the second fundamental forms along $\gamma$.

Moreover, using that the curves $t \mapsto \exp^{g_\pm}_{\gamma(s)}(\pm t\,\gn(\gamma(s)))$ are $g_\pm$-geodesics in $S_{\pm}$, the Gauss
  equation (see~\cite[Chapter~8, Theorem~8.5]{Lee18}) applied to the smooth surfaces
  $S_\pm$ embedded in $T_\pm$ simplifies to
  \begin{equation}
    \label{eq:fan-gauss}
    \Riemann_\pm(\dot\gamma,\gn,\gn,\dot\gamma) = K_\pm 
    + \left | \sff^{S}_{\pm}(\dot \gamma, \gn) \right |^{2} .
  \end{equation}
Here $\Riemann_\pm$ denotes the Riemann curvature tensors of $T_{\pm}$, whereas $K_{\pm}$ and $\sff^{S}_{\pm}$ denote the Gauss curvatures and (vector-valued) second fundamental forms of $S_{\pm}$, respectively, where  $\sff^S_{\pm}(U, V)$ is defined as the orthogonal projection of
  $\nabla^{g_\pm}_U V$ to the normal bundle of $S_\pm$ for any $U$ and  $V$ in the tangent bundle of $S_\pm$.

The correction term $\left|\sff^{S}_{\pm}(\dot\gamma,\gn)\right|^{2}$ in \eqref{eq:fan-gauss} does not produce a singular contribution. Indeed, extend $\dot\gamma$ and $\gn$ to smooth tangent vector fields $E_\pm$ and $N_\pm$, respectively, on $S_\pm$, and let $q$ denote the piecewise function on $S$ defined by
$q_\pm:=\left|\sff^S_\pm(E_\pm,N_\pm)\right|^2$.
One-sided smoothness up to $\gamma$ implies that
$q\in L^\infty_{\mathrm{loc}}(S)$. Hence $q$ defines a regular distribution and has no singular component supported on $\gamma$. Therefore, the correction term in \eqref{eq:fan-gauss} does not contribute to the singular part of the curvature distribution.

Since $S$ intersects $F$ transversely and orthogonally along
$\gamma$, the measure on $\gamma$ induced by the facet measure
carries no additional Jacobian factor. Consequently, \eqref{eq:fan-gauss} and
\eqref{eq:slice-measure} suggest that the coefficient of
$\delta_F$ satisfies, pointwise along $F$,
\begin{equation}
  \label{eq:motiv-sff-jump}
  \Riemann_{\mathrm{sing}}(X,\gn,\gn,X)
  =
  \jump{\sff}(X,X)\,\delta_F,
  \qquad X\in T_pF.
\end{equation}
Polarization then gives
\[
  \Riemann_{\mathrm{sing}}(X,\gn,\gn,Y)
  =
  \jump{\sff}(X,Y)\,\delta_F,
  \qquad X,Y\in T_pF.
\]
This calculation therefore motivates the facet contribution
$\jump{\sff}\,\delta_F$ in the definition of a generalized
Riemann curvature tensor. For continuous metrics, the same relation also
follows from the standard hypersurface jump formula
\cite[Theorem~3.8 and Section~4.3]{LeFlochMardare2007}.

This concludes our motivation for why the jump of the second
fundamental form arises naturally in the definition of a generalized
Riemann curvature tensor. We stress that the calculation is a
consistency argument restricted to the subclass of continuous
piecewise smooth metrics. It shows that any generalized curvature
that agrees with classical distributional curvature on this subclass
must reproduce the facet contribution $\jump{\sff}\,\delta_F$ there.
The argument does not claim that this is the complete
singular contribution for a merely $tt$-continuous metric.

In the remainder of the article, we work under the weaker assumption
of $tt$-continuity.  In this setting, the usual nonlinear
coordinate expression for the Riemann curvature does not in general
define a canonical tensor-valued Schwartz distribution.  We therefore
define the generalized curvature through its action on test tensors,
combining the smooth elementwise curvature with facet contributions
determined by jumps of the second fundamental form and bone
contributions determined by angle defects.  The remainder of the paper
develops this construction and proves its convergence to the classical
Riemann curvature for smooth metrics.

\begin{figure}[H]
  \centering
  \tdplotsetmaincoords{60}{110}
  \begin{tikzpicture}
    [tdplot_main_coords, >=stealth,
    declare function=
    {
      pfft(\x)=pi+0.3*sin(deg(\x));
    },
    scale=1.6,
    ] 

    \def\fanlength{1.2}
    \coordinate (r0) at (0.3, {0.7+\fanlength*0.06},
      {pfft(2*0.7)+\fanlength*0.5});
    \coordinate (r1) at (0.3, {0.5*pi+\fanlength*0.06},
      {pfft(2*0.5*pi)+\fanlength*0.5});
    \coordinate (r2) at (0.3, {0.5*pi-\fanlength*0.06},
      {pfft(2*0.5*pi)-\fanlength*0.5});
    \coordinate (r3) at (0.3, {0.7-\fanlength*0.06},
      {pfft(2*0.7)-\fanlength*0.5});

    \fill[cyan!70!brown, opacity=0.5]
    plot[variable=\x,domain=0.7:0.5*pi,smooth] (0.3,\x,{pfft(2*\x)}) --
    (r2) -- (r3) -- cycle node [pos=0.0, text=cyan!50!brown, below, opacity=1]
    {$S_+$};

    \coordinate (a0) at (0, 0, {pfft(0)});
    \coordinate (a1) at (0, 1.1*pi, {pfft(2*1.1*pi)});
    \coordinate (a2) at (1.1*pi, 1.1*pi, {pfft(2*1.1*pi)});
    \coordinate (a3) at (0, 0.4, -0.2);
    \coordinate (b0) at (0, 2, 5);

    \draw[color=cyan!50!brown, line width=0.7mm] (a3) to[out=110, in=-120] (a0);
    \draw[color=cyan!50!brown, line width=0.7mm] (a3)--(a2);
    \draw[opacity=0.3, color=cyan!50!brown, line width=0.7mm] (a3)--(a1);

    \path
    [
    opacity=0.4, draw, draw opacity=1,
    color=blue,
    line width=0.7mm,
    left color=blue, right color=blue,
    middle color=blue!20, shading angle=72
    ]
    plot[variable=\x,domain=0:1.1*pi,smooth] (0,\x,{pfft(2*\x)}) --
    plot[variable=\x,domain=1.1*pi:0,smooth] (\x,\x,{pfft(2*\x)});

    \fill[red!50, opacity=0.8]
    plot[variable=\x,domain=0.7:0.5*pi,smooth] (0.3,\x,{pfft(2*\x)}) --
    (r1) -- (r0) -- cycle;
    \node[text=red, opacity=1, above right] at (r1) {$S_-$};

    \draw[blue, very thick]
    plot[variable=\x,domain=0.7:0.5*pi,smooth] (0.3,\x,{pfft(2*\x)})
    node[below right, text=blue] {$\gamma $};

    \draw[->, line width=0.5mm, green!50!blue]
    (0.3, 0.3*pi, {pfft(2*0.3*pi)})--++(0, -0.06, -0.5)
    node [right] {$\gn$};

    \draw[->, line width=0.5mm, blue, opacity=1]
    (0.3, 0.3*pi, {pfft(2*0.3*pi)})--++(0, 0.5, -0.06)
    node[pos=1.2] {$X$};

    \draw[->, line width=0.5mm, red!50!blue]
    (0.3, 0.3*pi, {pfft(2*0.3*pi)})--++(0, 0.06, 0.5)
    node [right] {$-\gn$};

    \draw[color=red!50, line width=0.7mm] (a0) to[out=60, in=-150] (b0);
    \draw[color=red!50, line width=0.7mm] (a2) to[out=60, in=-100] (b0);
    \draw[opacity=0.7, color=red!50, line width=0.7mm] (a1)--(b0);

    \node at ($(b0)!0.2!(a1)$) [below, red] {$T_-$};
    \node at ($(a3)!0.2!(a0)$) [above, cyan!50!brown] {$T_+$};
    \node at (a2) [above, blue] {$F\quad$};

  \end{tikzpicture}
  \caption{An illustration of two $N$-dimensional element manifolds
    $T_\pm$, equipped with smooth one-sided metrics $g_\pm$ that agree
    continuously across their shared facet $F$, and a piecewise smooth
    two-dimensional embedded manifold
    $S_+\cup S_-$ around a point on $F$. (The drawing is for $N=3$.)}
  \label{fig:sectional}
\end{figure}

\subsection{Related prior results}

Comprehensive overviews of the development and impact of Regge calculus over the past fifty years, with broad applications in relativity and quantum mechanics, can be found in the works of \cite{williams92, RW00, BOW18}. As previously mentioned, the first rigorous proof demonstrating the convergence of Regge's angle deficit to the scalar curvature, within a sequence of appropriate triangulations in the sense of measures, was achieved in~\cite{Cheeger84} by Cheeger, M\"uller, and Schrader. 
For a fixed piecewise flat (lowest-order Regge finite element) metric, Christiansen subsequently showed that the curvature of mollified smooth metrics converges to the angle deficit in the sense of measures~\cite{Christiansen13, Chr2024}. 
In Discrete Differential Geometry (DDG), angle deficit methods for approximating the Gauss curvature on triangulated surfaces are classical. {Convergence in the $\Linf$-norm up to quadratic order has been proven \cite{BCM03, Xu06, XX09} on triangulations satisfying certain conditions. However, convergence is not guaranteed  {on} general irregular grids.} In~\cite{LX07}, Regge's concept of angle deficit was extended to quadrilateral meshes. Noteworthy among the results applicable to higher-dimensional manifolds is the proof of convergence for approximated Ricci curvatures of isometrically embedded hypersurfaces $\subset \R^{N+1}$, as presented in \cite{Fritz13} and later utilized for Ricci flows in \cite{Fritz15}. 

{Regge's piecewise flat metrics fit naturally into the intrinsic viewpoint with curvature as a \emph{measure} supported on the codimension 2 strata. In Alexandrov's framework, the curvature of a non-smooth surface is a Radon (Borel) measure determined by the metric. For polyhedral (piecewise flat) metrics, it is purely atomic with masses equal to the vertex angle defects \cite{alexandrovAlexandrovSelectedWorks2006}. Alexandrov and Zalgaller extended these models within the class of surfaces of bounded integral curvature and developed approximation by polyhedral metrics together with a Gauss--Bonnet theorem in the measure-theoretic setting \cite{aleksandrovIntrinsicGeometrySurfaces}. Cheeger studied extensions of classical results in geometric analysis, including pseudomanifolds, in which such curvature measures appear naturally \cite{cheegerHodgeTheoryRiemannian1980,cheegerSpectralGeometrySingular1983}. In \cite[p.139f]{cheegerHodgeTheoryRiemannian1980}, Cheeger points out a remark of Gromov that ``Although this condition (for a piecewise flat pseudomanifold to have positive defect angles) is sometimes taken to be the analogue of positive sectional curvature [\dots] Gromov has suggested that positive curvature operator might be a more appropriate interpretation''. Our results confirm this observation: for Regge metrics, classical sectional curvature is not well-defined at the codimension-2 singular strata, yet a distributional notion of positive curvature operator can  be formulated---see Remark~\ref{rem:pos_curv_op},
  which extends the notion of positive curvature operator in the
    classical smooth setting as discussed in
    \cite[Chapter~9]{Peter16}.}

From the numerical analysis side, many of these constructions admit a clean expression in \emph{finite element exterior calculus} (FEEC) \cite{arnold06, arnold10}. Here, a \emph{Regge metric} can be viewed as a piecewise flat tensor field with continuous tangential--tangential components across faces. This recognition dates back to Sorkin \cite{Sorkin75}. Christiansen introduced subsequent developments in finite element structures for Regge calculus in \cite{Christiansen04}, and these elements gained popularity in FEEC under the designation ``Regge finite elements''~\cite{li18}. Regge elements approximating metric and strain tensors were further extended to arbitrary polynomial orders on triangles, tetrahedra, and higher-dimensional simplices in \cite{li18}, as well as quadrilaterals, hexahedra, and prisms in \cite{Neun21}. The effectiveness of Regge elements in discretizing portions of the Kr\"oner complex or the elasticity complex was explored in \cite{AH21, Christiansen11, Hauret13}. Notably, properties of Regge elements were leveraged to devise a method circumventing membrane locking for general triangular shell elements in~\cite{NS21}.

Finally, recent works successfully formulated distributional curvatures of Regge metrics and proved convergence rates. In~\cite{Gaw20}, the Regge finite elements on 2-manifolds were used to develop a high-order Gauss curvature approximation. The key ingredient was an integral representation of the angle deficit extended to high order. This formulation, which can be seen as a covariant version of the Hellan--Herrmann--Johnson (HHJ) method \cite{Com89}, enabled rigorous proofs of convergence at specific rates. This approach was reformulated in~\cite{BKG21} in terms of a nonlinear distributional Gauss curvature, consisting of elementwise Gauss curvature, jumps of geodesic curvature at edges, and atomic contributions (angle deficit) at vertices as sources of curvature---see also~\cite{Str2020} for a derivation of a curvature notion on singular surfaces in the sense of measures. Under the assumption that the metric approximation is produced by the canonical Regge interpolation operator, improved convergence for the distributional Gauss curvature was proven in~\cite{GNSW2023}. The first extension of distributional curvatures in dimension $N\geq 3$ has been proposed in \cite{GN2023}, where the distributional scalar curvature has been defined and analyzed in any dimension. Therein, the jump of the mean curvature at codimension 1 facets and the angle deficit at codimension 2 boundaries have been used as additional sources of curvature. Terms only up to codimension 2 boundaries are considered, which reflects the fact that second-order derivatives of the metric tensor are needed to compute curvatures. Further, in \cite{GN2023}, the $H^{-2}$-norm has been taken to measure the approximation error. The analysis there showed that, in $N=2$, convergence rates of $\mathcal{O}(h^{k+1})$ for Regge metrics of polynomial order $k\geq 0$ are obtained, whereas, for $N\geq 3$, approximation of order $k\geq 1$ is necessary. Here, $h$ denotes the meshsize of the triangulation. A significant difference between dimensions $N=2$ and $N\geq 3$ is the appearance of a ``second part'' of the integral representation besides the covariant HHJ bilinear form. This second part, responsible for the lack of convergence in the lowest-order case $k=0$ (piecewise constant metrics), has been identified as the distributional Einstein tensor \cite{GN2023b}, extending the classical one to non-smooth Regge metrics.

\subsection{Outline}
In Section~\ref{sec:distr_Riemann}, we motivate and define a generalized Riemann curvature for any given $tt$-continuous Regge metric on an $N$-manifold.
We show that our definition reduces to the distributional Gauss curvature of  \cite{BKG21, GNSW2023} when $N=2$ in \S\ref{subsec:spec_gauss_curv}.
Prior notions of distributional scalar curvature~\cite{GN2023} and
Einstein tensor~\cite{GN2023b} are shown to be recoverable from our
generalized Riemann curvature in \S\ref{ssec:spec-scal-curv} and
\S\ref{ssec:spec-einst-tens}, respectively.
A generalized  Ricci tensor for Regge metrics in arbitrary dimension
is derived in \S\ref{ssec:spec-ricci-curv}, which to our knowledge
has not appeared in prior literature. 

The main convergence results for the generalized Riemann curvature are in Section~\ref{sec:num_ana}, 
specifically, Theorems~\ref{thm:conv_Riemann} and 
Corollary~\ref{cor:conv_Riemann_int}.
We show that if a family of shape-regular triangulations and a sequence $\gappr$ of Regge metrics on them, interpolating a smooth metric $g$ using  piecewise polynomials of order $k\geq 0$ in $N=2$ ($k\geq 1$ in $N\geq 3$) are  given, then the generalized Riemann curvature produced by  $\gappr$ converges to the exact smooth Riemann curvature of $g$ at the
rate $\mathcal{O}(h^{k+1})$ in the $H^{-2}$-norm.
To prove this, we combine ideas from \cite{GNSW2023} and \cite{GN2023}, using the new ingredients introduced in the previous sections, described next.

A central new ingredient, found in Section~\ref{sec:lineariz}, is the
linearization of the distributional Riemann curvature tensor as the
underlying metric changes. The generalized curvature is given in
Section~\ref{sec:distr_Riemann} using geometrically natural test
functions having the same symmetries as the Riemann curvature
tensor. However, these test functions are metric-dependent, creating
difficulties in the analysis by generating opaque metric dependencies. To
circumvent these difficulties, we develop a version of the Uhlenbeck
trick \cite{Ham1986, Top2006}, mapping test functions to certain metric-independent ones, in \S\ref{ssec:metr-indep-test}.
This then allows us to compute the linearization of all curvature contributions, thus proving the main result of the section,  Theorem~\ref{thm:distr_Riem_evol}. {This technique, using a fixed triangulation $\T$, can directly be extended to a sequence of shape-regular triangulations $\T_h$ with meshsize $h\to 0$ by employing uniform estimates.}

The result of Theorem~\ref{thm:distr_Riem_evol}  sets the stage for Section~\ref{sec:distr_inc}, where the
incompatibility operator is defined in $N$~dimensions. Specifically, 
Theorem~\ref{thm:distr_Riem_evol} shows that
the linearization consists of two parts. In Section~\ref{sec:distr_inc},
we show that one part of the
linearization can be interpreted as the distributional covariant
incompatibility operator, as in \cite{GNSW2023}, but now extended to
any dimension, as shown in Theorem~\ref{thm:distr_inc_curved}. The
remaining part consists of terms without covariant derivatives
of the metric perturbation. The latter is zero in dimension $N=2$, so
it never appeared in~\cite{GNSW2023}. The final important ingredient is
the adjoint of the distributional covariant incompatibility operator,
the subject of Theorem~\ref{thm:distr_covariant_adjoint_inc}. Each of
these ingredients, interesting by itself, is used in a subsequent
section (Section~\ref{sec:num_ana}) on numerical analysis, which
proceeds by estimating each of the above-mentioned two parts of the
linearization. Among the technical tools used in these estimates are the codimension-1 and codimension-2 trace inequalities~\eqref{eq:trace_inequ_standard}--\eqref{eq:trace_inequ_codim2}, which bound norms on bones and facets in terms of element-interior norms.

In Section~\ref{sec:spec}, we simplify and specialize the
distributional Riemann curvature tensor in $N=2, 3$ dimensions. When
$N=2$, we obtain the distributional Gauss curvature as expected, and
we highlight the peculiarity that one of the two parts in the
linearization vanishes. When $N=3$, we focus on a curvature operator
$\Curvature$ that encapsulates the (skew-)symmetry properties of the
Riemann curvature. All terms are then written out in a computable form. 

The computable form is leveraged in Section~\ref{sec:num_examples}
where a numerical example is displayed in the $N=3$ case. The results
provide a practical illustration of asymptotic theoretical convergence
rates. They also indicate that the rates proved are likely not improvable
in general. In the lowest-order case of a piecewise constant metric
tensor, we observe a large pre-asymptotic regime of linear
convergence, which then eventually appears to degenerate to no
convergence, in accordance with the theory.

\section{Distributional densitized Riemann curvature tensor}
\label{sec:distr_Riemann}

Consider an open and bounded domain $\om \subset \R^N$, $N \ge 2$,
that is triangulated by finitely many bounded affine simplices.
On $\om$, a smooth metric
$\gex$ provides a Riemannian manifold structure.
We are not given $\gex$,
only an approximation of it, denoted by $g$. This approximation $g$
is a piecewise smooth metric with respect to a triangulation $\T$
(see, e.g., \cite{Cia1978, BS2008, ernFiniteElementsApproximation2021} for the usual definition of  a  triangulation). For topological background on simplicial complexes, triangulations, and piecewise-linear topology, see also~\cite{Munkres84,RourkeSanderson72}. Here 
the simplices of the triangulation are collected into
$\T$.
The goal of this section is to generalize the notion of Riemann
curvature for metrics $g$ that are only piecewise smooth, specifically for
$g$ in a Regge space defined in~\eqref{eq:ttspace} below.

A note on notation is in order. The symbols $\Xm T$, $\W^k (T)$, and
$\TT_l^k (T)$ denote, respectively, the spaces of smooth vector fields,
{alternating} $k$-form fields, and $(k, l)$-tensor fields on a submanifold $T$ of
$\om$.  Here, smoothness signifies not only infinite differentiability
at interior points, but also smoothness 
up to (including) the boundary of $T$.
In such symbols, replacement of the manifold $T$ by a collection of
subdomains, such as the triangulation $\T$, signifies the piecewise
smooth analogue with respect to the collection. Namely, $\TT_l^k(\T)$
is the Cartesian product of $ \TT_l^k(T) $ over some enumeration of
all $T \in \T$. Also, $\W^1 (\T) = \TT_0^1(\T)$ and
$\Xm \T = \TT_1^0(\T)$. Our metric $g$ is in $\TT_0^2(\T)$,
but is generally not in $\TT_0^2(\om)$.
Let $\triangle_{-m}T$ denote the set of
$(N-m)$-dimensional subsimplices of an $N$-simplex $T$.  The
collection of all subsimplices in $\triangle_{-1}T$ for all $T \in \T$
is the set $\F$ of facets of $\T$.
For an element $T\in\T$ and any $F \in \triangle_{-1}T$,
we shall refer to a vector field $\gn$ satisfying
\begin{equation}
  \label{eq:g-normal-vec}
 g|_T(\gn, X) = 0 \text{ for all $X \in \Xm F$} \quad \text{and}\quad
 g|_T(\gn, \gn)=1,
\end{equation}
as a {\em $g$-normal vector of $F$ in $T$.}  There are two possible
orientations of $\gn$ (and we will select one when needed). For
a vector field $X$ on $T$, let
\begin{equation*}
  \Proj X = X - g|_T(X, \gn) \gn
\end{equation*}
denote the pointwise {orthogonal} projection onto the tangent bundle 
of $F$. When {projecting} $k$-linear forms $B(X_1, \dots, X_k)$ to subsimplices,
we often need to consider cases where some arguments are fixed to some
specific vector {field}, while other arguments are {projected} to tangent
vector fields on the subsimplex. For example, on an
$F \in \triangle_{-1}T$, given smooth vector fields $X_i$ on $T$, the forms $B_{\gn \Fp \cdots \Fp}$ and $B_{\Fp \gn \Fp \cdots\Fp}$ are
defined by
  \begin{equation}
      \label{eq:nFnF-notation}
    \begin{aligned}  
 B_{\gn \Fp \cdots \Fp} (X_1, X_2,  \ldots, X_{k-1})
      &= B(\gn, \Proj X_1, \ldots,\Proj X_{k-1}),
      \\
 B_{\Fp \gn \Fp \cdots \Fp} (X_1, X_2, \ldots, X_{k-1})
      &= B(\Proj X_1, \gn, \Proj X_2, \ldots, \Proj X_{k-1}),  
    \end{aligned}
  \end{equation}
and this notation is extended to cases when more than one argument is
fixed with $\gn$ in the obvious fashion. {We note that using the interior product---see~\eqref{eq:contraction}---the first line in \eqref{eq:nFnF-notation} can be written as $B_{\gn \Fp \cdots \Fp}=(\gn\lrcorner B)_{\Fp\cdots\Fp}$.} When all arguments of $B$
are projected by $\Proj$, we abbreviate the resulting form  $B_{\Fp \dots \Fp}$ to
simply $B_{\Fp}$.  When an argument is left open,
and not projected to $\Xm F$, then we indicate so by a subscript
``$\cdot$'' in place of $\Fp$, e.g., for a form $B(X, Y, Z, W)$ taking
four arguments, the form $B_{\gn \Fp\Fp \cdot}$ takes three arguments and
is defined by
\begin{equation*}
 B_{\gn \Fp\Fp \cdot}(X_1, X_2, X_3) = B(\gn, \Proj X_1, \Proj X_2, X_3).
\end{equation*}
Obvious extensions of this notation to vectors other than $\gn$ and to
lower-dimensional subsimplices are used without much ado.

{Besides the projection $\Proj$, we also consider the restriction of $k$-linear forms $B$ to a subsimplex $D$ of $\T$. We define 
\begin{align}
  \label{eq:restriction}
 B_{\Dr}=B|_{\Dr},\qquad B_{\Dr}(X_1,\dots,X_k) = B(X_1,\dots,X_k) \text{ for } X_i\in\Xm D,
\end{align}
as the pullback of $B$ under the inclusion operator $D \hookrightarrow \Omega$. In contrast to the projection operator $\Proj$, the restriction is independent of the metric $g$. Note that on a facet $\Fr$, the restriction $B_{\Fr}$ does not coincide with the projection~$B_{\Fp}$ unless $g$ is continuous at $F$. However,  $B_{\Fp}|_{\Fr}=B_{\Fr}$. We adapt restrictions to individual arguments in the obvious way, e.g.,  $B_{\Fr\gn\Fr\Fr}(X_1, X_2, X_3) = B(X_1, \gn, X_2, X_3)$ for any $X_i \in \Xm F$.}

We also use
standard notions from smooth Riemannian geometry throughout, some of
which are collected in  Appendix~\ref{sec:geometrical-prelims} for ready reference.

\subsection{The $tt$-continuity}
\label{ssec:tt-continuity}

Let
$ \Sc(\T) = \{\sigma \in \TT_0^2(\T): \sigma(X, Y) = \sigma(Y, X)$ for
$X, Y \in \Xm \T \}$ be the space of symmetric covariant
2-tensors on $\om$ with no inter-element continuity in general. With $\Sc^+(\T)=\{\sigma \in \Sc(\T): \sigma(X, X)>0 \text{ for all
} 0\neq X \in \Xm \T\}$ we denote the subspace of positive-definite 2-tensors.
Recalling that $\F$ denotes the set of all mesh facets (of codimension~1),
we divide $\F$ into $\Fbnd$ consisting of facets contained in
$\d\om$ and the remainder, the set of interior facets, $\Fint$.  Every
$F \in \Fint$ is of the form $F = \triangle_{-1} T_+ \cap \triangle_{-1} T_-$ for two elements
$T_\pm \in \T$.  We say that a $\sigma \in \Sc(\T)$ has
``tangential-tangential continuity'' or {\em ``$tt$-continuity''} if
$ \sigma|_{T_+}(X, Y) = \sigma|_{T_-}(X, Y)$ for all tangential vector fields $X, Y \in \Xm F$
for every $F$ in $\Fint$ (i.e., with \eqref{eq:restriction}, $\sigma_{\Fr}$ is single-valued on all
$F \in \Fint$). This brings us to the {\em Regge space}
\begin{align}  \label{eq:ttspace}
  \Regge(\T)
  & =
 \{
    \sigma \in \Sc(\T): \sigma \text{ is $tt$-continuous}
 \},
\end{align}
and its  subset
$ \RR^+(\T) = \Regge(\T)\cap \Sc^+(\T)$. The approximate metric $g$ is assumed to be in
$\RR^+(\T)$. Also define
\begin{equation*}
  \oRegge(\T) = \{ \sigma \in \Regge(\T): \; \sigma(X, Y)=0 \text{ for all }
 X, Y \in \Xm F, F \in \F_\d\}.
\end{equation*}
{Note that $\sigma_{\Fr}$ is single-valued for $\sigma\in\Regge[]$; however, $\sigma_{\Fp}$ is in general not.}

\subsection{Curvature within elements}

Since a $g$ in $\RR^+(\T)$ is smooth within each element $T \in \T$, it
generates a unique Levi-Civita connection in each $T$, denoted by~$\nabla$. We refer to Appendix~\ref{sec:geometrical-prelims} for a summary of the geometric notation used. Using $\nabla$,  the Riemann curvature tensor within $T$ can be
computed using standard formulas:  following the sign convention
of~\cite{Lee18}, define 
\begin{equation}
  \label{eq:RiemannCurvTensor}
  \Riemann(X, Y, Z, W) =
 g (R_{X, Y} Z, W), \qquad X, Y, Z, W \in \Xm \T,
\end{equation}
where {for given $X$, $Y$, the map $Z\mapsto$}
$R_{X, Y} Z = \nabla_X \nabla_Y Z - \nabla_Y \nabla_X Z -
\nabla_{[X,Y]}Z$ is the Riemann curvature endomorphism. This element-by-element 
curvature tensor  $\Riemann \in \TT_0^4(\T)$ is only one of the contributors to the total distributional curvature defined below.

\subsection{Jump of the second fundamental form}
\label{ssec:jump-II}

The jumps of $g$ create further sources of curvature at lower-dimensional
facets, which must be added to the curvature within elements
to get good curvature approximations.
Recall the $g$-normal vector $\gn$ of \eqref{eq:g-normal-vec} {on $T\in\T$.}
The second fundamental form \cite{Lee18} of {$F\in\triangle_{-1}T$} considered as an
embedded submanifold depends on the orientation of $\gn$
and is defined by
\begin{equation}
  \label{eq:sff-def}
  \sff^{\gn}(X,Y)=
  -{g|_T}(\nabla_X\gn,Y) = {g|_T}(\gn, \nabla_X Y), \qquad X, Y \in \Xm F.
\end{equation}
The second equality follows from differentiating the identity
${g|_T}(\gn,Z)=0$ for $Z\in\Xm{F}$. {If necessary, we extend $\sff^{\gn}\in \TT^2_0(\d T)$ to $\sff^{\gn}_{\Fp}\in \TT^2_0(T)$ with \eqref{eq:nFnF-notation} and then drop the subscript $\Fp$.} Let the unique $g$-normal vector on $F$ that
points {\em inward} into an element $T \in \T$ be denoted by $\gn_F^T$.
Now consider an $F=\triangle_{-1} T_+\cap \triangle_{-1} T_-$ for some $T_\pm \in \T$.  As
$g\in\RR^+(\T)$ is solely $tt$-continuous, $\gn_F^{T_+}\neq -\gn_F^{T_-}$
in general. The {\em jump of the second fundamental form} across $F$
is defined by
\begin{equation}
  \label{eq:jump-sff}
  \jmp{\sff}(X, Y) = \sff^{\gn_F^{T_+}}(X, Y)
  + \sff^{\gn_F^{T_-}}(X, Y),    
\end{equation}
for all $ X, Y \in \Xm F$. This jump 
$\jmp{\sff}$ is a tensor field in  $\TT_0^2(\Fint)$ (which, per previous notation, is the Cartesian product of $\TT_0^2(F)$ over an enumeration of all $F$ in $\Fint$). It 
will act as a source of curvature on
facets.

The jump of a 
tensor $B^\gn \in \TT_0^k(\T)$ {with a linear (or odd) $\gn$-dependency} across $F$
is defined by
\begin{equation}
  \label{eq:jmp-tensor}
  \jump{B}(X_1, X_2, \ldots) =
 B^{\gn_F^{T_+}}(X_1, X_2, \ldots)
  + B^{\gn_F^{T_-}}(X_1, X_2, \ldots),       
  \qquad 
 X_i \in \Xm F,  
\end{equation}
in analogy with \eqref{eq:jump-sff}.

\subsection{Angle deficit}
\label{ssec:angle-deficit}

There are also sources of curvature along subsimplices of
codimension~2.  Let $\E$ denote the collection of all simplices in
$\triangle_{-2}T$ for all $T \in \T$. Divide it into $\Ebnd$
consisting of simplices in $\E$ lying on the boundary $\d\om$ and the
remainder $\Eint = \E \setminus \Ebnd$.  Given any $E \in \Eint$ there
is an $F \in \Fint$ such that $E \in \triangle_{-1} F$.  We refer to
a vector field
$\gcn \in \Xm F$
satisfying 
\begin{equation}
  \label{eq:g-conormal-vec}  
 g|_T(\gcn, \gn_F^T)=0,\quad F\in\triangle_{-1}T,
  \qquad
 g(\gcn, X)=0 \text{ for all $X \in \Xm E$,}
  \qquad
 g(\gcn, \gcn)=1,
\end{equation}
as a {\em $g$-conormal vector of $E$ in $F$}. {Note that the latter two expressions in \eqref{eq:g-conormal-vec} are single-valued due to the $tt$-continuity of $g$.}
There are two possible orientations for such a $\gcn$. When both $E$
and $F$ lie on the boundary of an element $T \in \T$, we
select the unique $g$-conormal vector of $E$ that
points into $F$ from $E$ and denote it by $\gcn_E^F$ (see  Figure~\ref{fig:vectors}).  Using the two
facets $F_\pm$ in $\triangle_{-1}T$ such that
$E = \triangle_{-1} F_+ \cap \triangle_{-1} F_-$, we define the following angle function {acting on $T_pT$ for $p\in E$}
\begin{align}
  \label{eq:interior_angle}
  \sphericalangle_E^T 
 =  \arccos\big( g\big|_T(\gcn^{F_+}_E,\gcn^{F_-}_E) \big). 
\end{align}
Let $\T_E = \{ T \in \T: E \in \triangle_{-2}T\}$. 
The  {\em angle deficit} 
at $E\in\Eint$ is defined by
\begin{align}
  \label{eq:angle-deficit}
\Theta_E = 2\pi -\sum_{T \in \T_E}\sphericalangle_E^T. 
\end{align}
This function, $\Theta=\Pi_{E\in\Eint}\Theta_E \in \TT_0^0(\Eint)$, will  act as a source
of curvature on $\Eint$.

\begin{figure}[ht!]
	\centering
	\begin{tikzpicture}
		\draw[thin] (0,0) to (2,0);
		\draw[thin] (0,0) to (1.4,1.5);
		\draw[thin] (0,0) to (1.6,-1.3);
		\draw (0,0) circle (3pt);
		
		\draw[-latex,color=red,thick] (1.2,0) to (1.2,-0.8);
		\draw[-latex,color=red,thick] (1.7,0) to (1.7,0.8);
		\draw[-latex,color=blue,thick] (0,0) to (0.8,0);
		
		\node (A) at (-0.3,-0.1) {$E$};
		\node (B) at (2,1.1) {$T_+$};
		\node (C) at (1.95,-1.) {$T_-$};
		\node (D) at (2.3,-0.1) {$F$};
		\node (E) at (1.4,0.4) {{\color{red}$\gn^{T_+}_F$}};
		\node (F) at (1.65,-0.5) {{\color{red}$\gn^{T_-}_F$}};
		\node (G) at (0.7,0.3) {{\color{blue}$\gcn^F_E$}};
	\end{tikzpicture}\hspace*{3cm}
	\begin{tikzpicture}
		\draw[thin] (0,0) to (1.4,1.5);
		\draw[thin] (0,0) to (1.6,-1.3);
		\draw (0,0) circle (3pt);
		
		\draw[-latex,color=red,thick] (1.4/2,1.5/2) to (1.4/2+1.5/3,1.5/2-1.4/3);
		\draw[-latex,color=red,thick] (1.6/2,-1.3/2) to (1.6/2+1.3/3,-1.3/2+1.6/3);
		\draw[-latex,color=blue,thick] (0,0) to (1.4/3,1.5/3);
		\draw[-latex,color=blue,thick] (0,0) to (1.6/3,-1.3/3);
		
		\node (A) at (-0.5,-0.1) {$E$};
		\node (B) at (2,0) {$T$};
		\node (C) at (0.7,1.2) {$F_+$};
		\node (D) at (0.7,-1.2) {$F_-$};
		\node (E) at (1.3,0.7) {{\color{red}$\gn_{F_+}^{T}$}};
		\node (F) at (1.5,-0.6) {{\color{red}$\gn_{F_-}^{T}$}};
		\node (G) at (0.1,0.6) {{\color{blue}$\gcn^{F_+}_E$}};
		\node (G) at (0.,-0.6) {{\color{blue}$\gcn^{F_-}_E$}};
	\end{tikzpicture}
	\caption{Visualization of $g$-normal and $g$-conormal vectors on {a} facet between two elements (left) and a single element (right).}
	\label{fig:vectors}
\end{figure}

\subsection{Riemann curvature for Regge metrics}

{Suppose that} $A \in \TT_0^4(\T)$ {has} the (skew) symmetries of the Riemann
curvature tensor, i.e.,
\begin{equation}
  \label{eq:symm-Riem}
 A(X, Y, Z, W) = -A(Y, X, Z, W) = -A(X, Y, W, Z) = A(Z, W, X, Y)
\end{equation}
for all $X, Y, Z, W \in \Xm \T$. We will generalize the Riemann
curvature tensor as a linear functional acting on 
such~$A$. We shall also require $A$ to have certain inter-element
continuity constraints described next. Recall that per~\eqref{eq:restriction}, 
${\Atnnt}$ is defined by
\begin{equation*}
 {\Atnnt}(X, Y) = A(X, \gn, \gn, Y), \qquad X, Y \in \Xm F,  
\end{equation*}
for any {$F\in\triangle_{-1}T$, $T\in\T$,} and a $g$-normal vector $\gn$ as in~\eqref{eq:g-normal-vec}. Note that ${\Atnnt}$ is
independent of the orientation of $\gn$. Further, due to \eqref{eq:symm-Riem}, {there holds} $A_{\Fp\gn\gn\Fp}=A_{\cdot\gn\gn\cdot}$. 

In general, since the limiting values of $A$ from adjacent elements are different,
${\Atnnt}$ is discontinuous
with multi-valued limits on element interfaces. We consider the
following continuity requirement:
\begin{equation}
  \label{eq:cont_A}
 {\Atnnt} \text{ is single-valued for all }F\in\Fint.
\end{equation}
Define the {\em test space} $\Ao$ by 
\begin{align}
  \label{eq:testspace_A}
  \begin{split}
    \Atest:=
    &\{ A\in\TT^4_0(\T)\,: \;
 A \text{ satisfies~\eqref{eq:symm-Riem} and \eqref{eq:cont_A}}\},
    \\
    \Atesto:=
    &\{ A\in \Atest\,: \, {\Atnnt} \text{ vanishes on all } F\in\Fbnd\}.
  \end{split}
\end{align}
The continuity condition \eqref{eq:cont_A} implies continuity at mesh interfaces of codimension two, as we now show.

\begin{lemma}
  \label{lem:cont_A_codim2}
 Let $A\in\Atest$ and $E \in \Eint$.  Let $F \in \Fint$ and $T
        \in \T$ be such that $E \in \triangle_{-1}F$ and
        $F \in \triangle_{-1} T$.  If $\gn$ is a $g$-normal vector of
        $F$ in $T$ (see \eqref{eq:g-normal-vec}) and $\gcn$ is a
        $g$-conormal vector of $E$ in $F$ (see
        \eqref{eq:g-conormal-vec}), then $A_{\gcn\gn\gn\gcn} = A(\gcn, \gn, \gn, \gcn)$ is
 single-valued for all $E\in\Eint$.
\end{lemma}
\begin{proof}
 Let $T_\pm \in \T$ share a facet $F = \triangle_{-1} T_+ \cap \triangle_{-1} T_-$ with
  $E \in \triangle_{-1} F$.  Let $\gt\in \Xm F$ denote some extension of
  $\gcn$ from $E$ to $F$, i.e., $\gt|_E = \gcn$. Because
 of~\eqref{eq:cont_A}, the values of $A(\gt,\gn,\gn,\gt)$ from $T_+$
 and $T_-$ are equal at any point on $F$. By the assumption that
 elements in $\TT_0^4(\T)$ are continuous up to the element
 boundaries, the same holds for points on $E$. Thus, the values of
  $A_{\gcn \gn\gn \gcn}$ from $T_+$ and $T_-$ coincide on $E$.

 Next, let $F_\pm \in \triangle_{-1} T$ be such that
  $\triangle_{-1} F_+ \cap \triangle_{-1} F_- = E$.  The proof will be finished
 if we show that at any point in $E$,
 the values of
  $A_{\gcn\gn\gn\gcn}$ from $F_+$ and $F_-$ coincide, i.e., if 
  $A(\gcn_E^{F_+},\gn_{F_+}^T,\gn_{F_+}^T,\gcn_E^{F_+})$
 and $A(\gcn_E^{F_-},\gn_{F_-}^T,\gn_{F_-}^T,\gcn_E^{F_-})$ coincide.
 (Figure~\ref{fig:vectors} illustrates the vectors involved.)
 Since $\{\gn_{F_+}^T,\gcn_E^{F_+}\}$ and
  $\{\gn_{F_-}^T,\gcn_E^{F_-}\}$ span the same 2-dimensional plane,
 there is an angle $\phi$ (measured in $g|_T$ {and thus, acting on $T_pT$ for $p\in E$}) by which
  $\gcn_{E}^{F_+}$ can be rotated in-plane to $\gcn_{E}^{F_-}$. Then,
 noting that $\gn_{F_+}^T$ and $\gn_{F_-}^T$ point into $T$,
  \begin{flalign}
    \label{eq:angle-rotation}
    \gcn_{E}^{F_-} = \cos(\phi)\gcn_{E}^{F_+} + \sin(\phi)\gn_{F_+}^T,
    \qquad
    \gn_{F_-}^T = \sin(\phi)\gcn_{E}^{F_+} - \cos(\phi)\gn_{F_+}^T.
\end{flalign}
Using the skew symmetries \eqref{eq:symm-Riem} of $A$,
\begin{flalign*}
 A(\gcn_E^{F_-},\gn_{F_-}^T,\gn_{F_-}^T,\gcn_E^{F_-}) &= 
 (\cos^2(\phi)+\sin^2(\phi))A(\gcn_E^{F_+},\gn_{F_+}^T,\gn_{F_-}^T,\gcn_E^{F_-}) \\
  &=
 A(\gcn_E^{F_+},\gn_{F_+}^T,\gn_{F_+}^T,\gcn_E^{F_+}),
\end{flalign*}
i.e., at $E$, the values of $A_{\gcn\gn\gn\gcn}$ from $F_+$ and $F_-$ coincide.
\end{proof}

\begin{definition}
  \label{def:gen-riemann-curv-regge}
 We define the {\em generalized densitized Riemann curvature} for the  non-smooth metric 
$g \in \RR^+(\T)$ to be the linear functional $\Rog : \Ao \to \R$
given by
\begin{equation}
  \label{eq:distr_Riemann}
  \begin{aligned}    
    \Rog(A)=\sum_{T\in\T}\int_T\langle \Riemann,A\rangle\,\vo{T}
    & +
    4\sum_{F\in\Fint}\int_F\langle\jmp{\sff},{\Atnnt}\rangle\,\vo{F}
    \\
    & +
    4\!\!\sum_{E\in\Eint}\int_E \Theta_E\, A_{\gcn\gn\gn\gcn}\,\vo{E}
  \end{aligned}
\end{equation}
for all $A\in \Atesto$, where $\vo T = \og|_T, \vo F = \og|_F$, and
$\vo E = \og|_E$ are the volume forms on $T$, $F$, and $E$,
respectively, given by the Riemannian volume form
$\og_g \equiv \og \in {\W^N(\T)}$ generated by $g$. (We drop the subscript in $\og_g$ when there can be no confusion on what metric is being used.)  Here and throughout,
$\ip{\cdot, \cdot}$ denotes the standard extension of the $g$-inner
product to tensors (see Appendix~\ref{sec:geometrical-prelims},
\eqref{eq:innerprod-extended}), $\gn$ is a $g$-normal vector of $F$, 
and $\gcn$ is a $g$-conormal vector of $E$. By Lemma~\ref{lem:cont_A_codim2}, the last term in~\eqref{eq:distr_Riemann} makes sense.
\end{definition}

{ Note that for tensors $A,B\in \TT^2_0(\F)$ the inner product $\ip{A,B}$ is understood with respect to $g_{\Fr}$ instead of $g$. The following relation holds between the restriction and the projection of covariant tensors, by virtue of which 
we conclude that 
  $\ip{A_{\Fp},B_{\Fp}}$ is single-valued.

\begin{lemma}
  \label{lem:restr_proj_relation}
  Let $A,B\in \TT^k_0(\T)$ be two $(k,0)$-tensors with single-valued  restrictions $A_{\Fr}$ and $B_{\Fr}$ on facets $F\in\F$. Then  on any
  $F\in\triangle_{-1}T$, $T\in\T$,
  \begin{align}
  \label{eq:restr_proj_relation}
  \ip{A_{\Fr},B_{\Fr}}=\ip{A_{\Fp},B_{\Fp}}.
\end{align}
\end{lemma}
\begin{proof}
 It suffices to prove the claim for $k=1$. Let
  $\{E_1,\dots, E_{N-1}, E_N\}$
 be any $g$-orthonormal frame
 with $E_1,\dots, E_{N-1}\in \Xm{F}$
 and $E_N=\gn_F^T$, abbreviated to $\gn$ here. Expanding $A_{\Fr}$ and $B_{\Fr}$ in this frame gives
  $$
  \ip{A_{\Fr},B_{\Fr}}=\sum_{i=1}^{N-1}A(E_i)B(E_i)
  $$
 and the right-hand side yields
  \begin{align*}
    \ip{A_{\Fp},B_{\Fp}} &= \ip{A-A(\gn)\gn^\flat, B-B(\gn)\gn^\flat}=\sum_{i=1}^{N-1}A(E_i)B(E_i)
  \end{align*}
  thus proving the result.
\end{proof}
}

The test space $\Ao$ contains infinitely smooth compactly supported
tensor functions on $\om$. Hence, $ \Rog(A)$ can be regarded as an
extension of a Schwartz distribution or a measure on~$\om$. Due to the
presence of the volume form, it is a {\em distribution
density}~\cite{Christiansen13} as in the title of this
section. Equipping $\Ao$ with a topology in which the right-hand side
of \eqref{eq:distr_Riemann} is continuous is an interesting issue (see
e.g., a similar issue in~\cite[Appendix~A]{GNSW2023}).
But this is not discussed further in this paper because it is not central to our main effort of
proving the {convergence} of~\eqref{eq:distr_Riemann}
through numerical analysis, where we
will only  examine convergence of $\Rog$ in the {$H^{-2}(\om)$-norm defined in \eqref{eq:hmknorm}}.

Note that there are many $g$-dependent quantities
in~\eqref{eq:distr_Riemann}. When we need to explicitly show that
dependence, we will write $g$ as an argument, i.e., instead of
$\Rog(A), \vo T,$ $\vo F,$ $\vo E,$ $\gn$, and $\gcn$, we write
$\Rog (g) (A),$ $\vol T$, $\vol F$, $\vol E$, $\gn(g)$, and $\gcn(g)$,
respectively. Even the test space $\Ao$ is $g$-dependent, an issue we
discuss in more detail in \S\ref{sec:lineariz}.

\subsection{Specialization to two-dimensional Gauss curvature case}
\label{subsec:spec_gauss_curv}
In the case of 2D manifolds ($N=2$), elements of the test
space can be generated by scalar fields in
\begin{equation*}
  \begin{gathered}  
    \VV(\T) = \{ u \in \W^0 (\T) :
    \; \text{$u$ is continuous on } \om\},
    \\ \nonumber 
    \Vo(\T)
 = \{ u \in \VV(\T): u|_{\d\om} = 0\},
  \end{gathered}  
\end{equation*}
by combining them with the Riemannian volume form~$\og$. Let
$v \in \Vo(\T)$. The tensor field $A = -v\,\og \otimes \og$,
\begin{equation}
  \label{eq:Achoice2D}
 A(X, Y, Z, W) = -v\;\og(X, Y) \;\og(Z, W),   
\end{equation}
obviously satisfies the symmetries in~\eqref{eq:symm-Riem}.  Moreover,
since $\og$ applied to any $g$-orthonormal frame yields $\pm 1$,
we have $\og(\gcn,\gn)\og(\gn, \gcn) = -1,$ so the continuity of $v$ implies~\eqref{eq:cont_A}.
Hence, $A \in \Ao$. We use this choice of $A$
in~\eqref{eq:distr_Riemann}. Then, {in a given system of coordinates $x_1$, $x_2$,} a computation
using~\eqref{eq:innerprod-extended} and
$\og = \sqrt{\det g} \; dx^1 \wedge dx^2$ shows that the first term
on the right-hand side of~\eqref{eq:distr_Riemann} takes the form
\begin{align}
  \label{eq:K-1}
  \ip{\Riemann, A}
  & =  \frac{ 4 R_{1221}\, v }{\det g} = 4 K \, v,
\end{align}
where $K$ is the Gauss curvature.

We proceed to the next term on the right-hand side of~\eqref{eq:distr_Riemann}.
It is easy to see that for the $A$ in~\eqref{eq:Achoice2D}, 
\[
 {\Atnnt} = v\, {(\gt^\flat \otimes \gt^\flat)_{\Fr}},
\]
for any
$g$-normalized tangent $\gt$ to the one-dimensional element boundary $\d T$. Note that $\gt$ and $\gcn$ are 
collinear at each point. Therefore, by the second equality
of~\eqref{eq:sff-def}, the geodesic curvature
$\GeodCurv^{\gn} = g( \nabla_\gt \gt, \gn) $ of $\d T$ equals
$\sff^{\gn}(\gt, \gt)$. Thus, the jump
of the geodesic curvature across $F = \triangle_{-1} T_+ \cap \triangle_{-1} T_-$, namely 
$\jump{\GeodCurv} = \GeodCurv^{\gn_F^{T_+}} + \GeodCurv^{\gn_F^{T_-}},$
satisfies
\begin{align*}
  \ip{\jump{\sff}, {\Atnnt}} = \jump{\GeodCurv} v .
\end{align*}

Finally, to relate to the last term on the right-hand side
of~\eqref{eq:distr_Riemann} on the angle deficit, observe that
\begin{equation}
  \label{eq:K-3}
  \Theta_E A_{\gcn\gn\gn\gcn} = \Theta_E \, v
\end{equation}
since $A_{\gcn\gn\gn\gcn} = -\og(\gcn,\gn)\og(\gn,\gcn) v = v$. Combining~\eqref{eq:K-1}--\eqref{eq:K-3} we find that
\begin{equation}
    \label{eq:distr_Gauss}
  \frac 1 4 \Rog( -v \,\og \otimes \og) =
  \sum_{T\in\T}\int_T K v \,\vo{T}
  +
  \sum_{F\in\Fint}\int_F \jmp{\GeodCurv} v\,\vo{F}
  +
  \sum_{E\in\Eint} \Theta_E\, v(E).
\end{equation}
The right-hand side above is exactly the two-dimensional
\emph{distributional densitized Gauss curvature} $\widetilde{K \og}(v)$
previously analyzed in~\cite{BKG21,
  GNSW2023}.

\subsection{Specialization to scalar curvature in any dimension}
\label{ssec:spec-scal-curv}

The scalar curvature in the smooth case is given by $S = \Riemann_{ijkl}g^{il} g^{jk}$.
Its generalized densitized version $\widetilde{S \og}$ can be obtained from our generalized Riemann curvature, as shown next.
The Kulkarni--Nomizu product $\owedge:{\Sc(\Omega)\times \Sc(\Omega)}\to \TM{0}{4}$ is defined as
\begin{align*}
  \begin{split}
 (h\owedge k)(X,Y,Z,W) &:= h(X,W)k(Y,Z)+h(Y,Z)k(X,W)\\
    &\quad- h(X,Z)k(Y,W)-h(Y,W)k(X,Z)
  \end{split}
\end{align*}
taking two {symmetric} $(2,0)$-tensors resulting into a $(4,0)$-tensor with the algebraic properties {\eqref{eq:symm-Riem}, which are also shared by the Riemann curvature tensor,}
i.e., $(h\owedge k)(X,Y,Z,W)=-(h\owedge k)(Y,X,Z,W)=(h\owedge k)(Y,X,W,Z)$ and $(h\owedge k)(X,Y,Z,W)=(h\owedge k)(Z,W,X,Y)$. We can choose ${A\in\Atesto}$ to be of the form
\begin{align}
  \label{eq:testfunc_scalar_curv}
 A= v\,g\owedge g,\qquad v\in \Vo,
\end{align}  
fulfilling the continuity conditions \eqref{eq:cont_A}, as on all $F\in\Fint$ and $E\in\Eint$
\begin{align*}
  &(g\owedge g)(X,\gn,\gn,Y) = 2(g(X,Y)-g(\gn,X)g(\gn,Y))=2\,g(X,Y) \forall X,Y\in\Xm{F},\\
  & (g\owedge g)(\gcn,\gn,\gn,\gcn)=2.
\end{align*}
Inserting \eqref{eq:testfunc_scalar_curv} into \eqref{eq:distr_Riemann} yields due to the (skew-)symmetry properties of $\Riemann$
\begin{align*}
  \langle \Riemann,A\rangle = 4\Riemann_{ijkl}g^{il}g^{jk}v=4 Sv,\quad \langle\jmp{\sff},{\Atnnt}\rangle=2 \jmp{H}v,\quad \Theta_E\, A_{\gcn\gn\gn\gcn}= 2\Theta_Ev,
\end{align*}
where $S$ and $H^{\gn}=\tr{\sff^{\gn}}$ denote the scalar and mean curvature, respectively. Thus, we obtain, {for every} $v \in \Vo$,
\begin{align*}
  \frac 1 4 \Rog(v\,g\owedge g) = \sum_{T \in \T} \int_T \Scalar v\, \vo{T} + 2\sum_{F \in \Fint} \int_F \llbracket H \rrbracket\, v\, \vo{F} + 2\sum_{E \in \Eint} \int_E \Theta_E v\, \vo{E},
\end{align*}
which coincides with the definition of the \emph{distributional densitized scalar curvature}, $\widetilde{S \og}(v)$, in any dimension proposed in \cite{GN2023}. The factor $2$ in the codimension-1 and codimension-2 boundary terms is consistent with \eqref{eq:distr_Gauss} as in two dimensions the scalar curvature is twice the Gauss curvature, $\Scalar=2\Gauss$.

\subsection{Specialization to Ricci curvature in any dimension}
\label{ssec:spec-ricci-curv}
For smooth metrics, the Ricci curvature is a contraction of the Riemann curvature tensor given by $\Ricci_{ij}=\Riemann_{kijl}g^{kl}$. {In two dimensions there holds $\Ricci=\Gauss\,g$, so we focus on the interesting case of dimension $N\ge 3$.} To develop its generalization to Regge metrics, 
consider test functions of the form
\begin{equation}
  \label{eq:A-Ricci}
 A=g\owedge \rho,
\end{equation}
where $\rho$ is any tensor in the {space $\mathcal{W}(\T)=\{ J\sigma\,:\,\sigma\in \oRegge(\T)\}$ and $J:\Sc(\T)\to \Sc(\T)$, $J\sigma=\sigma-1/2\tr[]{\sigma}g$, a bijective algebraic operator.} Hence, the $A$ in~\eqref{eq:A-Ricci} satisfies, for all $X,Y\in\Xm{F}$ and $F \in \Fint$,
\begin{align*}
 A(X,\gn,\gn,Y)
  &= \rho(X,Y)+g(Y,X)\rho(\gn,\gn)-g(\gn,X)\rho(\gn,Y)-g(\gn,Y)\rho(\gn,X) \\
  &= {(J\sigma)(X,Y) + g(X,Y)(J\sigma)(\gn,\gn)} \\
  &= {\sigma(X,Y) + g(X,Y)\sigma_{\gn\gn}-g(X,Y)\tr{\sigma} = \sigma(X,Y) -\mathrm{tr}(\sigma_{\Fr})g(X,Y),}
  \\
 A(\gcn,\gn,\gn,\gcn)
  &= {(J\sigma)(\gcn,\gcn)+(J\sigma)(\gn,\gn)=\sigma(\gcn,\gcn)+\sigma(\gn,\gn)-\tr{\sigma} = -\tr{\sigma_E}},
\end{align*}
where {$\sigma_E$ denotes the restriction to $E \in \Eint$ \eqref{eq:restriction}}. {Note that $\tr[]{\sigma}$, $\tr[]{\sigma_{\Fr}}$, and $\tr[]{\sigma_E}$ are computed with respect to $g|_T$, $g_{\Fr}$, and $g_E$, respectively, cf. \eqref{eq:trace-2-tensor}. Thus, there holds, e.g., $\tr[]{\sigma}=\tr[]{\sigma_{\Fr}}+\sigma_{\gn\gn}$.} We conclude that the $A$ in~\eqref{eq:A-Ricci} is  in $\Atesto$.
Inserting such $A$ into \eqref{eq:distr_Riemann} leads to a novel definition of the \emph{generalized densitized Ricci curvature tensor}
\begin{align*}
  \begin{split}
    \frac 1 4 \Rog(g \owedge \rho)=
    &\sum_{T\in\T}\int_T\langle \Ricci,\rho\rangle\,\vo{T}
      +\sum_{F\in\Fint}\int_F\langle\jmp{\sff},{\rho_{\Fr}+\rho_{\gn\gn}g_{\Fr}}\rangle\,\vo{F}
    \\
    + &
        \sum_{E\in\Eint}\int_E \Theta_E\, (\rho_{\gn\gn}+\rho_{\gcn\gcn})\,\vo{E},
  \end{split}
\end{align*}
a linear functional acting on any $\rho$ in $\mathcal{W}(\T)$. {In \cite{lottRicciMeasureSingular2016}, the Ricci curvature for singular metrics was investigated in the context of measures. The codimension 1 boundary term \cite[Section 4.2]{lottRicciMeasureSingular2016} coincides with ours.}

\subsection{Specialization to the Einstein tensor}
\label{ssec:spec-einst-tens}

Let $N\geq 3$.
The Einstein tensor for smooth metrics is
given using the above-mentioned Ricci curvature, {scalar curvature, and algebraic operator $J$} by 
$\Einstein=\Ricci-\frac{1}{2}\Scalar\,g{=J\,\Ricci}$.
Its generalization to Regge metrics {therefore directly follows from \eqref{eq:A-Ricci} using test functions of the form}
\begin{align*}
 A=g\owedge J\rho,\qquad
  \rho\in \oRegge(\T). 
\end{align*}
{Analogously to \S\ref{ssec:spec-ricci-curv}, we have for all $F\in\Fint$, $X,Y\in\Xm{F}$, and $E\in\Eint$}
\begin{align*}
  \Atnnt(X,Y)
  & {= \rho(X,Y) -\mathrm{tr}(\rho_{\Fr})g(X,Y)=\mathbb{S}_{\Fr}(\rho)(X,Y),}\qquad {A_{\gcn\gn\gn\gcn} = -\tr{\rho_E},}
\end{align*}
where $\mathbb{S}_{\Fr}(\rho)=\rho_{\Fr}-\tr{\rho_{\Fr}}g_{\Fr}$ is the trace-reversed part of $\rho$ restricted to a facet $F\in\Fint$. Inserting into \eqref{eq:distr_Riemann} yields
\begin{align*}
  \begin{split}
    \frac 1 4 \Rog( g \owedge J \rho)=
    &\sum_{T\in\T}\int_T\langle \Einstein,\rho\rangle\,\vo{T}
      +\sum_{F\in\Fint}\int_F\langle\jmp{\sff},\mathbb{S}_{\Fr}(\rho)\rangle\,\vo{F}-
        \sum_{E\in\Eint}\int_E \Theta_E\, {\tr{\rho_{E}}}\,\vo{E}.
  \end{split}
\end{align*}
The right-hand side is exactly
the \emph{distributional densitized Einstein tensor} 
$\widetilde{\Einstein\og}(\rho)$ 
defined and analyzed in~\cite{GN2023,GN2023b}. Note that $\langle\jmp{\sff},\mathbb{S}_{\Fr}(\rho)\rangle=\langle\jmp{\mathbb{S}_{\Fr}\sff},\rho_{\Fr}\rangle=\langle\jmp{\overline{\sff}},\rho_{\Fr}\rangle$ with the trace-reversed second fundamental form $\overline{\sff}^{\gn}=\sff^{\gn}-H^{\gn}\,g_{\Fr}$. In general relativity, the jump of the trace-reversed second fundamental form features in the well-known Israel junction condition~\cite{Isr1966}.

\section{Linearization of curvature}
\label{sec:lineariz}
The test spaces $\Atest$ and $\Atesto$ defined
in~\eqref{eq:testspace_A}, which provide test functions in the
generalized curvature  formula, depend on the metric tensor $g$,
since the continuity properties are given in terms of the $g$-normal
vector $\gn$.  While analyzing the convergence of the
approximate Riemann curvature tensor as the metric $g$ approaches the exact
metric tensor $\gex$, it is useful to work with $g$-independent test
functions that remain unchanged as $g \to \gex$. We will accomplish
this by constructing a $g$-independent test space $\Utest$ in
bijection with $\Atest$ in \S\ref{ssec:metr-indep-test}. {A key ingredient is the Hodge star isomorphism, depending on the metric, between $\W^2(\om)$ and $\W^{N-2}(\om)$ to translate the metric-dependent interface conditions \eqref{eq:cont_A} to independent ones, yielding the metric-independent space defined in~\eqref{eq:testspace_U} and used in Theorem~\ref{thm:distr_Riemann_U}.}
This then allows us to quantify changes in quantities that depend on  
the metric. Specifically, introducing intermediate metrics
between the exact metric $\gex$ and its non-smooth discretization ${g}$
by $\gpar(t)=\gex+t({g}-\gex)$, we study how elements of $\Atest$
evolve as $g(t)$ evolves with a fictitious ``time'' variable $t$
(in \S\ref{ssec:evol-test-funct}), followed by how each ingredient in
the generalized Riemann curvature tensor \eqref{eq:distr_Riemann}
evolves as $g$ evolves (in \S\ref{ssec:evol-riem-curv},
\S\ref{ssec:evol-second-fund}, and \S\ref{ssec:evol-angle-defic}). {In this section, we consider a fixed triangulation $\T$. However, by proving uniform estimates, the results can be extended to a sequence of triangulations $\T_h$ with meshsize $h\to 0$.}

\subsection{Metric-independent test space}
\label{ssec:metr-indep-test}
Let the symmetric dyadic product of two tensors $a$ and $b$ be denoted
by
$ a \odot b = \frac 1 2 \left( a \otimes b + b \otimes a \right)$.
Consider the vector bundle that associates to each point $p$ of $\om$
the vector space of linear combinations of symmetric dyadic products of
$(N-2)$-forms. We denote its smooth sections 
by
$\W^{N-2}(\T)^{\odot 2} \equiv \W^{N-2}(\T)\odot\W^{N-2}(\T)$.
In other
words, an element $U$ of $\W^{N-2}(\T)^{\odot 2}$ is a
linear combination of tensors of 
the form $u \odot v$ for some piecewise smooth forms $u$ and $v$ in
$\W^{N-2}(\T)$. 
The restriction of $U$ on an interior facet $F\in\Fint$, denoted by
$U_{\Fr}$ {\eqref{eq:restriction},} is a multi-valued function in general since its value depends on which of the elements sharing $F$ is used for the restriction. 
Requiring it to be single-valued, as done
next in~\eqref{eq:testspace_U}, places inter-element continuity
constraints. The following metric-independent test spaces with such continuity constraints on the
parameter domain $\om$, defined without using the Riemannian structure, are
useful when perturbing the metric:
\begin{align}
  \label{eq:testspace_U}
  \begin{split}
          \Utest&\! =\!\{
 U\in \W^{N-2}(\T)^{\odot 2}
          \,:\,U_{\Fr}(X_1,\dots, X_{N-2},Y_1,\dots,Y_{N-2}) \text{ is single-valued}\\
    &\qquad\qquad\qquad\qquad\qquad\qquad
                  \text{ for any } X_i, Y_i  \in\Xm{F}, F\in\Fint\},\\
    \Utesto&\! =\!\{ U\in \Utest: 
 U_{\Fr}(X_1,\dots, X_{N-2},Y_1,\dots,Y_{N-2})\! =\!0,
          \\
          & \qquad\qquad\qquad\qquad\qquad\qquad \text{ for any }
 X_i, Y_i \in \Xm{F},\,F\in\Fbnd\}.
  \end{split}
\end{align}
\begin{remark}
 One can easily verify that {the} continuity condition \eqref{eq:testspace_U} implies that for all $X_i, Y_i \in \Xm{E}$, $E\in\Eint$, $U_{\Er}(X_1,\dots, X_{N-2},Y_1,\dots,Y_{N-2})$ is single-valued {in a similar manner that we employed to establish Lemma~\ref{lem:cont_A_codim2}. In the two-dimensional case, $U\in\Utest$ is a scalar function, and in three dimensions it can be interpreted as a symmetric matrix, see \S\ref{sec:spec_2d} and \S\ref{sec:spec_3d}, respectively.}
\end{remark}
{We relate the spaces $\Utest$ and $\Atest$. To this end, let $\star$ be the Hodge dual operator \eqref{eq:Hodgestar}. We extend it to the symmetric product of $k$-forms, denote it by
  $\star^{\odot 2}:\W^k(\T)^{\odot 2}\to\W^{N-k}(\T)^{\odot 2}$, and define it by 
  \[
    \star^{\odot 2} (u \odot v) = (\star u)\odot (\star v),
    \qquad u, v \in \W^k(\T).
  \]
Then we define a linear mapping $\mapUA: \Utest \to \TT^4_0(\T)$ as
  \begin{equation}
    \label{eq:2}
    \mapUA = -\star^{\odot 2}.
  \end{equation}
 We also use the following equivalent definition of the mapping $\mapUA$
\begin{align}
  \label{eq:mapping_U_A}
 (\mapUA U)(X, Y, Z, W) =\langle U,\star(X^\flat\wedge Y^\flat)\odot\star( W^\flat\wedge Z^\flat)\rangle,\quad X, Y, Z, W \in \Xm \T,
\end{align}
where $\flat$ is the musical operator lowering indices, see the lines above
\eqref{eq:4} in the appendix. Indeed, given $U\in\Utest$, \eqref{eq:2} can be rewritten as
\begin{align*}
  -(\star^{\odot2}U)(X,Y,Z,W)&=\langle \star^{\odot2}U, X^\flat\otimes Y^\flat\otimes W^\flat\otimes Z^\flat\rangle\\
  &=\frac{1}{4}\langle \star^{\odot2}U, (X^\flat\wedge Y^\flat)\odot (W^\flat\wedge Z^\flat)\rangle\\
  &=\langle U, \star(X^\flat\wedge Y^\flat)\odot \star(W^\flat\wedge Z^\flat)\rangle =(\mapUA U)(X,Y,Z,W),
\end{align*}
where we have used \eqref{eq:hodge_star_iso} and \eqref{eq:4}. }
Denote the range of $\mapUA$ by
\begin{align}
  \label{eq:testspace_A2}
  \begin{split}
          \tilde{\mathcal{A}}_g \equiv  \tilde{\mathcal{A}} :=&\{ \mapUA U\,:\, U\in\Utest \}.
  \end{split}
\end{align}
Although $\Utest$ is independent of $g$, note that 
$\Atest$, $\tilde \Atest$, and  $\mapUA$ all depend on $g$. When needed, we will emphasize this dependence by writing them
as
$\Atest_g$, $\tilde \Atest_g$, and  $\mapUA_g$, respectively.
The next two lemmas show that \eqref{eq:testspace_A} and \eqref{eq:testspace_A2} define the same spaces. 

\begin{lemma}
 Let $g\in\RR^+(\T)$. Then $\tilde{\mathcal{A}}_g \subseteq \Atest_g$.
\end{lemma}
\begin{proof}
 Let $U \in \Utest$ and  $A = \mapUA_g U \in\tilde{\mathcal{A}}_g $. By
 definition~\eqref{eq:mapping_U_A}, $A$ clearly fulfills the
 symmetries in~\eqref{eq:symm-Riem} of $\Atest_g$, so it
 suffices to prove that the continuity constraint~\eqref{eq:cont_A}
 also holds.

 Let $F\in\Fint$ and $T\in \T$ be an element containing $F$. Let
  $\{E_1,\dots, E_{N-1}, E_N\}$
 be any $g$-orthonormal frame
 with $E_1,\dots, E_{N-1}\in \Xm{F}$
 and $E_N=\gn_F^T$, abbreviated to $\gn$ here. 
 Note that
  $E_1,\dots, E_{N-1}$ depend only on the tangential-tangential
 components of $g$ restricted to $F$ (as can be seen e.g.\ from
 Gram-Schmidt orthogonalization).
 Also, in the {components} of the $E_i$~frame {$g$ satisfies $g(E_i,E_j)=\delta_{ij}$.}

 First, consider the case $N>2$. Then by~\eqref{eq:Hodge-on-frame},
 the  Hodge duals
  \begin{align*}
    &\star (E_1\wedge E_N)
 = (-1)^{N-2}
 E_2 \wedge\dots\wedge E_{N-1},\\
    &\star (E_2\wedge E_N)
 = (-1)^{N-1}E_1 \wedge E_3 \wedge\dots\wedge E_{N-1}, 
  \end{align*}
 are expressed only in terms of tangential vectors on $F$.
 Letting 
  $X=E_1$, $W=E_2$, we then see that  
  \begin{align*}
 A(X,\gn,\gn,W)
    &= \langle U, \star (E_1\wedge E_N)\odot\star( E_2\wedge E_N)\rangle
  \end{align*}
 depends only on $U_{\Fr}(Z_1, \ldots, Z_{N-2}, Y_1, \ldots, Y_{N-2})$
 where $Z_i, Y_i \in \{ E_1, \ldots, E_{N-1}\}$
 are tangential vectors in $\Xm F$.
 Hence, by \eqref{eq:testspace_U}, $A(X,\gn,\gn,W)$ is single-valued
 on $F$.
 The same reasoning applies for all other choices
 of $X,W \in \{ E_1, \ldots, E_{N-1}\}$ and hence for any $X, W \in\Xm{F}$. 
 Thus, we conclude that $\Atnnt$
 is single-valued on all interior facets $F$.
  
 In the  $N=2$ case, 
 any $X, W \in \Xm F$ must be a  scalar multiple of $E_1$ and
  $\star (E_1 \wedge E_2) =1$, so 
  \begin{align*}
 A(E_1,\gn,\gn,E_1) &= \langle U, \star (E_1\wedge {E_2})\odot\star( E_1\wedge {E_2})\rangle = U,
  \end{align*} 
 which is single-valued on interior facets $F$ by
  \eqref{eq:testspace_U}.
\end{proof}

Next, we improve the inclusion of the previous lemma to an equality.

\begin{lemma}
  \label{lem:testspaceA}
 Let $g\in\RR^+(\T)$. Then $\tilde{\mathcal{A}}_g =\Atest_g$ and the mapping $\mapUA_g:\Utest\to\Atest_g$ is a bijection.
\end{lemma}
\begin{proof}
 First, we prove that $\mapUA_g$ is an injection. Let $U\in\Utest$ 
 be such that for $X, Y, Z, W \in \Xm \T$,
  \begin{align}
    \nonumber 
      0 = (\mapUA_g U)(X, Y, Z, W)
      & =
    \langle U,\star(X^\flat\wedge Y^\flat)
    \odot\star( W^\flat\wedge Z^\flat)\rangle
      \\     \label{eq:AgU=0}
      & =     \langle U,\star(X^\flat\wedge Y^\flat)
        \otimes \star( W^\flat\wedge Z^\flat)\rangle.
  \end{align}
 Within each mesh element, use a $g$-orthonormal basis
  $E^1, \ldots, E^N$ to expand {the symmetric tensor} $U$ in components $U_{\alpha\beta} E^{\alpha}\odot E^{\alpha}$ 
 for increasing
 multi-indices $\alpha=(\alpha_1, \ldots,\alpha_{N-2})$ and
  $\beta = (\beta_1, \ldots, \beta_{N-2})$, where $E^\alpha$
 abbreviates $E^{\alpha_1} \wedge \cdots \wedge E^{\alpha_{N-2}}$. 
  
 For each $\alpha$ and
  $\beta$, there are index pairs $i, j$ and $k, l$ and sign selections such that
  $ \pm E^i \wedge E^j \wedge E^\alpha
  $ and
  $ \pm E^k \wedge E^l \wedge E^\beta
  $ equal the volume form $\og$. Then, using
  $X = \pm E^i$ and $W = \pm E^k,$ with the selected signs,
 and $Y = E^j, Z = E^l$
 in~\eqref{eq:AgU=0},
  \begin{align*}
    0
     & \! =\!
 U_{\alpha\beta}
 ( \langle E^\alpha , \star(X^\flat\!\wedge\! Y^\flat) \rangle
      \langle E^\beta , \star( W^\flat\!\wedge\! Z^\flat)\rangle
      \!+\! 
      \langle E^\beta , \star(X^\flat\!\wedge\! Y^\flat) \rangle
      \langle E^\alpha , \star( W^\flat\!\wedge\! Z^\flat)\rangle).
  \end{align*}
 The term in parentheses equals $1$ for $\alpha \neq \beta$ and equals $2$ for $\alpha=\beta$, so  
  $U_{\alpha\beta}=0$.
 Repeating the argument for every coefficient of $U$, we
 conclude that $U=0$. 

 It now suffices to prove that the dimensions of $\Atest_g$ and $\Utest$ are equal. 
At any point,  the set of fourth-order tensors fulfilling the symmetry conditions \eqref{eq:symm-Riem} is $N(N-1)\big(N(N-1)+2\big)/8$-dimensional.
 The number of constraints in \eqref{eq:cont_A} at a facet point is $N(N-1)/2$ (and the linear independence of the constraints can be verified by inserting a basis of $\Xm{F}$). To see that these numbers match those for $\Utest$ in \eqref{eq:testspace_U},  first note that    $\dim\big(\W^{N-2}(\T)^{\odot 2}\big)=N(N-1)\big(N(N-1)+2\big)/8$. 
 Next, the number of linearly independent constraints in \eqref{eq:testspace_U} is also $N(N-1)/2$ on facets equaling the dimension of symmetric $(N-1)\times (N-1)$ matrices. Thus, the dimensions of $\Atest_g$ and $\Utest$ coincide, and $\mapUA_g$ is a bijection.
\end{proof}

Recall that \eqref{eq:distr_Riemann} generalized the densitized
Riemann curvature to a functional $\Rog(A)$ acting on $A$ in
$\Atest$. In view of the bijection $\mapUA$, we now consider the corresponding  functional on $U$, defined by 
\begin{align}
  \label{eq:def_curv_riemann}
  \RogU (U) :=\Rog( \mapUA U).
\end{align}

\begin{theorem}  \label{thm:distr_Riemann_U}
 The generalized curvature functional $\RogU (U)$
 can be calculated using 
 metric-independent test functions $U\in \Utesto$ by
\begin{align*}
  \begin{split}
    \RogU(U)
    &  =\sum_{T\in\T}
      \int_T\langle \Riemann,\mapUA U\rangle\,\vo{T}
      +
      4\sum_{F\in\Fint}\int_F\langle\jmp{\sff},
      \tnnt{\mapUA U}\rangle\,\vo{F}
      \\
      &\qquad+4
      \sum_{E\in\Eint}\int_E \Theta_E\, (\mapUA U)_{\gcn\gn\gn\gcn}\,\vo{E}.
  \end{split}
\end{align*}
\end{theorem}
\begin{proof}
 Apply Lemma~\ref{lem:testspaceA}.
\end{proof}

\begin{remark}[The Uhlenbeck trick]
 Our technique of putting the metric-dependent test function space $\Atest$
 in isomorphism with the fixed,
 metric-independent space $\Utest$
 may be viewed as  a variant of a well-known idea of
 Karen Uhlenbeck: cf.~\cite{Ham1986} or \cite[Section~9.4]{Top2006}.
\end{remark}
\begin{remark}[Finite elements for $\Utest$]
 In 2D,
  \eqref{eq:distr_Riemann} reduces to \eqref{eq:distr_Gauss},
 and the test space $\Utest$ can be discretized by  Lagrange finite elements~\eqref{eq:lag_fe}.
 In three spatial dimensions $\Utest$ consists of Regge functions, which are $tt$-continuous, and can be discretized using Regge finite elements~\eqref{eq:regge_fem_space}. The specialization of \eqref{eq:distr_Riemann} in 3D is
 presented and discussed in \S\ref{sec:spec_3d}. 
\end{remark}

\begin{remark}[Curvature operator]
  \label{rem:curv-oper}
 Often, the smooth Riemann curvature tensor $\Riemann$ is identified with an
 operator acting on bivectors (see e.g., \cite[Section
 3.1.2]{Peter16}) due to the symmetries of $\Riemann$. In our
 context, we can similarly define, within each element $T$, a {\em curvature
 operator} acting on 2-forms by setting its quadratic form,  
  $\CurvatureOther: \W^2(T) \times \W^2(T) \to \R$ by
  \begin{align}
    \label{eq:def_curvature_op}
    \CurvatureOther(X^\flat\wedge Y^\flat, W^\flat\wedge Z^\flat):= \Riemann(X,Y,Z,W),\quad \forall X,Y,Z,W\in\Xm{T}.
  \end{align}
 Clearly, $\CurvatureOther$ is a symmetric bilinear form in its two
 arguments, and generates a self-adjoint curvature
 operator.
 Next, define
  \begin{equation}
    \label{eq:Q-defn}
    \Curvature = \mapUA^{-1} \Riemann.
  \end{equation}
 Then, due to \eqref{eq:mapping_U_A}, the equation
  $\mapUA \Curvature = \Riemann$ can be written as
  \[
    \ip{ \Curvature, \star(X^\flat\wedge Y^\flat)\odot\star( W^\flat\wedge Z^\flat) }
 = \Riemann(X, Y, Z, W)
  \]
 for all $X,Y,Z,W\in\Xm{T}$.
 Comparing with~\eqref{eq:def_curvature_op}, we see that $\Curvature$ and
  $\CurvatureOther$ are related by
  \[
    \CurvatureOther(X^\flat\wedge Y^\flat, W^\flat\wedge Z^\flat)
 =\ip{ \Curvature, \star(X^\flat\wedge Y^\flat)\odot\star(
 W^\flat\wedge Z^\flat) } = \Riemann(X, Y, Z, W).    
  \]
 Thus, $\Curvature, \CurvatureOther$, and $\Riemann$ all contain the same
 information.  A generalized densitized version of $\Curvature$ is
 the functional $\RogU$ of Theorem~\ref{thm:distr_Riemann_U}.  In
 \S\ref{sec:spec_3d}, we will revisit the curvature operator in 3D.
\end{remark}

{\begin{remark}[Positive curvature operator]
  \label{rem:pos_curv_op}
  For smooth metrics, the positive curvature operator is a classical notion; see \cite[Chapter~9]{Peter16} for a thorough treatment in the context of the Bochner technique. To identify when the distributional curvature operator is ``positive,'' we take cue
  from the notion that a distributional operator is positive if its
  application to any positive test function yields a positive
  scalar. Generalizing to our context, by linearity, viewing each
$U\in\Utest\subset \W^{N-2}(\T)^{\odot 2}$ pointwise as a symmetric
covariant bilinear form on the space $\W^{N-2}T_p\T$, 
we call the tensor field $U$ positive if
$U_p(\xi,\xi)>0$ for every nonzero
$\xi\in\W^{N-2}T_p\T$ at every point $p$ in an element interior.
We declare $\RogU$ to be a positive curvature
  operator if $\RogU(U)>0$ for all such positive $U$.  We claim that
  $\RogU$ is positive if the angle defect $\Theta_E$ is a positive
  scalar, the jump of the second fundamental form~$\jmp{\sff}$ is
  positive definite, and the element-wise smooth curvature operator
  from \eqref{eq:Q-defn}, namely $\Curvature=\mapUA^{-1}\Riemann$ in
  $\Utest$, is positive. Indeed, for any positive $U\in\Utesto$ there
  holds for the codimension 1 and 2 terms
  \begin{align*}
    (\mapUA U)(\gcn,\gn,\gn,\gcn)&=\langle U,\star(\gcn^\flat\wedge \gn^\flat)\odot \star(\gcn^\flat\wedge \gn^\flat)\rangle>0,\,&&\forall E\in\Eint,\\
    \tnnt{\mapUA U}(X,X)&=\langle U,\star(X^\flat\wedge \gn^\flat)\odot \star(X^\flat\wedge \gn^\flat)\rangle>0,\,&&\forall X\in\Xm{F},\,F\in\Fint.
  \end{align*} 
  Based on the fact that the inner product of two self-adjoint
  positive definite endomorphisms on a finite-dimensional space is
  positive, the two inner products $\langle \Curvature, U\rangle$ and
  $\langle\jmp{\sff}, \tnnt{\mapUA U}\rangle$, whose arguments are
  interpreted as symmetric positive quadratic forms on the vector
  spaces $\W^{N-2}T\T$ and $\Xm{\T}$, respectively, are
  positive. This then yields the claim. It confirms Gromov's
  suggestion (quoted in Section~\ref{sec:intro}) that positive angle
  defects be interpreted as generating a positive curvature operator
  rather than a positive sectional curvature (and extends to beyond the piecewise flat case).
\end{remark}}

We conclude this subsection by giving coordinate formulas for the map
$\mapUA$ and its inverse,  anticipating its utility in computations.  
Let $\alpha = (\alpha_1, \ldots, \alpha_{N-2})$ and
$\beta= (\beta_1, \ldots, \beta_{N-2})$ denote multi-indices of
integers $\alpha_m, \beta_m \in \{1, \ldots, N\}$.  Let $g^{\alpha\beta} =
g^{\alpha_1 \beta_1} \dots g^{\alpha_{N-2} \beta_{N-2}}$. Mixing integer
indices $p, q$ and multiindex $\alpha$, {we write for the Levi-Civita symbol \eqref{eq:Levi-Civita-symbol}}
$\gveps^{pq\alpha}$ for
$\gveps^{p,q, \alpha_1 \dots \alpha_{N-2}}$ (and similarly
$\eveps^{pq\alpha}$).  Given $U \in \Utest,$ and a coordinate frame
$\{ \d_i \}_{i=1}^N$, we wish to relate the coefficients
\[U_{\alpha \beta} \equiv U_{\alpha_1, \ldots, \alpha_{N-2}, \beta_1,
  \ldots, \beta_{N-2}} := U(\d_{\alpha_1}, \ldots, \d_{\alpha_{N-2}},
\d_{\beta_1}, \ldots, \d_{\beta_{N-2}}), 
\]
with $A_{ijkl} = A(\d_i, \d_j, \d_k, \d_l)$ when $A = \mapUA U$.  We
extend the summation convention to multi-indices $\alpha$ and $\beta$
(so a sum over all components $\alpha_m$ and $\beta_m$ is implied in
the formula~\eqref{eq:mapping_U_A_coo} below).

\begin{proposition}
  \label{prop:mapA_coo}
 For any $U \in \Utest,$ using the above-mentioned coordinate notation,
  \begin{align}
    \label{eq:mapping_U_A_coo}
    \begin{split}
      &[\mapUA U]_{ijkl} 
 =\frac{-1}{[(N-2)!]^2}\gveps^{pq\alpha}\, \gveps^{rs\beta}\,U_{\alpha\beta}\,
 g_{pi} g_{qj}  g_{rk}g_{sl},\\
      & [\mapUA U]^{ijkl} 
 =\frac{-1}{[(N-2)!]^2}\gveps^{ij\alpha}\, \gveps^{kl\beta}\,U_{\alpha\beta}.
    \end{split}
  \end{align}        
\end{proposition}
\begin{proof}
 By the definition of $\mapUA$ in \eqref{eq:mapping_U_A},
  $[\mapUA U]_{ijkl} = \ip{ U, \star(\d_i^\flat\wedge
    \d_j^\flat)\odot\star( \d_l^\flat\wedge \d_k^\flat)}$.  Noting
 that \eqref{eq:Hodge-dxpdxq} implies
  \[
    \star(\d_i^\flat\wedge \d_j^\flat) = \frac{\sqrt{\det g}}{(N-2)!}
    \,\eveps_{ij\,\beta} \;dx^{\beta_1}\wedge\dots\wedge dx^{\beta_{N-2}},
  \]
 we obtain, using \eqref{eq:innerprod-extended}, 
  \begin{align*}
 [\mapUA U]_{ijkl}
    & = \frac{\det g}{[(N-2)!]^2}
 U_{\alpha\beta} g^{\alpha\gamma} g^{\beta {\eta}} \eveps_{ij\,{\gamma}} \eveps_{lk\,{\eta}}.
  \end{align*}
 Now simplifying using $\eveps_{ij\,{\gamma}} g^{\alpha\gamma}  = 
 g_{pi} g_{qj} \eveps_{rs\,{\gamma}} g^{rp} g^{sq} g^{\alpha\gamma}  =
 g_{pi} g_{qj} \eveps^{pq\,\alpha} \det(g^{-1}),
  $
 we obtain the first expression in \eqref{eq:mapping_U_A_coo}. The second follows by the direct computation of $[\mapUA U]^{ijkl} = g^{ip} g^{jq} g^{kr} g^{ls} [\mapUA U]_{pqrs}$.
\end{proof}

\begin{proposition}
The inverse $\mapUA^{-1}:\Atest\to \Utest$ has the explicit form
\begin{align}
  \label{eq:mapping_U_A_inv}
  \begin{split}
 (\mapUA^{-1}A)&(X_1,\dots,X_{N-2},Y_1,\dots,Y_{N-2}) \\
    &= -\frac{1}{4}\langle A, \star(X_1^\flat\wedge\dots\wedge X_{N-2}^\flat)\odot \star(Y_1^\flat\wedge\dots\wedge Y_{N-2}^\flat)\rangle,
  \end{split}
\end{align} for any $A \in \Atest$, and has the representation $\mapUA^{-1} = -\star^{\odot 2}$.  In coordinates, for multi-indices
$\alpha = (\alpha_1, \ldots, \alpha_{N-2})$ and
$\beta= (\beta_1, \ldots, \beta_{N-2})$,
\begin{align}
  \label{eq:mapping_U_A_inv_coo}
 (\mapUA^{-1}A)_{\alpha\beta} = -\frac{1}{4}\gveps_{\alpha\,ij}\gveps_{\beta\,kl}A^{ijkl}.
\end{align}
\end{proposition}
\begin{proof}
 To show that \eqref{eq:mapping_U_A_inv} is indeed the inverse we prove that $\mapUA^{-1}\mapUA=\idop$. First, we prove the representation $\mapUA^{-1} = -\star^{\odot 2}$ of \eqref{eq:mapping_U_A_inv}.
 By~\eqref{eq:hodge_star_iso},
    \begin{align*}
 (\mapUA^{-1}A)&(X_1,\dots,X_{N-2},Y_1,\dots,Y_{N-2})\\
      &=-\frac{1}{4}\langle A, \star(X_1^\flat\wedge\dots\wedge X_{N-2}^\flat)\odot \star(Y_1^\flat\wedge\dots\wedge Y_{N-2}^\flat)\rangle \\
      &= -\frac{1}{((N-2)!)^2}\langle \star^{\odot 2}A, X_1^\flat\wedge\dots\wedge X_{N-2}^\flat\odot Y_1^\flat\wedge\dots\wedge Y_{N-2}^\flat\rangle\\
      &= -\langle \star^{\odot 2}A, X_1^\flat\otimes\dots\otimes X_{N-2}^\flat\otimes Y_1^\flat\otimes\dots\otimes Y_{N-2}^\flat\rangle \\
      &= -(\star^{\odot 2}A)(X_1,\dots,X_{N-2}, Y_1,\dots, Y_{N-2}).
    \end{align*}
 Then, we have for all $U\in \Utest$ with \eqref{eq:2}  (and that $\star\circ\star=\pm 1$ \eqref{eq:hodge_star_twice})
    \begin{align*}
      \mapUA^{-1}(\mapUA U) = \star^{\odot 2}(\star^{\odot 2} U) = U.
    \end{align*}
 Coordinate expression \eqref{eq:mapping_U_A_inv_coo} can be derived analogously to \eqref{eq:mapping_U_A_coo}.
\end{proof}

\subsection{Evolution of test functions}
\label{ssec:evol-test-funct}

In the remainder of this section, we consider a family of 
metrics $\gpar(t)\in \Cone[\R, \Sc^+]$ parameterized by time $t$
and examine how relevant quantities change as
$\gpar(t)$ changes with time $t$.
{Write} $\gdot(t) = (d g_{ij} / dt) dx^i \otimes dx^j$ in a $t$-independent coordinate coframe $dx^i$. 
The specific task for this subsection is to investigate how the test tensor field  $A\in\Atest$ changes as the metric $g$ evolves.
Here, we rely on definition \eqref{eq:testspace_A2} and mapping \eqref{eq:mapping_U_A}. 
To this end, we need some preliminaries. Let $X_i \in \Xm\T$ be
time-independent vector fields and let $\sigma \in \TT^2_0(\T)$.  Let
$\Lsigma: \Xm \T \to \Xm \T$ be the endomorphism defined by
\begin{equation}
  \label{eq:Ldefn}
 g(\Lsigma X_1, X_2) = \sigma(X_1, X_2).
\end{equation}
By trace of $\sigma$ we mean the usual trace of the linear operator $\Lsigma$,
\begin{equation}
  \label{eq:trace-2-tensor}
  \tr \sigma := \tr {\Lsigma} {=[\Lsigma]_i^i} = g^{ij} \sigma_{ij}.
\end{equation}
It is easy to prove (e.g., using the
Jacobi formula for the derivative of the determinant) that the derivative of the
time-dependent Riemannian volume form $\omega \equiv \og_{g(t)}$ is
given by
\begin{equation}
  \label{eq:derivative-vol}
  \frac{d}{dt} \og_{g(t)}(X_1, \ldots, X_N) = \frac 1 2 \tr\gdot \, \og_{g(t)}(X_1, \ldots, X_N),
\end{equation}
where the $\tr \gdot$ is calculated as in~\eqref{eq:trace-2-tensor}.
{For codimension 1 and 2 subsimplices $D$ we define $\LsigmaD$ as in \eqref{eq:Ldefn} with $g_D$ and restricting to vector field  $X_i\in\Xm{D}$.}

Given any  $A \in \TT^k_0(\T)$, define
$\Lsigma^{(\ell)} A \in \TT_0^k(\TT)$, for any
$\ell =1, \ldots, k$,  by
\begin{align}
  \label{eq:Lsig-defn}
 (\Lsigma^{(\ell)}A)(X_1, \ldots, X_k)
 := A(X_1, \ldots, X_{\ell-1}, \Lsigma X_\ell, X_{\ell+1}, \ldots, X_k).
\end{align}
When $\sigma$ is symmetric, it is easy to verify that
\begin{align}
  \label{eq:Lsig-symm}
  \ip{\Lsigma^{(\ell)} A, B} & = \ip{ A, \Lsigma^{(\ell)} B}, \qquad
 A, B \in \TT_0^k (\T), 
  \\
  \label{eq:Lsig-1}
  \ip{ \Lsigma^{(\ell)} A, B} & = \ip{ \Lsigma^{(1)}A,  B}, \qquad
 A, B \in \Atest, \;\ell=1, \dots, 4.                                 
\end{align}
{Relying on Leibniz' formula,} define $\dot{A} \in \TT^k_0(\T)$ by
\[
  \dot{A}(Y_1, \ldots, Y_k)
 = \frac{d}{dt} A(Y_1, \ldots, Y_k)
  - \sum_{i=1}^k A\left( Y_1, \ldots, \frac{d Y_i}{dt}, \ldots, Y_k\right)
\]
for possibly time-dependent vector fields $Y_i \in \Xm \T$. When $Y_i$ are time-independent, the last sum vanishes. 
In components of the time-independent coordinate frame $\{\d_1,\dots,\d_N\}$, this means that the components of $\dot{A}$ are obtained by differentiating the components of $A$, i.e.,  
$\dot A_{i_1, \ldots, i_k} =  d A_{i_1, \ldots, i_k} / dt$.
Using this notation, we have the following lemma.

\begin{lemma}
  \label{lem:evolution_scalar_product}
 For any  $A,B\in\TT^k_0(\T)$,  {and $C,D\in \TT^k_0(\F)$,}
  \begin{align*}
          \frac{d}{dt}\langle A,B\rangle
          &=\langle\dot A, B\rangle
            +
            \langle A, \dot B\rangle
            -
            \sum_{\ell=1}^k \langle \Ldotg^{(\ell)} A, B\rangle,\\
            {\frac{d}{dt}\langle C,D\rangle}
          &{=\langle\dot C, D\rangle
            +
            \langle C, \dot D\rangle
            -
            \sum_{\ell=1}^k \langle \LdotgF^{(\ell)} C, D\rangle.}
  \end{align*}
\end{lemma}
\begin{proof}
  {We only prove the first statement; the second follows analogously.} Since  $\dot{g}^{ij} =-g^{i{m}}\,\dot{g}_{{m}l}\,g^{lj}$,
  {by \eqref{eq:innerprod-extended},}
  \begin{align*}
    \frac{d}{dt}\langle A,B\rangle
    &
 =
      \frac{d}{dt}(
 A_{i_1\dots i_k}g^{i_1j_1}\ldots g^{i_kj_k}B_{j_1\dots j_k})      
    \\
    & =
      \langle{\dot A, B}\rangle
      +
      \langle{A, \dot B}\rangle
      -
      \sum_{\ell=1}^k
 A_{i_1\dots i_k}\; g^{i_1j_1}\cdots
 (g^{i_\ell {m}}\dot{g}_{{m}l}g^{l j_\ell} ) \cdots 
 g^{i_kj_k} \; B_{j_1\dots j_k}
  \end{align*}
 and the $\ell$th summand in the last sum is easily seen to equal 
  $\langle \Ldotg^{(\ell)} A, B\rangle$.
\end{proof}

\begin{lemma}
  \label{lem:identity_volume_form_sym_matrix}
 Let $\omega\in{\W^N}(\T)$ be the volume form. Then for symmetric
 2-tensors $\sigma\in \Sc(\T),$
  \begin{align}
    \label{eq:identity_volume_form_sym_matrix}
    \sum_{i=1}^N \Lsigma^{(i)}\omega
    &= \tr{\sigma}\,\omega.
  \end{align}
\end{lemma}
\begin{proof}
 Let $X_i \in \Xm \T$, and $E_i \in \Xm \T$ be an oriented
  $g$-orthonormal frame. Further, let $b_i \in \R^N$ denote the column vector
 whose $j$th component equals {$g(X_i, E_j)$}, and $L$ be the
 matrix whose $(i, j)$th entry is $\sigma(E_i, E_j)$.  Then, by
  \eqref{eq:Ldefn}, the $j$th component of $L b_i$ equals
  {$g(\Lsigma X_i, E_j)$} and the trace of the matrix $L$ equals $\tr\sigma$.
 Combined with the fact that
  $\omega(X_1, \ldots, X_N) = \det [b_1, \ldots, b_N]$, we see that
  \eqref{eq:identity_volume_form_sym_matrix} is the same as
  \begin{equation}
    \label{eq:determinant-identity} 
    \sum_{i=1}^N\det[b_1,\dots, L\,b_i,\dots,b_N]= \tr{L}\,\det[b_1,\dots,b_N].
  \end{equation}
 The identity~\eqref{eq:determinant-identity} actually holds for
 arbitrary matrices $L \in \R^{N \times N}$ and  vectors $b_i \in \R^N$
 and can be readily verified as follows. Both sides of~\eqref{eq:determinant-identity} are alternating and linear in $b_i$, so they are $N$-forms, and can only differ by a scalar factor. That the scalar factor, in this case, is $\tr{L}$, can be seen by substituting the standard Euclidean unit basis vectors for $b_i$.
\end{proof}

\begin{lemma}
  \label{lem:evolution_test_function}
 For time-independent vector fields $X,Y,Z,W\in \Xm \T$ and $A\in\Atest,$
 the time derivative of $A(X, Y, Z, W)$ is given by 
  \begin{align*}
    \dot{A}(X,Y,Z,W)
    & = 
 A(\Ldotg X,Y,Z,W) +A(X,\Ldotg Y,Z,W)  \nonumber
    \\
    &
      + A(X,Y,\Ldotg Z,W)+A(X,Y,Z,\Ldotg W) -\tr{\gdot}A(X,Y,Z,W).
  \end{align*}
\end{lemma}
\begin{proof}
 By Lemma~\ref{lem:testspaceA}, there exists a time-independent
  $U \in \Utest$ such that $A = \mapUA_g U$, which by
  \eqref{eq:mapping_U_A} is the same as
  $ A(X, Y, Z, W) = \langle{U, V} \rangle $ {is true for any}
  $V = \star(X^\flat\wedge Y^\flat)\odot\star( W^\flat\wedge
 Z^\flat)$. Since $\dot{U} = 0$, by Lemma~\ref{lem:evolution_scalar_product},
  \begin{equation*}
    \dot{A}(X,Y,Z,W) = \langle{U, \dot{V}}\rangle
    -
    \sum_{i=1}^{2(N-2)}
    \langle \Ldotg^{(i)} U, V \rangle.    
  \end{equation*}  
 We compute $\dot{V}$ using 
  \begin{align*}
    \frac{d}{dt}\big(\star (X^\flat\wedge Y^\flat)\big)
    & =
      \frac{d}{dt} \og(X, Y, \dots)
      && \text{by \eqref{eq:Hodgestar-vecs}},
    \\
    & =  \frac{1}{2}\tr[]{\gdot}\; \star(X^\flat\wedge Y^\flat)
      && \text{by~\eqref{eq:derivative-vol} {and \eqref{eq:Hodgestar-vecs}}}.
  \end{align*}
 Then 
  \begin{align}
    \label{eq:Adot}
    \dot{A}(X,Y,Z,W)
    &=\tr[]{\gdot}\,\langle U, V\rangle
      -
    \sum_{i=1}^{2(N-2)}    \langle U, \Ldotg^{(i)}  V \rangle.          
  \end{align}
  {Applying once more \eqref{eq:Hodgestar-vecs} and Lemma~\ref{lem:identity_volume_form_sym_matrix} to $V=\star(X^\flat\wedge Y^\flat)\odot\star(Z^\flat\wedge W^\flat)$,}
  {\begin{align*}
    \sum_{i=1}^{2(N-2)}    \Ldotg^{(i)}  V &=2\tr \gdot V- \star((\Ldotg X)^\flat \wedge Y^\flat) \odot\star (W^\flat \wedge Z^\flat)
    \\
    &\quad
      - \star(X^\flat \wedge (\Ldotg Y)^\flat) \odot\star (W^\flat \wedge Z^\flat) - \star(X^\flat \wedge Y^\flat) \odot\star ((\Ldotg W)^\flat \wedge Z^\flat)
    \\
    &\quad
      - \star(X^\flat \wedge Y^\flat) \odot\star (W^\flat \wedge (\Ldotg Z)^\flat).
  \end{align*}}
 Taking the inner product with $U$
   {and substituting in} \eqref{eq:Adot}, we obtain the result.  
\end{proof}

\subsection{Evolution of Riemann curvature tensor}
\label{ssec:evol-riem-curv}

We proceed to study how each term in our formula for the
generalization of Riemann curvature~\eqref{eq:distr_Riemann} evolves
in time with $g(t)$. This subsection focuses on the integrand over
$T$ in \eqref{eq:distr_Riemann} and gives its time derivative (in Lemma~\ref{lem:evolution_vol}).

Given a $k$-permutation $\pi \in \Perm_k$,  let
$(S_\pi A) (X_1, \ldots, X_k) = A (X_{\pi(1)}, \ldots, X_{\pi(k)})$
for all $X_i \in \Xm{\T}$ and $A \in \TT_0^k (\T)$. A transposition $(i, j)$
thus generates $S_{(i, j)}$ which swaps the $i$th and $j$th input
arguments of a tensor. 
The transposition  $S_{(2, 3)}$ on 4-tensors is of particular interest to us: we abbreviate it to  $S$, i.e., 
\begin{equation}
  \label{eq:defn-S}
 (S A) (X, Y, Z, W) = A(X, Z, Y, W),
\end{equation}
for any $X, Y, Z, W \in \Xm{\T}$ and $A \in \TT_0^4(\T)$,  or, in coordinates,
$S_{ijkl}^{pqrs} = \delta_i^p \delta_j^r \delta_k^q \delta_l^s$ so that 
$S_{ijkl}^{pqrs} A_{pqrs} = A_{ikj l}$. 
It is easy to see that $S_{(i,j)}$ is self-adjoint, 
\begin{equation}
  \label{eq:S-selfadj}
  \ip{ S_{(i,j)} A, B} = \ip{ A, S_{(i,j)} B}
\end{equation}
for all $A, B \in \TT_0^4 (\T)$. The operation of skew-symmetrizing a
tensor $A$ with respect to its $i$th and $j$th arguments is
$P_{ij} = \frac 1 2 (I - S_{(i, j)})$. Of particular interest to us is
$P = P_{13} \circ P_{24} : \TT_0^4 (\T) \to \TT_0^4 (\T)$, which
skew-symmetrizes with respect to arguments 2, 4, followed by 1,
3. Expanding $P$ in terms of $S_{(i, j)}$, 
\begin{equation}
  \label{eq:proj-using-S}
 P = \frac 1 4 \big(I - S_{(1, 3)} - S_{(2,4)} + S_{(1,3)} \circ S_{(2,4)}\Big).   
\end{equation}
Disjoint transpositions commute, so
$S_{(1,3)} \circ S_{(2,4)} = S_{(2,4)} \circ S_{(1,3)}$. This fact, together with~\eqref{eq:S-selfadj}, immediately implies that $P$ is a
self-adjoint projection,
\begin{equation}
  \label{eq:P_selfad-proj}
  \ip{ P A, B} = \ip{ A, P B}, \qquad
 P^2 = P.
\end{equation}
Note that for any covariant 4-tensor $A$, combining~\eqref{eq:defn-S}
and~\eqref{eq:proj-using-S},
\begin{equation}
  \label{eq:SPA}
  \begin{aligned}
 (SP A)& (X, Y, Z, W)
 = (PA) (X, Z, Y, W)\\
 = \frac{1}{4}\Big[ & A(X,Z, Y,W)-A(Y, Z,X,W)
                         + A(Y,W,X,Z) -A(X,W,Y, Z)  \Big].
  \end{aligned}
\end{equation}
Also note that by the symmetries of the test space $\Atest$, 
\begin{equation}
  \label{eq:PSA}
 (PSA) (X, Y, Z, W) = (SA)(X, Y, Z, W), \qquad \text{ for any } A \in \Atest.
\end{equation}
Using these operators, as well as the second-order covariant derivative (see
\eqref{eq:2ndDerivative}), we obtain
{a} characterization of the derivative of $\Riemann$. {The following result is a reformulation of the general theory of a one-parameter deformation of a metric $g(t)$ \cite[Proposition 2.3.5]{Top2006}. The involved connection $\nabla$ has to be understood with respect to the metric $g=g(t)$ for the value of the parameter $t$ in which the derivative $\dot{g}(t)$ is taken.} 

\begin{lemma}
  \label{lem:evolution_Riemann}
 The action of the time derivative of the Riemann curvature tensor
  \eqref{eq:RiemannCurvTensor} on time-independent vector fields
  $X,Y,Z,W\in\Xm{T}$, for any $T \in \T$,  is given by  
  \begin{align*}
    \dot \Riemann (X, Y, Z, W)
    &\! =\!  2(SP\nabla^2\gdot)(X,Y,Z,W) + \frac 1 2
 \Big[\Riemann(X,Y,Z,\Ldotg W)+\Riemann(X,Y,\Ldotg Z,W)\Big].
  \end{align*}
\end{lemma}
\begin{proof}
 By \cite[Proposition 2.3.5]{Top2006} (accounting for
 the  different sign convention of the curvature endomorphism {chosen there and the permuted roles of $Z$ and $W$}),
  \begin{align*}
    \dot \Riemann (X, Y, Z, W)
    &= \frac 1 2
 \bigg[
 (\nabla_{X, Z}^2  \gdot) (Y, W)-(\nabla_{X, W}^2 \gdot) (Y,Z)
      + \gdot(R_{X, Y}Z, W)
    \\
    &\quad\; +(\nabla_{Y, W}^2 \gdot) (X, Z)-(\nabla_{Y, Z}^2 \gdot)  (X, W) 
      -\gdot(R_{X, Y} W, Z)
 \bigg].
  \end{align*}
  \eqref{eq:Ldefn} and \eqref{eq:RiemannCurvTensor} imply
  $\gdot(R_{X, Y}Z, W) = g(R_{X, Y}Z, \Ldotg W) = \Riemann(X, Y, Z,
  \Ldotg W)$. Hence, the result follows from~\eqref{eq:SPA}.
\end{proof}

\begin{lemma}
  \label{lem:evolution_vol}
 Let $A\in\Atesto$, $T\in\T$. Then 
  \begin{align*}
    \frac{d}{dt}\Big(\langle\Riemann,A\rangle\,\vo{T}\Big)
    &=\left(2\langle \nabla^2\gdot,SA\rangle +
      \langle \Ldotg^{(1)}\Riemann,A\rangle
      -\frac{\tr{\gdot}}{2}\langle\Riemann,A\rangle\right)\vo{T}.
  \end{align*}
\end{lemma}
\begin{proof}
 By {Lemma~\ref{lem:evolution_scalar_product}} and \eqref{eq:derivative-vol}, 
  \[
    \frac{d}{dt}
 \Big(\langle\Riemann,A\rangle\,\vo{T}\Big) =
 \big(\langle\dot{\Riemann},A\rangle
    +\langle \Riemann,\dot{A}\rangle-\sum_{\ell=1}^4
    \langle \Ldotg^{(\ell)}\Riemann, A\rangle
    +\frac{1}{2}\langle\tr{\gdot}\Riemann,A\rangle\big)\vo{T}.
  \]
 The first two terms admit  the following identities:  
  \begin{align*}
    \ip{\Riemann, \dot{A}}
    & = \sum_{\ell=1}^4
      \ip{ \Ldotg^{(\ell)}  \Riemann, A }
      -\langle\tr{\gdot}\Riemann,A\rangle,
    &&\text{by Lemma~\ref{lem:evolution_test_function}
 and~\eqref{eq:Lsig-symm},}     
    \\
    \ip{\dot\Riemann, A}
    & = 2\ip{SP \nabla^2 \gdot, A}  +
      \frac 1 2 \sum_{\ell=3}^4 \ip{\Ldotg^{(\ell)} \Riemann, A},
    && \text{by Lemma~\ref{lem:evolution_Riemann},}
    \\
    & = 2\ip{ \nabla^2 \gdot, SA}  +
      \ip{\Ldotg^{(1)} \Riemann, A},
    &&\text{by~\eqref{eq:S-selfadj}, \eqref{eq:P_selfad-proj}, \eqref{eq:PSA}, and \eqref{eq:Lsig-1}.}
  \end{align*}
 Using them and simplifying, we obtain the result.
\end{proof}

\subsection{Evolution of second fundamental form}
\label{ssec:evol-second-fund}

Next, we consider the codimension-1 boundary term in
\eqref{eq:distr_Riemann} and compute its time derivative (in
Lemma~\ref{lem:evolution_bnd}).  Let us begin by recalling that the
time derivative of the Levi-Civita connection can be computed by
differentiating the terms in the Koszul formula: specifically,
by~\cite[Proposition~2.3.1]{Top2006},
\begin{equation}
  \label{eq:linearedKoszul}
 g\left(\frac{d}{d t} (\nabla_Y X), Z\right)
 =  \frac{1}{2} \big[
 (\nabla_X \gdot) (Y, Z)
  +
 (\nabla_Y \gdot) (X, Z)
  -
 (\nabla_Z \gdot)(X, Y) \big]  
\end{equation}
for any time-independent vector fields $X,Y, Z\in \Xm \T$. Since both
sides are linear in $Z$,
\eqref{eq:linearedKoszul} also holds for time-dependent
$Z \equiv Z(t) \in \Xm \T$ by expanding $Z(t)$ in a time-independent
frame such as $\d_i$. The next auxiliary result we need is as follows
(cf.~\cite[Lemma 2.5]{GN2023}).

\begin{lemma} \label{lem:covdot} Let $F \in \F$, $X(t) \in \Xm \T$ be
 a possibly time-dependent vector field, and let $\gn$ denote a
  $g$-normal vector on $F$. Then, 
 abbreviating $\gdot(\gn,\gn)$ to $\gdot_{\gn\gn}$, 
  \begin{align}
    \label{eq:dgn/dt}
    \frac{d \gn}{dt}
    &
 = \frac{1}{2}\gdot_{\gn\gn}\gn-\Ldotg\gn,
    \\     \label{eq:dt-g(n,X)}
    \frac{d}{d t} g(\gn,X)
    &= 
      \frac{1}{2} \gdot_{\gn\gn} g(\gn,X) + g(\gn,\dot{X}).
  \end{align}
\end{lemma}
\begin{proof}
 Let $\{\gt_i\}_{i=1}^{N-1}$ denote a $g(t)$-orthonormal basis of the
 tangent space at a point in $F$. Since the tangent space $\Xm{F}$ depends
 only on the triangulation $\T$ (and is independent of $t$), the
 difference quotients $(\gt_i(t_2) - \gt_i(t_1)) / (t_2 - t_1)$ are in
  $\Xm F$ at different times $t_1, t_2$, so
  $\dot{\gt}_i = d \gt_i / dt \in \Xm F$ and
  $g( \gn, \dot{\gt}_i) =0$. Hence, by the product rule,
  \begin{align*}
    &0=\frac{d}{dt}g(\gn,\gt_i)
 =\gdot(\gn,\gt_i) + g(\dot{\gn},\gt_i). 
    \intertext{Since $g(\gn,\gn)=1$ for all $t$, we also have}
    &0=\frac{d}{dt}g(\gn,\gn) = \gdot(\gn,\gn) + 2g(\dot{\gn},\gn).
  \end{align*}
 After differentiating the expansion of
  $\dot{\gn}$ in the basis $\{\gt_1,\dots,\gt_{N-1},\gn\}$ these identities imply, 
  \begin{align*}  
    \frac{d\gn}{dt}
    &= g(\dot{\gn},\gn)\gn + \sum_{i=1}^{N-1}g(\dot{\gn},\gt_i)\gt_i =-\frac{1}{2}\gdot(\gn,\gn)\gn -\sum_{i=1}^{N-1}\gdot(\gn,\gt_i)\gt_i
 = \frac{1}{2}\gdot(\gn,\gn)\gn - \Ldotg\gn,
  \end{align*}
 where we have used~\eqref{eq:Ldefn} in the last step. This proves~\eqref{eq:dgn/dt}. Using it,
        \eqref{eq:dt-g(n,X)} readily follows by applying the
 product rule and simplifying {in view of \eqref{eq:Ldefn}.}
\end{proof}

Let $T\in\T$ and $F\in\triangle_{-1}T$ be a facet of $T$.  Recall that
per the notation in~{\eqref{eq:restriction}},  for  facet tangent vectors $X, Y, Z \in \Xm F$, 
\begin{subequations}
  \label{eq:nabla-sigmaF-def}
  \begin{gather}
    \nablans{\sigma}{\Fr\gn \Fr}(X, Y)
 := (\nabla \sigma)( X, \gn, Y),
    \qquad 
    \nablans{\sigma}{\gn \Fr\Fr} (X, Y) := (\nabla \sigma)(\gn, X, Y),
    \\
 (\nabla{\sigma})_{\Fr}(X, Y, Z) :=
 (\nabla \sigma)(X, Y, Z), \qquad \sigma_{\Fr}(X, Y)
 := \sigma(X, Y).
  \end{gather}
\end{subequations}
We also use
the interior product $\lrcorner$ (see~\eqref{eq:contraction}) and a triple product notation {taking three $(2,0)$-tensors and returning a scalar}
\begin{align}
  \label{eq:triple-prod}
  \begin{split}
    \rho : \sigma : \eta = \rho_{ij} {\sigma_{ab}g^{ak}g^{bj} \eta_{kc}g^{ci}} \text{ for } \rho, \sigma, \eta \in \TT_0^2 (\T),\\
    {\rho : \sigma : \eta = \rho_{ij} \sigma_{ab}g_{\Fr}^{ak}g_{\Fr}^{bj} \eta_{kc}g_{\Fr}^{ci} \text{ for } \rho, \sigma, \eta \in \TT_0^2 (\F),}
  \end{split}
\end{align}
with {appropriate metric-based index raising}. Like in Lemma~\ref{lem:restr_proj_relation},  the relation $\rho_{\Fr} : \sigma_{\Fr} : \eta_{\Fr}=\rho_{\Fp} : \sigma_{\Fp} : \eta_{\Fp}$ holds for $\rho, \sigma, \eta \in \TT_0^2 (\T)$ with single-valued tangential components. In particular, the second fundamental form $\sff^\gn \in \TT_0^2(\F)$ is naturally defined on the facet~$F$; it extends to a symmetric $(2,0)$-tensor on $T$ with vanishing normal components via the tangential projection, and by Lemma~\ref{lem:restr_proj_relation} this extension is consistent with the second line of~\eqref{eq:triple-prod}. The following lemma is needed for the main result (Theorem~\ref{thm:distr_Riem_evol}) of this section.

\begin{lemma}
  \label{lem:triple-prod}
 Let  $T\in\T$,  $F\in\triangle_{-1}T$ be a facet of $T$, $A \in \Atest$,
 and $\sigma \in \TT_0^2(T)$.
 Then
  \[
    \ip{ (\nabla (\gn \lrcorner \sigma))_{\Fr}, \Atnnt}
 = \ip{\nablans\sigma{\Fr\gn \Fr}, \Atnnt} - \sff^\gn : \sigma_{\Fr} : \Atnnt.
  \]
\end{lemma}
\begin{proof}
 Let $X, Y \in \Xm F$. It is easy to see that
  $g( \nabla_X\gn, \gn) = 0$. Hence, we may expand $\nabla_X\gn$ in a
 tangential $g$-orthonormal frame $\{\gt_i\}_{i=1}^{N-1}$ of $\Xm{F}$
  {writing} $\nabla_X\gn = \sum_{i=1}^{N-1}g(\nabla_X\gn, \gt_i) \gt_i$. Using it, 
  \begin{align*}
    \nablans\sigma{\Fr \gn \Fr} (X, Y)
    & = (\nabla_X \sigma) (\gn, Y)
 = \nabla_X( \sigma(\gn, Y)) 
      - \sigma( \gn, \nabla_XY) - \sigma(\nabla_X \gn, Y)
    \\
    & = \nabla_X( \sigma(\gn, Y)) 
      - \sigma( \gn, \nabla_XY)
      - \sum_{i=1}^{N-1}g( \nabla_X \gn, \gt_i)\sigma(\gt_i,Y)
    \\
    & =
 (\nabla_X (\gn \lrcorner \sigma))(Y)
      + \sum_{i=1}^{N-1}\sff^\gn(X,\gt_i)\sigma(\gt_i,Y).
  \end{align*}
 Hence,
  \[
    \ip{\nablans\sigma{\Fr \gn \Fr} , \Atnnt}
 =
    \ip{ (\nabla (\gn \lrcorner \sigma))_{\Fr}, \Atnnt}
    + \sum_{i,j,k=1}^{N-1}
    \sff^\gn(\gt_i,\gt_k)\sigma(\gt_k,\gt_j)A(\gt_i,\gn,\gn,\gt_j)
  \]
 and the lemma follows since
  \begin{equation}
    \label{eq:17Ainnj}
 A(\gt_i,\gn,\gn,\gt_j) = A(\gt_j,\gn,\gn,\gt_i)
  \end{equation}
 by the symmetries of $A$ in \eqref{eq:symm-Riem}.
\end{proof}

\begin{lemma}
  \label{lem:evolution_bnd}
 Let $T\in\T$ and $F\in\triangle_{-1}T$ be a facet of $T$. Then, for any 
  $A\in\Atesto$
  \begin{align*}
      \frac{d}{dt}
      &\Big( \langle\sff^{\gn},{\Atnnt}\rangle\,\vo{F} \Big)
 = \frac{1}{2}\big\langle (\gdot_{\gn\gn}-\tro{\gdot_{\Fr}})
        \sff^{\gn}
        +
        2\nablans\gdot{\Fr \gn \Fr}
        -
        \nablans\gdot{\gn \Fr\Fr}, \; 
 {\Atnnt}\big\rangle\,\vo{F}.
  \end{align*}
\end{lemma}
\begin{proof}
 Since  Lemma~\ref{lem:evolution_scalar_product} gives
  \begin{align}
    \label{eq:dt-ip-II-A}
    \frac{d}{dt}\langle\sff^{\gn},{\Atnnt}\rangle
    &= \ip{\dot{\sff}^{\gn},{\Atnnt}}
      +
      \ip{\sff^{\gn},\frac{d}{dt}{\Atnnt}}
      -\sum_{\ell=1}^2\langle \LdotgF^{(\ell)}\sff^{\gn},{\Atnnt}\rangle,
  \end{align}
 we proceed to compute and simplify the first two terms on the
 right. Let $X, Y \in \Xm F $ be time-independent vector
 fields. To compute the time derivative of the second fundamental
 form, we start by differentiating \eqref{eq:sff-def}
  \begin{align*}
    \frac{d}{dt}\sff^{\gn}(X,Y)
    &= \frac{d}{dt}g(\gn,\nabla_XY) = \frac 1 2
      \gdot_{\gn \gn} g(\gn , \nabla_XY) + g\left( \gn, \frac{d}{dt} \nabla_XY\right) \\
    &= \frac{1}{2}\gdot_{\gn \gn}\sff^{\gn}(X,Y) + \frac{1}{2}\Big[(\nabla_X\gdot)(\gn,Y)+(\nabla_Y\gdot)(\gn,X)
      - (\nabla_{\gn}\gdot)(X,Y)\Big],
  \end{align*}
 where we {first} used~\eqref{eq:dt-g(n,X)} of Lemma~\ref{lem:covdot}{, and then}
   \eqref{eq:linearedKoszul}.  Noting that
  $\Atnnt(X, Y)$ and
  $(\nabla_X\gdot)(\gn,Y)+(\nabla_Y\gdot)(\gn,X)= \nablans\gdot{\Fr\gn
    \Fr}(X,Y)+\nablans\gdot{\Fr\gn \Fr}(Y,X)$ are both symmetric in $X,Y$, we obtain
  \begin{align}
    \label{eq:evolution_sff}
 \left\langle\dot{\sff}^{\gn},{\Atnnt}\right\rangle
    &= \frac{1}{2}\gdot_{\gn \gn}\langle\sff^{\gn},{\Atnnt}\rangle
      +
      \ip{\nablans\gdot{\Fr \gn \Fr}
      -\frac{1}{2}\nablans\gdot{\gn \Fr\Fr},{\Atnnt}}.
  \end{align}
 Next, note that  for any time-independent $X, Y \in \Xm F$, 
  \begin{align*}
    \frac{d}{dt}
    & {\Atnnt}(X, Y)
 = \dotAtnnt(X, Y) +
 A(X,\dot\gn,\gn,Y)+ A(X,\gn,\dot\gn,Y)
    \\
    &=
      \dotAtnnt(X, Y) 
      +\gdot_{\gn \gn}{\Atnnt}(X, Y)
      -A(X,\Ldotg\gn,\gn,Y)-A(X,\gn,\Ldotg\gn,Y)\\
    &=(\gdot_{\gn \gn}-\tro\gdot)\,{\Atnnt}(X, Y)
      +A(\Ldotg X,\gn,\gn,Y)+A(X,\gn,\gn,\Ldotg Y),
  \end{align*}
 where the second and third equalities followed from
 Lemmas~\ref{lem:covdot} and \ref{lem:evolution_test_function}, respectively.
 Hence, 
  \begin{align*}
 \Big\langle\sff^{\gn}
 ,\frac{d}{dt}{\Atnnt}\Big\rangle
      &=\langle\sff^{\gn},(\gdot_{\gn \gn}-\tro\gdot)\,{\Atnnt}
    + \Ldotg^{(1)}\Atnnt + \Ldotg^{(4)}\Atnnt \rangle\\
    &{=\langle\sff^{\gn},(\gdot_{\gn \gn}-\tro\gdot)\,{\Atnnt}
    + 2\LdotgF^{(1)}\Atnnt \rangle,} 
  \end{align*}
 using the symmetry of $\sff^{\gn}$, \eqref{eq:Lsig-1}, {and that by \eqref{eq:Ldefn} $Q(\Ldotg^{(1)}X)=\LdotgF^{(1)}X$ for $X\in\Xm{F}$.}
  
 Finally, we {collect the above identity with \eqref{eq:evolution_sff} to express \eqref{eq:dt-ip-II-A}, using again \eqref{eq:Lsig-symm}}
 to obtain  
  \begin{align*}
    \frac{d}{dt}\langle\sff^{\gn},{\Atnnt}\rangle
    &= \frac{1}{2}\gdot_{\gn \gn}\langle\sff^{\gn},{\Atnnt}\rangle
      +
      \ip{\nablans\gdot{\Fr \gn \Fr}
      -\frac{1}{2}\nablans\gdot{\gn \Fr\Fr},{\Atnnt}}
     \\
     &\quad+
      \ip{\sff^{\gn},(\gdot_{\gn \gn}-\tro{\gdot})\,{\Atnnt}}.
  \end{align*}
 Since $\gdot_{\gn \gn}-\tro{\gdot} = -\tr{\gdot_{\Fr}}$, the lemma now
 follows from the facet version of \eqref{eq:derivative-vol}, namely
  $\dot{\volform}_{F}=\frac{1}{2}\tr{\gdot_{\Fr}}\,\vo{F}$.
\end{proof}

\subsection{Evolution of angle deficit}
\label{ssec:evol-angle-defic}

Finally, we consider the time evolution of the last term (of
codimension~2) of \eqref{eq:distr_Riemann}. Given $E \in \Eint$, let
$\F_E = \{ F \in \F: E \in \triangle_{-1}F\}$ denote the collection of
facets sharing $E$. Each $F \in \F_E$ is shared by two mesh elements
$T_\pm \in \T$.  Recall that the previously defined $g$-conormal
$\gcn_E^{F}$ and $g$-normals $\gn_F^{T_\pm}$ (see
Subsections~\ref{ssec:jump-II}--\ref{ssec:angle-deficit} and
Figure~\ref{fig:vectors}), are such that $\gn^{T_{\pm}}_F$ points
inward from $F$ with respect to $T_{\pm}$ and $\gcn^{F}_E$ points
outward from $E$ into $F$.  Using them, and letting
$\sigma_\pm = \sigma|_{T_\pm}$, we define, for any
$\sigma \in \TT_0^2(\T)$, a jump-like function on $E$ {in $F$} by
\begin{equation}
  \label{eq:def_n_cn_jump}  
 {\jump{ \sigma_{\gcn\gn}}}_F^E =
  \sigma_+(\gn_F^{T_+},\gcn^{F}_E)+\sigma_-(\gn_F^{T_-},\gcn^{F}_E).
\end{equation}
{Recall that $\sigma_E$ denotes the restriction on tangential vector fields of $E$, i.e.,  for all $X, Y \in \Xm E$, $\sigma_E(X, Y) = \sigma(X, Y)$.}
\begin{lemma}
  \label{lem:evolution_bbnd}
 There holds for each $E\in\Eint$ and $A\in\Atesto$
  \begin{align*}
          \frac{d}{dt}\big(\Theta_EA_{\gcn\gn\gn\gcn}
          &\,\vo{E}\big) = -\frac{1}{2}\left(\sum_{F \in \F_E}
            \jmp{\gdot_{\gn\gcn}}^E_F+\tr{\gdot_E}\Theta_E\right)A_{\gcn\gn\gn\gcn}\,\vo{E}.
  \end{align*}
\end{lemma}
\begin{proof}
 Denote by $\{\gt_j\}_{j=1}^{N-2}$ the $g$-orthonormal frame of $\Xm{E}$.
 As in the proof of Lemma~\ref{lem:covdot}, the time derivatives of $\gt_j$ and $\gcn$ lie
 in their respective tangent spaces, i.e., $\dot{\gt}_i \in \Xm{E},$
  $\dot{\gcn}\in\Xm{F},$ and the identities
  \begin{align*}
    &0=\frac{d}{dt}g(\gcn,\gt_i) =\gdot(\gcn,\gt_i) + g(\dot{\gcn},\gt_i),\\
    &0=\frac{d}{dt}g(\gcn,\gcn) = \gdot(\gcn,\gcn) + 2g(\dot{\gcn},\gcn),\\
    &0=\frac{d}{dt}g(\gcn,\gn) =\gdot(\gcn,\gn) + g(\gcn,\dot{\gn}),
  \end{align*}
 imply that
  \begin{align}
    \label{eq:dot-gcn}
    \dot{\gcn}
 =\frac{1}{2} \gdot(\gcn,\gcn)
    \gcn-\Ldotg\gcn.
  \end{align}
 Hence, {in view of the above expression,} by Lemma~\ref{lem:evolution_test_function}, Lemma~\ref{lem:covdot}, and the symmetries \eqref{eq:symm-Riem} of $A$, 
\begin{align*}
  \frac{d}{dt}A_{\gcn\gn\gn\gcn}
  &= \dot{A}_{\gcn\gn\gn\gcn}+2A(\dot{\gcn},\gn,\gn,\gcn)+2A(\gcn,\dot{\gn},\gn,\gcn)\\
  &=-\tr{\gdot}A_{\gcn\gn\gn\gcn}+2A(\Ldotg\gcn,\gn,\gn,\gcn)+2A(\gcn,\Ldotg\gn,\gn,\gcn) \\
  &\quad +A(\gdot(\gcn,\gcn)\gcn-2\Ldotg\gcn,\gn,\gn,\gcn)
    + A(\gcn,\gdot(\gn,\gn)\gn-2\Ldotg\gn,\gn,\gcn)\\
  &=\big(\gdot(\gn,\gn)+\gdot(\gcn,\gcn)-\tr{\gdot}\big)A_{\gcn\gn\gn\gcn} = -\tr{\gdot_E}A_{\gcn\gn\gn\gcn}.
\end{align*}
{For the angle deficit \eqref{eq:angle-deficit}, we follow the idea of \cite[Lemma 3.2]{GN2023} by first computing the variation of the angle \eqref{eq:interior_angle}
\begin{align*}
  \cos\phi:=\cos\sphericalangle_E^T=g|_T(\gcn^{F_+}_E,\gcn^{F_-}_E).
\end{align*} 
By \eqref{eq:dot-gcn}, abbreviating $g=g|_T$, $\gn_{\pm}=\gn^T_{F_\pm}$, $\gcn_{\pm}=\gcn^{F_+}_E$, 
\begin{align*}
  -\dot{\phi}\sin\phi \!&= \gdot(\gcn_+,\gcn_-) \!+\! g(\frac{1}{2} \gdot(\gcn_+,\gcn_+)
    \gcn_+\!-\!\Ldotg\gcn_+,\gcn_-)\!+\!g(\gcn_+,\frac{1}{2} \gdot(\gcn_-,\gcn_-)
    \gcn_-\!-\!\Ldotg\gcn_-)\\
    &=  \frac{1}{2} \gdot(\gcn_+,\gcn_+)g(\gcn_+,\gcn_-)+\frac{1}{2} \gdot(\gcn_-,\gcn_-)g(\gcn_+,
    \gcn_-) -\gdot(\gcn_+,\gcn_-).
\end{align*}
With $g(\gcn_+,\gcn_-) = \cos\phi$, splitting the last term into two, and using \eqref{eq:angle-rotation}, we obtain
\begin{align*}
  -\dot{\phi}\sin\phi &=\frac{1}{2} \gdot(\gcn_+,\gcn_+\cos\phi-\gcn_-)+\frac{1}{2} \gdot(\gcn_-\cos\phi-\gcn_+,\gcn_-)\\
  &=\frac{1}{2} \gdot(\gcn_+,-\sin\phi\gn_+)+\frac{1}{2} \gdot(\sin\phi\gn_-,\gcn_-).
\end{align*}
Thus, $\dot{\phi} = \frac{1}{2} \left(\gdot(\gn_+,\gcn_+)-\gdot(\gn_-,\gcn_-)\right)$ and rearranging the sum from \eqref{eq:angle-deficit} yields
\begin{align*}
\frac{d}{dt}\Theta_E = -\frac{1}{2}\sum_{F \in \F_E}\jmp{\gdot_{\gn\gcn}}^E_F.
\end{align*}}
We also have the analogue of~\eqref{eq:derivative-vol} for $E$, namely
$\dot{\volform}_{E}=\frac{1}{2}\tr{\gdot_E}\,\vo{E}$. Using each of
these identities, the lemma follows after applying the Leibniz rule to
differentiate $\Theta_EA_{\gcn\gn\gn\gcn} \vo{E}$.
\end{proof}

Putting the lemmas  together, we obtain the main result of this
section, Theorem~\ref{thm:distr_Riem_evol}, below. There, the time derivative
of the previously defined generalized densitized Riemann curvature is
expressed  in terms of two forms $a(g; \sigma, U)$ and
$b (g; \sigma, U)$, both linear in $\sigma = \gdot$ and the
metric-independent test tensor field $U$, but nonlinear in the metric
$g$.  Note that $a(g; \sigma, U)$ does not contain any spatial
derivatives or jumps of $\sigma$, but $b(g; \sigma, U)$ {collects them all}.

\begin{theorem}
  \label{thm:distr_Riem_evol}
 Let $\sigma=\dot{g}(t),$ $U \in \Utesto$ and $A = \mapUA U \in\Atesto$. Then
  \begin{align}
          \frac{d}{dt}
          \RogU(U) = 
 a(g;\sigma,U)+b(g;\sigma,U),\label{eq:evolution_distr_Riemann}
  \end{align}
where with $\mathbb{S}_{\Fr}(\rho)=\rho_{\Fr}-\tr{\rho_{\Fr}}g_{\Fr}$ for all $\rho\in\Sc(\T)${, the space of symmetric $\rho\in\TT_0^2(\T)$}, 
\begin{align}
 a(g;\sigma,U)
  &:=\sum_{T\in\T}\int_T\Big(
    \langle \Lsigma^{(1)}\Riemann,\, A \rangle
    -\frac{1}{2}\tr{\sigma}\,\langle\Riemann,\, A
    \rangle\Big)\vo{T}\nonumber\\
  &+2\sum_{F \in \Fint}\int_F
    \jmp{\sff}:\mathbb{S}_{\Fr}(\sigma): \Atnnt\, \vo{F}\nonumber
  \\
  &- 2\sum_{E \in \Eint}\int_E\tr{\sigma_E}\,\Theta_E
 A_{\gcn\gn\gn\gcn}\,\vo{E},\label{eq:def_ah}
  \\
 b(g;\sigma,U)
  &:= 2\sum_{T\in\T}\int_T\langle \nabla^2\sigma,\;
 SA\rangle\;\vo{T}\nonumber \\
  &+2\sum_{F \in \Fint}\int_F
    \langle\jmp{
    \sigma_{\gn\gn}\sff +
    \nablans\sigma{\Fr\gn \Fr}+
 (\nabla (\gn \lrcorner \sigma))_{\Fr}
    -\nablans\sigma{\gn \Fr\Fr}},
    \; \Atnnt\rangle\,\vo{F}\nonumber\\
  &- 2\sum_{E \in \Eint}
    \int_E\sum_{F \in \F_E}
    \jmp{\sigma_{\gcn\gn}}^E_F\, A_{\gcn\gn\gn\gcn}\,\vo{E}.  \label{eq:def_bh}
\end{align}
\end{theorem}
\begin{proof}
 We differentiate each term of the identity of
 Theorem~\ref{thm:distr_Riemann_U} and apply
 Lemmas~\ref{lem:evolution_vol}, \ref{lem:evolution_bnd}, and
  \ref{lem:evolution_bbnd} for the terms on the right of
 codimension 0, 1, and 2, respectively. Then we apply
 Lemma~\ref{lem:triple-prod}{, grouping the simplices $T_+$ and $T_-$ having common facet $F\in\triangle_{-1}(T_+)\cap \triangle_{-1}(T_-)$, and rewriting} the codimension 1 terms.
 Next, we set $\sigma = \gdot(t)$. Collecting all terms with spatial
 derivatives of $\sigma$ and jumps of $\sigma$, we see that their sum
 equals $b$, while the remainder equals $a$.
\end{proof}

At this point, the grouping of terms into $a(\dots)$ and $b(\dots)$ may appear ad hoc. However,  as we will show in the next section, the form $b(\dots)$ in \eqref{eq:def_bh}  contains (up to a factor $-2$)  a distributional version of a generalized covariant incompatibility operator of $\sigma$ (see Theorem~\ref{thm:distr_inc_curved}).  In our numerical analysis of \S\ref{sec:num_ana} we will estimate $a(g;\cdot,\cdot)$ and $b(g;\cdot,\cdot)$ independently. In two dimensions, $a(\dots)$ vanishes, as is shown in \S\ref{sec:spec_2d}.

\section{Distributional incompatibility operator and its adjoint}
\label{sec:distr_inc}
The incompatibility operator acting on a 2-tensor $\sigma$
is well studied in two and three-dimensional Euclidean
domains. Motivated by the fact that in two dimensions it arises
from the linearization of the Gauss curvature, we
propose a generalization of the incompatibility operator (in \eqref{eq:1} below)
to higher dimensions by examining the 
linearization of the Riemann curvature tensor. Taking this further
(in Theorem~\ref{thm:distr_inc_curved} and Definition~\ref{def:distr-incomp-oper}) we define 
a distributional generalization of incompatibility when $\sigma$ is not
smooth and belongs only to $\RR(\T)$.
Finally, defining an adjoint of this generalized incompatibility
operator (in Definition~\ref{def:div-div-dist}), 
the key result  (Theorem~\ref{thm:distr_covariant_adjoint_inc})
of this section is presented.

To define the generalized incompatibility of a smooth $\sigma$, first 
recall that the result of applying the projector $P$ in
\eqref{eq:proj-using-S} on any covariant {$A\in \TT_0^4(\T)$} is
\begin{align}
  \begin{split}
    \label{eq:projector}
    (PA)(X,Y,Z,W)=\frac{1}{4}\Big[&A(X,Y,Z,W)-A(Z,Y,X,W)
    \\
    + &  A(Z,W,X,Y) -A(X,W,Z,Y)  \Big].
  \end{split}
\end{align}
Since $P = P_{13} \circ P_{24} = P_{24} \circ P_{13}$, the result $B = PA$ after an
application of $P$ 
is skew-symmetric in its first and third arguments as well
as its second and fourth arguments. It is also immediate
from~\eqref{eq:projector} that $B(X,Y,Z,W)=B(Y,X,W,Z),$ i.e., $B = PA$ satisfies
\begin{align}
  \label{eq:sym_B}
 B(X,Y,Z,W)=-B(Z,Y,X,W)=-B(X,W,Z,Y)=B(Y,X,W,Z).
\end{align}
The same symmetries can also be obtained from $\Atest$ using the
operator $S$ in \eqref{eq:defn-S}. 
{For a more compact notation, we 
define the $B$ test function space as follows. Letting} 
\begin{align*}
  \Btest=\{ SA\,:\, A\in\Atest\}, \qquad\Btesto=\{ SA\,:\, A\in\Atesto\},
\end{align*}
we see  that any $B\in\Btest$ satisfies~\eqref{eq:sym_B}.  Note that tensors
in $\Btest$ also have the additional continuity condition
\eqref{eq:cont_A}, i.e., 
\begin{align}
  \label{eq:B_nnFF}
  \Btnnt
\end{align} is single-valued on interior facets
$F\in\Fint$. {Further, by the skew symmetry \eqref{eq:sym_B} of $B$ there holds $B_{\Fp\gn\gn\Fp}=B_{\cdot\gn\gn\cdot}$.}  Let us denote by $\D \Atest$ and $\D\Btest$ the
subspaces of $\Atest \cap \TT_0^4(\om)$ and $\Btest \cap \TT_0^4(\om)$
consisting of (globally smooth) compactly supported tensor fields.

From Lemma~\ref{lem:evolution_Riemann}, we know that the linearization
of the Riemann curvature tensor {comprises in its expression} a scalar multiple of
$S P \nabla^2 \sigma$ with $\sigma =\gdot$ when $g$ is smooth.
We relate incompatibility to this linearization, motivated
by the known relationship between incompatibility and curvature linearization in two dimensions~\cite{GNSW2023}. Namely, {given a fixed globally smooth Riemannian 
metric over $\om$,} we define the covariant incompatibility {tensor of any} smooth symmetric
$\sigma\in \Sc(\Omega)$ by
\begin{equation}
  \label{eq:1}
  \INC \sigma := -S P \nabla^2 \sigma,
\end{equation}
{where $\nabla = \nabla^g$ and later $\volform=\volform_g$ and the divergence operator is meant in the $g$-sense.} Then, for any $A \in \Atest$,
letting $B = SA$, the $4$-tensor $\INC\sigma$ satisfies
\begin{equation}
  \label{eq:6}
  \ip{ \INC\sigma, A} = -\ip{SP \nabla^2 \sigma, A}
 = -\ip{ \nabla^2 \sigma, SA}= -\ip{ \nabla^2 \sigma, B}  
\end{equation}
since, at each point on the manifold, $B = SA $ is in the range of
$P$, and $S$ is self-adjoint.  Integrating \eqref{eq:6} using the
volume form $\og$ of the smooth metric, we find that 
\begin{equation}
  \label{eq:7}
  \int_\om
  \ip{ \INC\sigma, \Phi} \,\og 
 = -\int_\om \ip{ \sigma, \div \div S \Phi} \,\og,
  \qquad \Phi \in \D \Atest.
\end{equation}
Here we have used integration by parts (see \eqref{eq:ibp_volume}) twice, recalling that $\Phi$ has compact support.

Next, suppose $\sigma \in \Regge(\T)$. Then $\nabla^2 \sigma$ is, in general,  not
definable as a classical derivative, and definition~\eqref{eq:1} needs extension. For this, we first define a linear functional
$\widetilde{\nabla^2\sigma}$ on $\D \Btest$ by
\begin{equation}
  \label{eq:8}
  \widetilde{\nabla^2\sigma} (\Psi)
 :=
  \int_\om \ip{ \sigma, \div \div \Psi} \,\og,
  \qquad \Psi \in \D \Btest, \; \sigma \in \Regge(\T).
\end{equation}
This generalizes the second covariant derivative's action on smooth
tensors possessing the symmetries of $\Btest$.  Using it, we extend
$\INC \sigma$, taking motivation from \eqref{eq:6}, as a linear
functional $\widetilde{\INC \sigma}$ acting on $\D \Atest$, as follows:
\begin{equation}
  \label{eq:9}
  \widetilde{\INC \sigma}(\Phi) : = -\widetilde{\nabla^2\sigma} (S\Phi)
 = {-}\int_\om \ip{ \sigma, \div \div S \Phi}\,\og, 
  \qquad \Phi \in \D  \Atest, \; \sigma \in \Regge(\T).
\end{equation}
By~\eqref{eq:7}, we see that this is indeed an extension of the smooth case~\eqref{eq:6}. The next result
characterizes this generalized incompatibility
in terms of  integrals over mesh
components, including the smooth incompatibility computed element-by-element. 

{We now show that the just introduced tensor $\INC$ is directly related to the term $b(g;\sigma,U)$ computed in Theorem~\ref{thm:distr_Riem_evol}, indeed one has:}
  
\begin{theorem}
  \label{thm:distr_inc_curved}
  {Let $g\in\Sc^+(\Omega)$ be a given smooth Riemannian metric on the domain $\om$. For any $\sigma\in\Regge(\T)$ and $\Phi \in \D \Atest$ one has}
  \begin{align}
    \widetilde{\INC \sigma}(\Phi)
    &=\sum_{T\in\T}
 \bigg(\int_T\langle\INC \sigma,\Phi\rangle\,\vo{T}
      \nonumber\\
      &\qquad\qquad-
      \int_{\d T}
      \langle
      \sigma_{\gn\gn}\sff^{\gn} + 
      \nablans\sigma{\Fr \gn \Fr} + (\nabla (\gn \lrcorner \sigma))|_{\Fr} -
      \nablans\sigma{\gn \Fr\Fr}
 ,
      \Phi_{\Fr\gn\gn \Fr}\rangle\,\vo{\d T}\bigg)\nonumber
    \\
    &\quad-\sum_{E\in\Eint}
      \sum_{F \in \F_E}
      \int_E
      \jmp{\sigma_{\gcn\gn}}^E_F\Phi_{\gcn\gn\gn\gcn}\,\vo{E},
      \label{eq:distr_cov_inc}
  \end{align}
  {where $\nabla = \nabla^g$, $\volform=\volform_g$, and later the divergence operator is meant in the $g$-sense. Further,} in the integrals over $\d T$, the $g$-normal $\gn$ points
 into the element $T$, and we have used the notation in
  \eqref{eq:nabla-sigmaF-def} and \eqref{eq:def_n_cn_jump}.
 Hence,
 for any smooth compactly supported 
  $U$ in $\Utesto$, letting $\Phi = \mapUA U \in \D \Atesto$, the
 form $b(\dots)$ in \eqref{eq:def_bh} satisfies
  \begin{equation}
    \label{eq:rel_b_inc_smooth}
 b(g;\sigma,U)=-2\,\widetilde{\INC \sigma}( \Phi).
  \end{equation}
\end{theorem}

We proceed to prove \eqref{eq:distr_cov_inc} of
Theorem~\ref{thm:distr_inc_curved} using the next two
lemmas. Identity~\eqref{eq:rel_b_inc_smooth} immediately follows
from~\eqref{eq:distr_cov_inc}, {by grouping the codimension 1 terms along facets $F$ common to simplices $T_+$ and $T_-$ in $\T$,} but is stated to connect to the
developments in the previous section by showing the relationship
between the $b(\dots)$ in Theorem~\ref{thm:distr_Riem_evol} and the
just-defined distributional incompatibility operator of \eqref{eq:9}.

\begin{lemma}
  \label{lem:distr_inc_curved_part1}
 In the setting of Theorem~\ref{thm:distr_inc_curved}, letting $\Psi = S\Phi\in \D \Btest$, 
  \begin{align}
    \nonumber
    \widetilde{\nabla^2\sigma}(\Psi)
    & = \sum_{T\in\T}\bigg(
      -\int_T \langle\nabla\sigma,\div\Psi\rangle\,\vo{T}
    \\  \nonumber
    & \hspace{1.cm} + \int_{\d T}
 \Big(\langle\sigma_{\gn\gn}\sff^{\gn}
      +
 (\nabla(\gn \lrcorner  \sigma))_{\Fr},
      \;
      \Psitnnt\rangle
      -\langle\sff^{\gn} \otimes\sigma_{\gn \Fr},
      \Psi_{\Fr\Fr\gn\Fr}\rangle\Big)\,
      \vo{\d T}\bigg)
    \\
    & \quad+\sum_{F\in\Fint}\int_{\d F}
      \langle\jmp{ \sigma_{\gn \Fr}},\,
      \Psi_{\gcn\gn\gn \Fr}\rangle\,\vo{\d F},
      \label{eq:distr_inc_curved_part1}    
  \end{align}
 where, in the last term, the $g$-conormal vector $\gcn$ on $\d F$
 points into the facet $F$.
\end{lemma}
\begin{proof}
 Starting from \eqref{eq:8} and integrating by parts
 elementwise using \eqref{eq:ibp_volume},
  \begin{align}
    \widetilde{\nabla^2\sigma}(\Psi)
    &= \sum_{T\in\T}\int_T \langle\sigma, \div\div \Psi\rangle\,\vo{T}\nonumber
    \\
    &=\sum_{T\in\T}
 \left( -\int_T \langle\nabla\sigma,\div\Psi\rangle\,\vo{T}      
      - \int_{\d T}
      \ip{ \gn^\flat \otimes \sigma, \div\Psi}\, \vo{\d T} \right).
      \label{eq:intermediate1'}
  \end{align}
 Then using the notation in \eqref{eq:nFnF-notation} to
 split the last integrand using normal  and tangential components {to facets $F\subset \d T$},
  \begin{equation}
    \label{eq:10}
  \begin{aligned}
    \ip{\gn^\flat \otimes \sigma, \,\div\Psi}
    & = \ip{\sigma_{\gn\gn}, (\div\Psi)_{\gn \gn \gn}}
      +
      \langle{\sigma_{\Fp\gn }},(\div\Psi)_{\gn   \Fp \gn }\rangle
    \\
    &
      +\langle{\sigma_{\gn\Fp }},(\div\Psi)_{\gn \gn   \Fp }\rangle
      +\langle{\sigma_{ \Fp }},(\div\Psi)_{\gn  \Fp\Fp }\rangle.
  \end{aligned}    
  \end{equation}
 Since $\Psi$ is skew-symmetric in its second and fourth
 arguments---see \eqref{eq:sym_B}---we know that $\div\Psi$ is
 skew-symmetric in its first and third arguments (as can be seen
 from~\eqref{eq:div-def}), so $(\div\Psi)_{\gn \gn \gn}$ and
  $(\div\Psi)_{\gn \Fp \gn }$ vanish. The $tt$-continuity
 of $\sigma$ {together with \eqref{eq:restr_proj_relation} and the smoothness of $g$} implies that the $\sigma_{\Fp}$ term in~\eqref{eq:10} {coming from $T_+$ and $T_-$ adjacent to $F$ add to zero: $\langle\sigma_{\Fp},(\div\Psi)_{\gn^+\Fp\Fp}\rangle+\langle\sigma_{\Fp},(\div\Psi)_{\gn^-\Fp\Fp}\rangle=0$.}
 Using all these observations (and the jump notation in \eqref{eq:jmp-tensor}), 
 we find that~\eqref{eq:10} implies that at any point of an interior facet, 
  \begin{align*}
    \langle \jmp{\gn^\flat \otimes \sigma}, \div\Psi \rangle
    & = \langle \jmp{\sigma_{\gn \Fp }},(\div\Psi)_{\gn \gn \Fp }\rangle
 = \langle \jmp{\sigma_{\gn \Fp }},(\divFR\Psi)_{\gn \gn \Fp }\rangle.
  \end{align*}
 Here we have used that, by \eqref{eq:divF-def} and \eqref{eq:divF-div_relation}, the surface
 divergence satisfies
  $ (\divFR \Psi )_{\gn \gn \Fp} = (\div \Psi )_{\gn \gn \Fp} - (\nabla
  \Psi)_{\gn \gn \gn \gn \Fp}$ where the last term vanishes due to skew
 symmetries of $\Psi$.  Moreover, since $(\divFR\Psi)_{\gn \gn \gn }$
 vanishes and since 
  \begin{equation}
    \label{eq:141}
    \gn\lrcorner\sigma = \sigma_{{\gn\gn}}{\gn}^\flat + \sigma_{{\gn}  \Fp}, 
  \end{equation}
 we may rewrite
  $
    \langle \jmp{\sigma_{\gn \Fp }},(\divFR\Psi)_{\gn \gn \Fp }\rangle
 = \ip{ \gn^\flat \otimes \gn^\flat \otimes \jmp{ \gn \lrcorner \sigma },
      \divFR \Psi }.
  $
 Accounting for all the above-mentioned cancellations,
  \eqref{eq:intermediate1'} becomes
  \begin{align}
    \widetilde{\nabla^2\sigma}(\Psi)
    &=\sum_{T\in\T}
      \!\left( \!-\int_T \langle\nabla\sigma,\div\Psi\rangle\,\vo{T}      
      - \int_{\d T}
      \ip{ \gn^\flat \otimes \gn^\flat \otimes (\gn \lrcorner \sigma),
      \, \divFR \Psi}\,\vo{\d T}\! \right).
      \label{eq:intermediate1.5}
  \end{align}

 The last term is now in a form suitable for 
 integration by parts 
 on each facet $F$ of $\d T$
 using \eqref{eq:ibp_surface}. Doing so, we get 
  \begin{equation}
    \label{eq:13}
  \begin{aligned}
    -\int_F 
    \langle \gn^\flat &\otimes \gn^\flat
      \otimes (\gn \lrcorner \sigma ),
     \; \divFR\Psi \rangle
    \,\vo{F}
 = \int_F H^{\gn}\,\langle \gn^\flat \otimes \gn^\flat \otimes \gn^\flat \otimes
 (\gn \lrcorner \sigma),\,
      \Psi\rangle\,\vo{F}
    \\
    & +\int_F \langle\nablaF({\gn}^\flat\otimes{\gn}^\flat\otimes
 ({\gn \lrcorner \sigma})),\Psi\rangle\,\vo{F}
      +
      \int_{\d F}\langle\gcn^\flat \otimes \gn^\flat \otimes\gn^\flat\otimes
 ({\gn \lrcorner \sigma}),\,\Psi\rangle\,\vo{\d F},
  \end{aligned}    
  \end{equation}
 where $\gcn$ denotes the conormal vector on $\d F$ pointing into
  $F$, $H^{\gn}$ denotes the mean curvature (see~\eqref{eq:H}), and
  $\nablaF$ is as in \eqref{eq:11}.  By the skew-symmetry of $\Psi$ in
 its first and third arguments, $\Psi(\gn, \gn, \gn, X)=0 \forall X\in\Xm{\T}$,
 so the integral with $H^{\gn}$ vanishes.

 To simplify the {integral showing} $\nablaF$ in~\eqref{eq:13}, we begin by noting
 that \eqref{eq:11} implies $\nablaF A( \gn, \dots) = 0$ for any
 tensor~$A$. Hence, the inner product in the integrand can be
 evaluated using a $g$-orthonormal basis $\{\gt_i\}_{i=1}^{N-1}$ of
 the tangent space $\Xm{F}$, i.e.,
  \begin{align}
    \label{eq:14} 
    &\langle
    \nablaF
 ({\gn}^\flat\otimes{\gn}^\flat\otimes({\gn\lrcorner\sigma })),
      \Psi\rangle =
      \sum_{i=1}^{N-1}
      \ip{  \nablaF_{\gt_i} ({\gn}^\flat\otimes{\gn}^\flat
      \otimes ({\gn\lrcorner\sigma })),
      \gt_i \lrcorner \Psi} 
    \\     \nonumber
 = \sum_{i=1}^{N-1}\langle
 (&\nablaF_{\gt_i}{\gn}^\flat)\otimes{\gn}^\flat
      \otimes ({\gn\lrcorner\sigma })
      +
 {\gn}^\flat\otimes(\nablaF_{\gt_i}{\gn}^\flat)
      \otimes
 ({\gn\lrcorner\sigma })
      +
 {\gn}^\flat\otimes{\gn}^\flat\otimes \nablaF_{\gt_i}
 ({\gn\lrcorner\sigma }),
      \gt_i \lrcorner \Psi\rangle,
  \end{align}
 where we have also used the Leibniz rule. Note that
  $\nablaF\gn^\flat = -\sff^{\gn}$ can be expressed in terms of
  $\gt_i^\flat \otimes \gt_j^\flat$ for all $i,j$ (without $\gn$).
 Also using \eqref{eq:141}, the first term in~\eqref{eq:14} becomes
  \begin{align*}
    \sum_{i=1}^{N-1}&
    \langle
 (\nablaF_{\gt_i}{\gn}^\flat)\otimes{\gn}^\flat\otimes
 ({\gn\lrcorner\sigma}), 
      \gt_i \lrcorner \Psi
      \rangle 
      \\
      &=
      \ip{\nablaF \gn^\flat \otimes \gn^\flat \otimes \sigma_{\gn \Fp}, \, \Psi}
      +\sum_{i=1}^{N-1}
      \sigma_{\gn\gn} \ip{
 (\nablaF_{\gt_i}{\gn}^\flat)
      \otimes
 {\gn}^\flat
      \otimes
       \gn^\flat, 
      \,       \gt_i \lrcorner \Psi
 }
    \\
    & =
      -
      \langle\sff^\gn\otimes{{\gn^\flat}}\otimes {\sigma_{{\gn} \Fp }},\Psi\rangle
      -
      \sum_{j=1}^{N-1}\sum_{i=1}^{N-1}
      \, {\sigma_{{\gn}{\gn}}}\; \sff^\gn(\gt_i, \gt_j)\,
      \Psi(\gt_i,\gt_j,{\gn},{\gn})
    \\
    & =
      -
      \langle\sff^\gn\otimes{{\gn^\flat}}
      \otimes {\sigma_{{\gn} \Fp }},\Psi\rangle
      +
      \,
      \ip{
 {\sigma_{{\gn}{\gn}}}\,\sff^\gn,
      \Psi_{\Fp\gn\gn\Fp}},
  \end{align*}
 since
  $\Psi(\gt_i,\gt_j,{\gn},{\gn}) = -\Psi(\gt_i,{\gn},{\gn}, \gt_j)$.
 Using  the
 skew symmetry of $\psi_{ij}(X) = \Psi(\gt_i, \gn, \gt_j, X)$ for all $X\in\Xm{\T}$ with
 respect to $i, j$ and the symmetry of $\sff^\gn(\gt_i, \gt_j)$ with
 respect to $i, j$, we find that the second term in~\eqref{eq:14} must vanish
  \begin{align*}
    -\sum_{i=1}^{N-1}
    \langle{\gn}^\flat
    \otimes
 (\nablaF_{\gt_i} \gn^\flat)
    \otimes
 ({\gn\lrcorner\sigma }),\gt_i \lrcorner \Psi
    \rangle
    &=
      \sum_{j=1}^{N-1}\sum_{i=1}^{N-1}      
      \langle
      \,\sff^{{\gn}}({\gt_i, \gt_j})\;
 ({\gn\lrcorner\sigma }), \, \psi_{ij}
      \rangle=0.
  \end{align*}
 Thus, \eqref{eq:14} becomes
  \begin{equation}
    \label{eq:15}
  \langle\nablaF
 ({\gn}^\flat\otimes{\gn}^\flat\otimes ({\gn\lrcorner\sigma })),\Psi\rangle
 =\langle {\sigma_{{\gn}{\gn}}\sff^{{\gn}}}
      +
      \nablaF(\gn \lrcorner  \sigma),
      \Psi_{\Fp\gn\gn\Fp}
      \rangle
      -
      \langle\sff^{{\gn}}\otimes{\gn}^\flat\otimes
 {\sigma_{{\gn} \Fp}},\Psi\rangle.    
  \end{equation}

 Finally, using~\eqref{eq:15} in~\eqref{eq:13} and substituting the result
 into~\eqref{eq:intermediate1.5},
  \begin{align*}
    \widetilde{\nabla^2\sigma}&(\Psi)
 =\sum_{T\in\T}
 \bigg[
     -\int_T \langle\nabla\sigma,\div\Psi\rangle\,\vo{T}      
      + \int_{\d T} \Big(\langle {\sigma_{{\gn}{\gn}}\sff^{{\gn}}}
      +
      \nablaF(\gn \lrcorner  \sigma),
      \Psi_{\Fp\gn\gn\Fp}
      \rangle
      \\
    & -
      \langle\sff^{{\gn}}\otimes{\gn}^\flat\otimes
 {\sigma_{{\gn} \Fp}},\Psi\rangle\Big)\,\vo{\d T}
    + \sum_{\{F \in \F: F \subset \d T \}}
      \int_{\d F}\langle\gcn^\flat \otimes \gn^\flat \otimes\gn^\flat\otimes
 ({\gn \lrcorner \sigma}),\,\Psi\rangle\,\vo{\d F} \bigg].
  \end{align*}
  {Since $\Psi=0$ on $\d\om$, the} last integrand, upon use of the decomposition~\eqref{eq:141},
 becomes a sum of two terms,
  $    \langle\gcn^\flat \otimes \gn^\flat \otimes\gn^\flat\otimes
 ({\gn \lrcorner \sigma}),\,\Psi\rangle
 =
  \langle \sigma_{\gn \Fp},\,\Psi_{\gcn \gn \gn \Fp} \rangle
  +       \langle \sigma_{\gn\gn},\,\Psi_{\gcn \gn \gn \gn} \rangle,     
  $
 of which the latter vanishes due to the skew-symmetric properties of
  $\Psi$.
  {\eqref{eq:restr_proj_relation} yields the stated result.}
\end{proof}

\begin{lemma}
  \label{lem:distr_inc_curved_part2}
 In the setting of Theorem~\ref{thm:distr_inc_curved},
 letting $\Psi = S\Phi$, 
  \begin{align*}
          \sum_{F\in\Fint}\int_{\d F}
          \langle\jmp{\sigma_{\gn \Fr}},\Psi_{\gcn\gn\gn \Fr}\rangle\,\vo{\d F} = \sum_{E\in\Eint}\sum_{F \in \F_E}\int_E\jmp{\sigma_{\gn\gcn}}^E_F \Psi_{\gcn\gn\gn\gcn}\,\vo{E}.
  \end{align*}
\end{lemma}
\begin{proof}  
 This proof extends an idea of \cite{Christiansen11} from 3D to
 arbitrary dimension.
 Due to the $tt$-continuity of $\sigma$ and the skew-symmetry of $\Psi$ there holds with the projection $P$ \eqref{eq:proj-using-S}{, in view of \eqref{eq:PSA} together with \eqref{eq:restr_proj_relation},}
  \begin{align*}
    \langle\jmp{\gn \lrcorner  \sigma},{\Psi_{\gcn\gn\gn\cdot}}\rangle &=\langle\jmp{\gn \lrcorner  \sigma},{\Psi_{\gcn\gn\gn\cdot}}\rangle + \langle\jmp{\sigma_{\Fr}},\Psi_{\gcn\gn FF}\rangle + \langle\jmp{\sigma_{\Fr\gn}},\Psi_{\gcn\gn \Fr \gn}\rangle\\
    &=\langle\jmp{\sigma
 },{\Psi_{\gcn\gn\cdot\cdot}}\rangle =\langle P(\gcn^\flat\otimes\gn^\flat\otimes\jmp{\sigma}),\Psi\rangle
  \end{align*}
 and thus,
  \begin{align}
    \sum_{F\in\Fint}\int_{\d F}\langle\jmp{\gn \lrcorner  \sigma},{\Psi_{\gcn\gn\gn\cdot}}\rangle\,\vo{\d F} = \sum_{F\in\Fint}\int_{\d F}\langle P(\gcn^\flat\otimes\gn^\flat\otimes\jmp{\sigma}),\Psi\rangle\,\vo{\d F}.\label{eq:intermediate3}
  \end{align}

	\begin{figure}[ht!]
		\centering
		\begin{tikzpicture}
			\draw[thin] (0,0) to (1.4,1.5);
			\draw[thin] (0,0) to (1.6,-1.3);
			\draw (0,0) circle (3pt);
			
			\draw[-latex,color=red,thick] (1.4/2,1.5/2) to (1.4/2+1.5/3,1.5/2-1.4/3);
			\draw[-latex,color=red,thick] (1.6/2,-1.3/2) to (1.6/2-1.3/3,-1.3/2-1.6/3);
			
			\draw[-latex,color=black,thick] (1.6*0.8,-1.3*0.8) to (1.6*0.8+1.3/3,-1.3*0.8+1.6/3);
			
			\draw[-latex,color=blue,thick] (0,0) to (1.4/3,1.5/3);
			\draw[-latex,color=blue,thick] (0,0) to (1.6/3,-1.3/3);
			
			\node (A) at (-0.5,-0.1) {$E$};
			\node (B) at (2,0) {$T$};
			\node (C) at (0.5,1.) {$F_+$};
			\node (D) at (1.2,-0.6) {$F_-$};
			\node (E) at (1.8,0.7) {${\color{red}\gn_+}=\gn^T_{F_+}$};
			\node (F) at (0.9,-1.2) {{\color{red}$\gn_-$}};
			\node (G) at (0.1,0.6) {{\color{blue}$\gcn_+$}};
			\node (G) at (0.1,-0.6) {{\color{blue}$\gcn_-$}};
			\node (E) at (2.1,-0.8) {$\gn^T_{F_-}$};

			\centerarc[-latex,color=teal, thick](0,0)(-40:50:0.4)
			\node (E) at (0.65,0.) {{\color{teal}$\theta$}};
		\end{tikzpicture}
		\caption{Illustration of $g$-normal and $g$-conormal vectors in the proof of Lemma~\ref{lem:distr_inc_curved_part2}.}
		\label{fig:bbnd_for_proof}
	\end{figure}

 First, we focus on $\gcn^\flat\otimes\gn^\flat\otimes\jmp{\sigma}$ by reordering the sum over codimension-2 boundaries. Let $\{\gt_i\}_{i=1}^{N-2}$ be a $g$-orthonormal basis of $\Xm{E}$. For each element $T$ containing $E$ as the intersection of the two facets $F_+$, $F_- \in\triangle_{-1}T$ we choose the orientation $(\gn_+,\gcn_+)$ and $(\gn_-,\gcn_-)$ such that both build a right-handed orthonormal basis. Without loss of generality we assume that $\gn_+=\gn^T_{F_+}$ and $\gn_-=-\gn^T_{F_-}$, cf. Figure~\ref{fig:bbnd_for_proof} and \cite{Christiansen11}. Then,
  \begin{align*}
    \sum_{F\in\Fint}\gcn^\flat\otimes\gn^\flat\otimes\jmp{\sigma} = \sum_{E\in\Eint}\sum_{T\supset E}\gcn_+^\flat\otimes\gn_+^\flat\otimes\sigma|_T- \gcn_-^\flat\otimes\gn_-^\flat\otimes\sigma|_T.
  \end{align*}
 Let $\theta$ denote the angle to transform $(\gn_-,\gcn_-)$ into $(\gn_+,\gcn_+)$ and $R(\rho)$ the rotation matrix such that $R(0)\gn_-=\gn_-$, $R(0)\gcn_-=\gcn_-$ and $R(\theta)\gn_-=\gn_+$, $R(\theta)\gcn_-=\gcn_+$. Then
  \begin{align*}
    \gcn_+^\flat\otimes\gn_+^\flat\otimes\sigma|_T- \gcn_-^\flat\otimes\gn_-^\flat\otimes\sigma|_T=\int_0^{\theta}\frac{d}{d\rho}\big(R(\rho)\gcn_-^\flat\otimes R(\rho)\gn_-^\flat\otimes\sigma|_T\big)\,d\rho.
  \end{align*}
 Computing {(since $\frac{d}{d\rho}(R(\rho))\gcn^\flat_-=-R(\rho)\gn^\flat_-$ and $\frac{d}{d\rho}(R(\rho))\gn^\flat_-=R(\rho)\gcn^\flat_-$)}
   \begin{align*}
    \frac{d}{d\rho}\!\big(R(\rho)\gcn_-^\flat\otimes R(\rho)\gn_-^\flat\otimes\sigma|_T\!\big) \! =\! -\!R(\rho)\gn_-^\flat\otimes R(\rho)\gn_-^\flat\otimes\sigma|_T\!+\!R(\rho)\gcn_-^\flat\otimes R(\rho)\gcn_-^\flat\otimes\sigma|_T
   \end{align*}
 reveals that the integrand, and thus also the integral, is symmetric in its first two components. Therefore, we obtain the symmetry
   \begin{align}
    \label{eq:codim2_sym}
    \sum_{F\in\Fint}P(\gcn^\flat\otimes\gn^\flat\otimes\jmp{\sigma}) = \sum_{F\in\Fint}P(\gn^\flat\otimes\gcn^\flat\otimes\jmp{\sigma}).
   \end{align}
 Next, we expand $\jmp{\sigma}$ into (co)normal and tangential components and use the skew-symmetry of {elements in $P(\TT_0^4(\T))$ (see \eqref{eq:projector} and below)} and $tt$-continuity of $\sigma$ (only here, we implicitly sum over $i,j$ if applicable)
\begin{align*}
 P(\gcn^\flat\otimes\gn^\flat\otimes\jmp{\sigma})&=\jmp{\sigma(\gt_i,\gt_j)}P(\gcn^\flat\otimes\gn^\flat\otimes\gt^\flat_i\otimes \gt^\flat_j)+\jmp{\sigma(\gn,\gn)}P(\gcn^\flat\otimes\gn^\flat\otimes\gn^\flat\otimes \gn^\flat)\\&+\jmp{\sigma(\gcn,\gcn)}P(\gcn^\flat\otimes\gn^\flat\otimes\gcn^\flat\otimes \gcn^\flat)+\jmp{\sigma(\gcn,\gn)}P(\gcn^\flat\otimes\gn^\flat\otimes\gcn^\flat\otimes \gn^\flat)\\&+\jmp{\sigma(\gn,\gcn)}P(\gcn^\flat\otimes\gn^\flat\otimes\gn^\flat\otimes \gcn^\flat)+\jmp{\sigma(\gcn,\gt_i)}P(\gcn^\flat\otimes\gn^\flat\otimes\gcn^\flat\otimes \gt^\flat_i)\\&+\jmp{\sigma(\gt_i,\gcn)}P(\gcn^\flat\otimes\gn^\flat\otimes\gt^\flat_i\otimes \gcn^\flat)+\jmp{\sigma(\gn,\gt_i)}P(\gcn^\flat\otimes\gn^\flat\otimes\gn^\flat\otimes \gt^\flat_i)\\&+\jmp{\sigma(\gt_i,\gn)}P(\gcn^\flat\otimes\gn^\flat\otimes\gt^\flat_i\otimes \gn^\flat)\\
  &=\jmp{\sigma(\gn,\gcn)}P(\gcn^\flat\otimes\gn^\flat\otimes\gn^\flat\otimes \gcn^\flat)+\jmp{\sigma(\gn,\gt_i)}P(\gcn^\flat\otimes\gn^\flat\otimes\gn^\flat\otimes \gt^\flat_i),
\end{align*}
and analogously
\begin{align*}
 P(\gn^\flat\otimes\gcn^\flat\otimes\jmp{\sigma})&=\jmp{\sigma(\gcn,\gn)}P(\gn^\flat\otimes\gcn^\flat\otimes\gcn^\flat\otimes \gn^\flat)+\jmp{\sigma(\gt_i,\gn)}P(\gn^\flat\otimes\gcn^\flat\otimes\gt^\flat_i\otimes \gn^\flat).
\end{align*}
Due to the proven symmetry \eqref{eq:codim2_sym} we have with the symmetries of $\sigma$ and $P$
\begin{align*}
  0&=\sum_{F\in\Fint}\Big(P(\gcn^\flat\otimes\gn^\flat\otimes\jmp{\sigma})-P(\gn^\flat\otimes\gcn^\flat\otimes\jmp{\sigma})\Big)\\
  &=\sum_{F\in\Fint}\sum_{i=1}^{N-2}\jmp{\sigma(\gn,\gt_i)}\Big(P(\gcn^\flat\otimes\gn^\flat\otimes\gn^\flat\otimes \gt^\flat_i)-P{(\gt^\flat_i\otimes\gn^\flat\otimes\gn^\flat\otimes \gcn^\flat)}\Big)\\
  &=\sum_{F\in\Fint}\sum_{i=1}^{N-2}\jmp{\sigma(\gn,\gt_i)}P\Big(\gcn^\flat\otimes\gn^\flat\otimes\gn^\flat\otimes \gt^\flat_i-\gt^\flat_i\otimes\gn^\flat\otimes\gn^\flat\otimes \gcn^\flat\Big),
\end{align*}
and thus there must {hold}
\begin{align*}
  \sum_{F\in\Fint}\sum_{i=1}^{N-2}\jmp{\sigma(\gn,\gt_i)}P(\gcn^\flat\otimes\gn^\flat\otimes\gn^\flat\otimes \gt^\flat_i)=0.
\end{align*}
Using this identity in \eqref{eq:intermediate3} gives the desired result {(see \eqref{eq:def_n_cn_jump} for the definition of ${\jump{ \sigma_{\gcn\gn}}}_F^E$)}

\begin{align*}
  \sum_{F\in\Fint}\int_{\d F}\langle\jmp{\gn \lrcorner  \sigma},{\Psi_{\gcn\gn\gn\cdot}}\rangle\,\vo{\d F} &= \sum_{F\in\Fint}\int_{\d F}\Big(\langle\jmp{\sigma(\gn,\gcn)}\gcn^\flat\otimes\gn^\flat\otimes\gn^\flat\otimes \gcn^\flat\\
  &\qquad+\sum_{i=1}^{N-2}\jmp{\sigma(\gn,\gt_i)}\gcn^\flat\otimes\gn^\flat\otimes\gn^\flat\otimes \gt^\flat_i,\Psi\rangle\Big)\,\vo{\d F}\\
  &= \sum_{E\in\Eint}\sum_{F\supset E}\int_E\jmp{\sigma(\gn,\gcn)}^E_F \Psi_{\gcn\gn\gn\gcn}\,\vo{E},
\end{align*}
finishing the proof.
\end{proof}

\bigskip

\begin{proof}[{\bf{Proof of Theorem~\ref{thm:distr_inc_curved}}}]
 Let $\Psi = S\Phi\in \D \Btest$.
 We start by integrating by parts the first term of
  \eqref{eq:distr_inc_curved_part1} (in
 Lemma~\ref{lem:distr_inc_curved_part1}) using \eqref{eq:ibp_volume},
  \begin{align}
    \label{eq:161}
    \int_T -\langle\nabla \sigma, & \div\Psi\rangle\,\vo{T} = \int_T\langle\nabla^2\sigma,\Psi\rangle\,\vo{T}
                                    +
                                    \int_{\d T}\langle\gn^\flat\otimes\nabla\sigma,\Psi\rangle\,\vo{\d T}.
  \end{align}
 Splitting the boundary integrand into normal and tangential
 components {to facets $F\subset \d T$} and omitting terms that vanish by the skew-symmetries \eqref{eq:sym_B} of
  $\Psi$,
  \begin{align}
    \label{eq:14-int-parts}
    \ip{\gn^\flat\otimes\nabla\sigma,\Psi}
    & = \ip{ \nablans\sigma{\gn \Fp\Fp}, \Psi_{\gn\gn \Fp\Fp} }
      + \ip{ \nablans\sigma{\Fp\Fp\gn}, \Psi_{\gn \Fp\Fp \gn} }
      + \ip{ (\nabla\sigma)_{\Fp\Fp\Fp}, \Psi_{\gn \Fp\Fp\Fp} }.
  \end{align}
 Since the first two terms on the right are present in the identity
 of the theorem, we focus on the last term. It can be understood by
 splitting $\sigma$ into normal and tangential components {to facets $F\subset \d T$}, 
  \begin{equation}
    \label{eq:sigma-split}
    \sigma = \gn^\flat\otimes \sigma_{\gn \Fp} +
    \sigma_{\Fp\gn} \otimes\gn^\flat+\sigma_{\gn\gn}\gn^\flat\otimes\gn^\flat
    + \sigma_{\Fp}.    
  \end{equation}
 To compute $(\nablaF \sigma)_{\Fp} =  (\nabla \sigma)_{\Fp\Fp\Fp}$
 using this decomposition, we start by applying  
 the Leibniz rule to the first term on the right-hand side. Then, using 
  $\sff^\gn = -\nablaF \gn^\flat$ and~\eqref{eq:141},
  $ \big(\nablaF (\gn^\flat\otimes \sigma_{\gn \Fp})\big)_{\Fp} 
 = \big( \gn^\flat \otimes \nablaF \sigma_{\gn \Fp}\big)_{\Fp}
  -\sff^\gn \otimes  \sigma_{\gn \Fp}
 = -\sff^\gn \otimes \sigma_{\gn \Fp}, 
    $
 so
  \begin{equation*}
    \ip{  \big( \nablaF (\gn^\flat\otimes\sigma_{\gn \Fp}) \big)_{\Fp},
      \Psi_{\gn \Fp\Fp\Fp}}
 =
    \ip{  -\sff^\gn \otimes \sigma_{\gn \Fp}, 
      -\Psi_{\Fp\Fp\gn \Fp}}.
  \end{equation*}
 The derivative of the second term on the right-hand side of
  \eqref{eq:sigma-split} is computed  similarly, but carefully accounting
 for the order of arguments, namely for $X_i \in \Xm F$,
  \begin{align*}
 \big(\nablaF (\sigma_{\Fp\gn} \otimes \gn^\flat)\big)_{\Fp}
 (X_1, X_2, X_3) 
    & =
 \big(\nablaF_{X_1} \big(\sigma_{\Fp\gn} \otimes \gn^\flat)\big)
 (X_2, X_3) 
    \\
    & =
 \big( \nablaF_{X_1} \sigma_{\Fp\gn} \otimes \gn^\flat
      +
      \sigma_{\Fp\gn} \otimes \nablaF_{X_1} \gn^\flat
 \big)
 (X_2, X_3)
    \\
    & = -\sigma_{\gn \Fp} (X_2) \; \sff^\gn (X_1, X_3),
  \end{align*}
 since $\gn^\flat(X_3)=0$.
 As a result, by the symmetry of $\sff^\gn(X_1, X_3)$ and the skew-symmetry
 of $\Psi(\gn, X_1, X_2, X_3)$ in $X_1$ and $X_3$,
  \[    \ip{ \nablaF \big(\sigma_{\Fp\gn} \otimes \gn^\flat\big)_{\Fp},
      \Psi_{\gn \Fp\Fp\Fp} } = 0.
  \]
 The third term in~\eqref{eq:sigma-split} does not contribute since
  $(\nablaF (\sigma_{\gn\gn}\gn^\flat\otimes\gn^\flat))_{\Fp}
 =0$.

 Using these observations to
 simplify the last term of \eqref{eq:14-int-parts}, we obtain
  \[
    \begin{aligned}
      \sum_{T \in \T} \int_{\d T}
      \langle\gn^\flat\otimes\nabla\sigma,\Psi\rangle&\,\vo{\d T}
 =
      \sum_{T \in \T}\int_{\d T}
 \bigg(
        \ip{ \nablans\sigma{\Fp\gn \Fp}, \Psi_{\gn \Fp\Fp \gn} } 
        -
        \ip{ \nablans\sigma{\gn \Fp\Fp}, \Psi_{\gn \Fp\Fp \gn} }
      \\
      &+ 
        \ip{ { \sff^\gn \otimes  \sigma_{\gn \Fp}},
        \Psi_{\Fp\Fp\gn\Fp} } 
 \bigg)\,\og_{\d T} + \sum_{F \in \Fint} \int_F \ip{\jump{\sigma_{\Fp}}, \Psi_{\gn \Fp\Fp\Fp}} \,\og_F.
    \end{aligned}
  \]
 By the $tt$-continuity of $\sigma$, the last term vanishes.
 The penultimate  term appears with the opposite sign in the
 expression for $\widetilde{\nabla^2 \sigma} (\Psi)$ in
 Lemma~\ref{lem:distr_inc_curved_part1}. That expression, put
 together with~\eqref{eq:161} and the above, gives
  \begin{align*}
    \widetilde{\nabla^2\sigma}(\Psi)
 =\sum_{T\in\T}&\bigg[  \int_T\langle\nabla^2\sigma,\Psi\rangle\,\vo{T}
    \\
    & +\int_{\d T}
 \Big(
      \ip{ \nablans\sigma{\Fp\gn \Fp\Fp}-\nablans\sigma{\gn \Fp\Fp}, \Psi_{\Fp \gn\gn \Fp} }
      +
      \ip{ { \sff^\gn \otimes  \sigma_{\gn \Fp}},
        \Psi_{\Fp\Fp\gn\Fp} } \Big)\,\og_{\d T}
    \\
    & + \int_{\d T}
 \Big(\langle\sigma_{\gn\gn}\sff^{\gn}
      +
      \nabla(\gn \lrcorner  \sigma),
      \;
      \Psi_{\Fp\gn\gn \Fp}\rangle
      -\ip{ { \sff^\gn \otimes  \sigma_{\gn \Fp}},
        \Psi_{\Fp\Fp\gn\Fp} }\Big)\,
      \vo{\d T} \bigg]
    \\
    & +\sum_{F\in\Fint}\int_{\d F}
      \langle\jmp{ \sigma_{\gn \Fp}},\,
      \Psi_{\gcn\gn\gn \Fp}\rangle\,\vo{\d F}.
  \end{align*}
 Using Lemma~\ref{lem:distr_inc_curved_part2} for the integrands over
 codimension-2 boundary terms, recalling that
  $\widetilde{\INC \sigma}(\Phi) = -\widetilde{\nabla^2\sigma}
 (\Psi)$ by~\eqref{eq:9},   
 and {since by definition $S=S_{(2,3)}$, $\Psi=S\Phi$, one has}
  $\Psi_{F\gn\gn F}=\Phi_{F \gn\gn F}$ and
  $\Psi_{\gcn\gn\gn\gcn}=\Phi_{\gcn\gn\gn\gcn}$,  
 we prove \eqref{eq:distr_cov_inc} {with \eqref{eq:restr_proj_relation}}.
 The proof of \eqref{eq:rel_b_inc_smooth} follows by comparing the
 terms in \eqref{eq:distr_cov_inc} and \eqref{eq:def_bh} and
 applying~\eqref{eq:6}.
\end{proof}

\begin{remark}
 In Theorem~\ref{thm:distr_inc_curved}, assuming that the metric $g$
 is smooth, we obtained \eqref{eq:rel_b_inc_smooth}, namely
  $b(g;\sigma,U)=-2\,\widetilde{\INC \sigma}( \Phi)$, for smooth $U$
 and $\Phi=\mapUA U$. However, all terms in \eqref{eq:def_bh}
 defining $b(g;\sigma,U)$ make sense even when $\Phi$ is less
 smooth, as long as $\Phi$ has the continuity properties of
  \eqref{eq:cont_A}. This motivates us to conjecture that
 Theorem~\ref{thm:distr_inc_curved} also holds for $B\in\Btesto$.  To
 prove this rigorously, it suffices to show that smooth functions are
 dense in $\Btest$, which is an interesting question on its own, but
 not needed for the current analysis of the correctness of our
 generalized Riemann curvature expression. Instead, all we need, for
 now, is the following definition, which just amounts to
 a renaming of $b(\dots)$.
\end{remark}

\begin{definition}
  \label{def:distr-incomp-oper}
 For any $\sigma \in \Regge(\T)$, define $\widetilde{\INC \sigma}$ as
 a linear functional on $\Atesto$ by
  \begin{align*}
    \widetilde{\INC \sigma} (A)
 = &\sum_{T\in\T}
 \bigg(\int_T\langle \INC \sigma, A\rangle\,\vo{T}
      \\
      &\qquad-
      \int_{\d T}
      \ip
 {
      \sigma_{\gn\gn}\sff^{\gn} + 
      \nablans\sigma{\Fr \gn \Fr} + (\nabla (\gn \lrcorner \sigma))_{\Fr} -
      \nablans\sigma{\gn \Fr\Fr}
 ,
        \Atnnt}\,\vo{\d T}\bigg)\nonumber
    \\
    &-\sum_{E\in\Eint}
      \sum_{F \in \F_E}
      \int_E
      \jmp{\sigma_{\gcn\gn}}^E_FA_{\gcn\gn\gn\gcn}\,\vo{E},
  \end{align*}
 where the jump in the last term is as defined in
  \eqref{eq:def_n_cn_jump}, and in the integrals over $\d T$, the
  $g$-normal $\gn$ points into the element $T$.  With this definition,
 in view of Lemma~\ref{lem:testspaceA}, we have
  \begin{equation}
    \label{eq:rel_b_inc}
 b(g;\sigma,U)=-2\,\widetilde{\INC \sigma}( \mapUA U ),
    \qquad U \in \Utesto.
  \end{equation}
 Moreover, if {$A$ is a smooth test function, $A\in\D\Atest$,
 then by Theorem~\ref{thm:distr_inc_curved},
 this extended definition coincides with
  $\widetilde{\INC \sigma}(A)$ as defined in~\eqref{eq:9}.}  
\end{definition}

The remainder of this section focuses on an adjoint-like operator of
the above-defined generalized $\INC$. This is needed later for the
numerical analysis. {More precisely,
  in Theorem~\ref{thm:conv_Riemann},
  to gain the convergence in $\Hmtwo$ of the generalized curvature towards the Riemann curvature for a sequence of Regge metrics $g_h$ approaching  a Riemannian metric $g$, we will need to show that  $b(g;\sigma, U)$
  appearing in \eqref{eq:def_bh}
  with $\sigma = g - g_h$ goes to  $0$ in an appropriate sense. For this,  we will integrate $\INC$ twice by parts and work with its adjoint.}

We have already seen in \eqref{eq:9} that
$\widetilde{\INC \sigma}(\Phi) = \int_\om \ip{ \sigma, \div \div S
  \Phi}\, \og$ for smooth $\Phi \in \D \Atest$. This motivates the next
definition.

\begin{definition}
  \label{def:div-div-dist}
 For any $B \in \Btesto$, we define
  $\widetilde{\div\div B}$ as a linear functional on $\Regge(\T)$ by 
  \begin{equation*}
 (\widetilde{\div\div B})(\sigma) :=
    -\widetilde{\INC\sigma} (SB),
    \qquad B \in \Btesto, \; \sigma \in \Regge(\T),
  \end{equation*}
 where the right-hand side is as in Definition~\ref{def:distr-incomp-oper}.
\end{definition}

Note that neither $\sigma$ nor $g$ is assumed to be globally smooth
in the next result. Given $T\in \T$, $F \in \triangle_{-1} T$ and
$E \in \triangle_{-1}F$, in analogy with \eqref{eq:def_n_cn_jump}, and
using the oriented $g$-normals and $g$-conormals there (see
Figure~\ref{fig:vectors}), define
\begin{equation}
  \label{eq:jmp-B-conormal-normal-EE}
  \jmp{B_{\gn\gcn \Er\Er}}^E_F=B^+_{\gn_F^{T_+}\gcn_E^F \Er\Er}+B^-_{\gn_F^{T_-}\gcn_E^F \Er\Er}.   
\end{equation}
Here, per the notation in {\eqref{eq:restriction}}, the subscript $\Er$ indicates arguments that are {restricted} onto the tangent space of $E$.
The proof of the following theorem is given after a few necessary lemmas.

\begin{theorem}
  \label{thm:distr_covariant_adjoint_inc}
  {Given a Regge metric $g \in \Regge^+(\T)$, one has for any $\sigma\in\Regge(\T)$ and $B\in\Btesto$}
  \begin{align}
    \nonumber 
 (\widetilde{\div\div B})(\sigma)
    &=\sum_{T\in\T}\bigg[\int_T\langle\sigma,\div\div B\rangle\,\vo{T} +
    \int_{\d T}\Big(\langle\sigma_{\Fr},
 (\div B+\divFR B)_{\gn \cdot\cdot}\rangle
    \\ \nonumber
    & \hspace*{2cm}
      + \sigma_{\Fr}:\overline{\sff}^\gn:B_{\Fr\gn\gn \Fr}- \langle\sff^\gn\otimes\sigma_{\Fr},
 B_{\Fr}\rangle\Big)\,\vo{\d T}\bigg]
    \\
    & +\sum_{E\in\Eint}\sum_{F \in \F_E}\int_E\langle\sigma_{\Er},
      \jmp{B_{\gn\gcn \Er\Er}}^E_F\rangle\,\vo{E},\label{eq:distr_adjoint_inc}
  \end{align}
  {where $\volform=\volform_g$ and the divergence operator is meant in the $g$-sense. Further,} the triple product
  $\sigma_{\Fr}:\overline{\sff}^\gn:{\Btnnt}$ {is defined by}
  \eqref{eq:triple-prod}, $\overline{\sff}^\gn=\mathbb{S}_{\Fr}(\sff^\gn)=\sff^\gn-H^{\gn}g_{\Fr}$ the trace-reversed second fundamental form (with mean curvature $H^{\gn}$), the jump $\jmp{B_{\gn\gcn \Er\Er}}^E_F$ is as
 defined in~\eqref{eq:jmp-B-conormal-normal-EE}, and in the integrals
 over $\d T$, the $g$-normal $\gn$ points into the element $T$.
\end{theorem}

\begin{remark}
 Note that in \eqref{eq:distr_adjoint_inc} only the $tt$-components of $\sigma$ are involved in the codimension-1 and codimension-2 terms, which are single-valued for $\sigma\in\Regge[\T]$. {This is crucial for distributional differential operators. In Definition~\ref{def:distr-incomp-oper} of the distributional incompatibility operator $\Atnnt$ and $A_{\gcn\gn\gn\gcn}$ are single-valued and the terms involving $\sigma$ jump, whereas in the adjoint operator $\sigma_{\Fr}$ is single-valued and the terms involving $B=SA$ jump.}
\end{remark}

\begin{lemma}
  \label{lem:auxiliary_codim2}
 Under the assumptions of Theorem~\ref{thm:distr_covariant_adjoint_inc} there holds
  \begin{align*}
    \sum_{E\in\Eint}\sum_{F\supset E}\int_E&\jmp{\langle\sigma,{B_{\gn\gcn\cdot\cdot}}+2{B_{\gn\gn\gcn\cdot}}\otimes\gn\rangle}_F\,\vo{E}\\
    &=\sum_{E\in\Eint}\sum_{F\supset E}\int_E(\langle \sigma_{\Er},\jmp{B_{\gn\gcn \Er\Er}}_F^E\rangle- \jmp{\sigma_{\gcn\gn}}_F^EB_{\gcn\gn\gn\gcn})\,\vo{E}.
  \end{align*}
\end{lemma}
\begin{proof}
 We split the terms {on the left-hand side below on each $T\supset F$} into normal {(to facets $F$)}, conormal {(normal to $E$ in $F$),} and codimension 2 components {(tangent to $E$),} only write the {non-zero} ones and simplify
\begin{align*}
  \langle\sigma, {B_{\gn\gcn\cdot\cdot}}+ 2{B_{\gn\gn\gcn\cdot}}\otimes \gn^\flat&\rangle = \langle \sigma_{\Er},B_{\gn\gcn \Er\Er}\rangle + \langle \sigma_{\Er\gn},B_{\gn\gcn \Er\gn}+2B_{\gn\gn\gcn \Er}\rangle +\langle \sigma_{\gcn \Er},B_{\gn\gcn\gcn \Er}\rangle \\
  &\quad+ \sigma_{\gcn\gn}(B_{\gn\gcn\gcn\gn}+ 2 B_{\gn\gn\gcn\gcn})\\
  &=\langle \sigma_{\Er},B_{\gn\gcn \Er\Er}\rangle - \langle \sigma_{\Er\gn},B_{\gn\gcn \Er\gn}\rangle +\langle \sigma_{\gcn \Er},B_{\gn\gcn\gcn \Er}\rangle- \sigma_{\gcn\gn}B_{\gcn\gn\gn\gcn}\\
  &=\langle \sigma_{\Er},B_{\gn\gcn \Er\Er}\rangle- \langle \sigma_{\gn \Er},B_{\gn\gcn \Er\gn}\rangle -\langle \sigma_{\gcn \Er},B_{\gcn\gcn \Er\gn}\rangle- \sigma_{\gcn\gn}B_{\gcn\gn\gn\gcn}.
\end{align*}
For the second identity, we used the (skew-)symmetry properties \eqref{eq:sym_B} of $B$, $B_{\gn\gn\gcn \Er}=- B_{\gn\gcn \Er\gn}$ and $B_{\gn\gn\gcn\gcn}=-B_{\gn\gcn\gcn\gn}$. For the last equality we used the symmetry of $\sigma$ and that again by \eqref{eq:sym_B} $B_{\gn\gcn\gcn \Er} = -B_{\gcn\gcn \Er\gn}$.

The first and last terms together read due to the continuity conditions on $\sigma$ and $B$ with the jumps \eqref{eq:def_n_cn_jump}, cf. Figure~\ref{fig:vectors},
\begin{align*}
  \sum_{E\in\Eint}\sum_{F\supset E}\int_E&\jmp{\langle \sigma_{\Er},B_{\gn\gcn \Er\Er}\rangle- \sigma_{\gcn\gn}B_{\gcn\gn\gn\gcn}}_F\,\vo{E} \\
  &=\sum_{E\in\Eint}\sum_{F\supset E}\int_E\big(\langle \sigma_{\Er},\jmp{B_{\gn\gcn \Er\Er}}_F^E\rangle- \jmp{\sigma_{\gcn\gn}}_F^EB_{\gcn\gn\gn\gcn}\big)\,\vo{E}.
\end{align*}
Next, we show with the same notation as in the proof of Lemma~\ref{lem:distr_inc_curved_part2} that the sum over the remaining middle two terms is zero. Therefore, we reorder the sum, consider an integral representation of the difference {involving} the rotation $R$ {(see the proof of Lemma~\ref{lem:distr_inc_curved_part2})} such that $R(0)\gn_- = \gn_-$ and $R(\theta)\gn_- = \gn_+$, and prove that the integrand is zero. To this end, we define $F(X,Y,Z,W):= R(t)X^\flat\otimes R(t)Y^\flat\otimes \sigma_{(R(t)Z)\Er}\otimes R(t)W^\flat$ and compute, {applying below the Leibniz rule first to $\frac{d}{dt}(F(\gn_-,\gcn_-,\gn_-,\gn_-))$, then to $\frac{d}{dt}(F(\gcn_-,\gcn_-,\gcn_-,\gn_-))$,}
\begin{align*}
  \sum_{E\in\Eint}\sum_{F\supset E}&\jmp{\langle \sigma_{\gn \Er},B_{\gn\gcn \Er\gn}\rangle +\langle \sigma_{\gcn \Er},B_{\gcn\gcn \Er\gn}\rangle}_F\\
 =\sum_{E \in \Eint}\sum_{T \supset E}&\langle \gn_+^\flat\otimes \gcn_+^\flat\otimes \sigma_{\gn_+\Er}\otimes\gn_+^\flat -\gn_-^\flat\otimes \gcn_-^\flat\otimes \sigma_{\gn_-\Er}\otimes\gn_-^\flat  \\
  &\quad+ \gcn_+^\flat\otimes \gcn_+^\flat\otimes \sigma_{\gcn_+\Er}\otimes\gn_+^\flat- \gcn_-^\flat\otimes \gcn_-^\flat\otimes \sigma_{\gcn_-\Er}\otimes\gn_-^\flat,B\rangle\\
 =\sum_{E \in \Eint}\sum_{T \supset E}&\int_0^{\theta}\frac{d}{dt}\langle F(\gn_-,\gcn_-,\gn_-,\gn_-) +F(\gcn_-,\gcn_-,\gcn_-,\gn_-),B\rangle\,dt\\
 =\sum_{E \in \Eint}\sum_{T \supset E}&\int_0^{\theta}\langle F(\gcn_-,\gcn_-,\gn_-,\gn_-) -F(\gn_-,\gn_-,\gn_-,\gn_-) + F(\gn_-,\gcn_-,\gcn_-,\gn_-)\\
  &\quad+F(\gn_-,\gcn_-,\gn_-,\gcn_-)-F(\gn_-,\gcn_-,\gcn_-,\gn_-)-F(\gcn_-,\gn_-,\gcn_-,\gn_-)\\
  &\quad-F(\gcn_-,\gcn_-,\gn_-,\gn_-)+F(\gcn_-,\gcn_-,\gcn_-,\gcn_-),B\rangle\,dt\\
 =\sum_{E \in \Eint}\sum_{T \supset E}&\int_0^{\theta}\langle F(\gcn_-,\gcn_-,\gn_-,\gn_-) +F(\gn_-,\gcn_-,\gcn_-,\gn_-)-F(\gn_-,\gcn_-,\gcn_-,\gn_-)\\
  &\quad-F(\gcn_-,\gcn_-,\gn_-,\gn_-),B\rangle\,dt=0.
\end{align*}
In the penultimate equation, we used the skew-symmetries \eqref{eq:sym_B} of $B$ and recognized in the final step that the remaining terms cancel. This concludes the proof.
\end{proof}

\begin{lemma}
  \label{lem:auxiliary_codim1}
 Under the assumptions of Theorem~\ref{thm:distr_covariant_adjoint_inc} {one has}
  \begin{align*}
    \ip{\sigma, \divFR B_{\gn \cdot\cdot\cdot}}
    & = \ip{\sigma, (\divFR B)_{\gn\cdot\cdot } }
      - \ip{ \sff^\gn \otimes \sigma, B_{\Fp\Fp\cdot\cdot}}\,,
    \\
    \ip{\sigma, \divFR(B_{\cdot\gn\gn \cdot} \otimes \gn^\flat) }
    & = \ip{ \sigma_{\gn\cdot}, (\divFR B)_{\gn\gn \cdot} }
      + \ip{ \sff^\gn \otimes \sigma_{\gn\cdot}, B_{\gn \Fp\Fp\cdot }}
      - \sigma_{\Fp}: \sff^\gn : B_{\Fp \gn\gn \Fp}\,.
  \end{align*}
\end{lemma}
\begin{proof}
 We take the surface divergence \eqref{eq:divF-def} of the above expression using a $g$-orthonormal basis $\{\gt_i\}_{i=1}^{N-1}$ of the tangent space. Then {use the Leibniz rule and express the second fundamental form} as $\nabla_{\gt_i}\gn^\flat = -\sum_{j=1}^{N-1}\sff^\gn(\gt_i,\gt_j)\gt_j^\flat$
  \begin{align*}
    \ip{\sigma,\divFR B_{\gn \cdot\cdot\cdot}} &= \sum_{i=1}^{N-1}\ip{\sigma, (\nabla_{\gt_i}B)_{\gn\gt_i\cdot\cdot} + B_{(\nabla_{\gt_i}\gn)\gt_i\cdot\cdot}}\\
    &=\sum_{i=1}^{N-1}\left(\ip{\sigma,(\nabla_{\gt_i}B)_{\gt_i\gn\cdot\cdot}} - \sum_{j=1}^{N-1}\ip{\sigma, \sff^\gn(\gt_i,\gt_j)B_{\gt_j\gt_i\cdot\cdot}}\right)\\
    &=\ip{\sigma,(\divFR B)_{\gn\cdot\cdot}}-\ip{ \sff^\gn \otimes \sigma, B_{\Fp\Fp\cdot\cdot}}.
  \end{align*}
 For the second line, we used the symmetries \eqref{eq:sym_B} of $B$ and $\sigma$ for the first term, and for the last line, the symmetry of the second fundamental form. 
  
 We proceed analogously for the second identity
\begin{align*}
  &\langle\sigma,\divFR (B_{\cdot\gn\gn\cdot}\otimes\gn^\flat)\rangle \\
  &=\!\ip{\sigma, \sum_{i=1}^{N-1}(\nabla_{\gt_i}B)_{\gt_i\gn\gn\cdot}\otimes\gn^\flat + B_{\gt_i(\nabla_{\gt_i}\gn)\gn\cdot}\otimes\gn^\flat+B_{\gt_i\gn(\nabla_{\gt_i}\gn)\cdot}\otimes\gn^\flat+B_{\gt_i\gn\gn\cdot}\otimes\nabla_{\gt_i}\gn^\flat}\\
  &=\!\ip{\sigma_{\gn\cdot},(\divFR B)_{\gn\gn\cdot}} \\
  &\quad-\! \sum_{i,j=1}^{N-1}\ip{\sigma, \sff^\gn(\gt_i,\gt_j) B_{\gt_i\gt_j\gn\cdot}\otimes\gn^\flat\!+\!\sff^\gn(\gt_i,\gt_j)B_{\gt_i\gn\gt_j\cdot}\otimes\gn^\flat \!+\! \sff^\gn(\gt_i,\gt_j)B_{\gt_i\gn\gn\cdot}\otimes\gt_j}.
\end{align*}
By the (skew) symmetry of $B$ and the second fundamental form the second term $\sum_{i,j=1}^{N-1}\sff^\gn(\gt_i,\gt_j)B_{\gt_i\gn\gt_j\cdot}$ vanishes. {The last line of the above formula can be rewritten, using also the triple product \eqref{eq:triple-prod},} as
\begin{align*}
  &-\sum_{i,j=1}^{N-1}\ip{\sigma, \sff^\gn(\gt_i,\gt_j) B_{\gt_i\gt_j\gn\cdot}\otimes\gn^\flat + \sff^\gn(\gt_i,\gt_j)B_{\gt_i\gn\gn\cdot}\otimes\gt_j}\\
  &\hspace*{2cm}=\ip{\sff^\gn\otimes \sigma_{\gn\cdot},B_{\gn\Fp\Fp\cdot}}-\sigma_{\Fp}:\sff^\gn:B_{\Fp\gn\gn\Fp},
\end{align*}
which concludes the proof.
\end{proof}

\begin{proof}[{\bf{Proof of Theorem~\ref{thm:distr_covariant_adjoint_inc}}}]

  {Letting $s_{\Eint} = \sum_{E\in\Eint} \sum_{F \in \F_E} \int_E
  \jmp{\sigma_{\gcn\gn}}^E_FB_{\gcn\gn\gn\gcn}\,\vo{E}$ and examining the expression given in Definition~\ref{def:div-div-dist}, we can write {with \eqref{eq:restr_proj_relation}}
  \begin{align}
    &(\widetilde{\div\div B}
 )
 (\sigma)-s_{\Eint}
 =
      -\widetilde{\INC\sigma} (SB) -s_{\Eint} \label{eq:div-div-thm-1}
    \\
 =
     \!\sum_{T\in\T}&
      \int_T\langle\nabla^2\sigma,B\rangle\,\vo{T}
      \!+\!
      \int_{\d T}
      \ip
 {
      \sigma_{\gn\gn}\sff^{\gn} \!+\! 
      \nablans\sigma{\Fp \gn \Fp} \!+\! (\nabla (\gn \lrcorner \sigma))_{\Fp} \!-\!
      \nablans\sigma{\gn \Fp\Fp}
 ,
 B_{\Fp\gn\gn \Fp}}\,\vo{\d T}.\nonumber
  \end{align} 
 Looking at the term involving the second covariant derivative, we use \eqref{eq:ibp_volume} to integrate it by parts}

  \begin{equation}
    \label{eq:div-div-thm-5}
    \begin{aligned}
      \int_T\langle\nabla^2\sigma,B\rangle\,\vo{T}
      &= \int_T-\langle\nabla\sigma,\div B\rangle\,\vo{T} -\int_{\d T} \langle \nabla \sigma,B_{\gn \cdot\cdot\cdot}\rangle\,\vo{\d T}\\
      &=\int_T\langle\sigma,\div\div B\rangle\,\vo{T}
        +\int_{\d T}\langle \sigma, (\div B)_{\gn \cdot\cdot}\rangle\,\vo{\d T}-\int_{\d T}
        \langle \nabla \sigma,B_{\gn\cdot\cdot\cdot}\rangle\,\vo{\d T}.
    \end{aligned}    
  \end{equation}
  {We split the last integrand into its normal and tangential components.} 
 From the resulting eight terms, eliminating those
 that are zero by the skew-symmetries
 \eqref{eq:sym_B} of $B$, we obtain, {recalling the notation $(\nabla\sigma)_{\Fp}=(\nabla\sigma)_{\Fp\Fp\Fp}$,}
  \begin{align}
    \label{eq:181}
    \langle \nabla \sigma,B_{\gn\cdot\cdot\cdot}\rangle
    & =
      \langle (\nabla \sigma)_{\Fp\Fp\gn}, B_{\gn \Fp\Fp \gn}\rangle + 
    \langle (\nabla \sigma)_{\gn \Fp\Fp}, B_{\gn \gn \Fp\Fp}\rangle +
      \langle (\nabla \sigma)_{\Fp}, B_{\gn \Fp\Fp\Fp}\rangle,
  \end{align}
 which we then substitute into~\eqref{eq:div-div-thm-5}.
  
  {We rewrite \eqref{eq:div-div-thm-1} using \eqref{eq:div-div-thm-5}  and \eqref{eq:181}, observing a cancellation of the term $\ip{(\nabla\sigma)_{\Fp\Fp\gn},B_{\gn \Fp\Fp\gn}}$ coming from \eqref{eq:181} with $-\ip{(\nabla\sigma)_{\Fp\gn \Fp},B_{\Fp\gn\gn \Fp}}$ already appearing in \eqref{eq:div-div-thm-1} (use \eqref{eq:sym_B} and the symmetries of $\sigma$)}
  \begin{align*}
    &(\widetilde{\div\div B}
 )
 (\sigma)-s_{\Eint}
 =  \sum_{T\in\T}\Big(
      \int_T\ip{ \sigma, \div\div B}\,\vo T\\
     &\quad+ 
      \int_{\d T}
      \ip
 {
      \sigma_{\gn\gn}\sff^{\gn} + 
      \nablans\sigma{\Fp \gn \Fp} + (\nabla (\gn \lrcorner \sigma))_{\Fp} -
      \nablans\sigma{\gn \Fp\Fp}
 ,
 B_{\Fp \gn\gn \Fp}}\,\vo{\d T}
    \\
    &\quad+
      \int_{\d T}
 \Big[ \langle \sigma, (\div B)_{\gn \cdot\cdot}\rangle
      -\langle (\nabla \sigma)_{\Fp}, B_{\gn \Fp\Fp\Fp}\rangle
      +\langle (\nabla \sigma)_{\gn \Fp\Fp} - (\nabla \sigma)_{\Fp\gn \Fp},
 B_{\Fp\gn \gn \Fp}\rangle       
 \Big]\,\vo{\d T}\Big).      
  \end{align*}    
  {The surviving term $\ip{(\nabla\sigma)_{\Fp\gn \Fp},B_{\Fp\gn\gn \Fp}}$} can be rewritten using Lemma~\ref{lem:triple-prod}.  Although
 Lemma~\ref{lem:triple-prod} was proved only for $A \in \Atest$,
 using the symmetries of $B$ in \eqref{eq:sym_B} in place of those of
  $A$, we find that \eqref{eq:17Ainnj} holds also for $B$, so the same
 proof there shows that the identity of the lemma also holds for
  $B \in \Btest$. Therefore,
  $\ip{ (\nabla(\gn \lrcorner \sigma))_{\Fp}, B_{\Fp\gn\gn\Fp}}=\ip{
 (\nabla\sigma)_{\Fp\gn \Fp},B_{\Fp\gn\gn \Fp}}-\sff:\sigma_{\Fp}:B_{\Fp\gn\gn\Fp}$.
 These observations lead to
  \begin{align}
    \label{eq:div-div-thm-intermediate1}
    \begin{split}
 (\widetilde{\div\div B})(\sigma)
 =&\, s_{\Eint}
         + \sum_{T \in \T}  \int_T \ip{ \sigma, \div\div B}\, \vo T
        + \int_{\d T} \ip{ \sigma, (\div B)_{\gn \cdot\cdot} }\, \vo{\d T}
      \\
      & + \sum_{T \in \T} \int_{\d T}
 \Big[
        \ip{ \sigma_{\gn\gn} \sff^\gn, B_{\Fp \gn\gn \Fp} }
        - \sff^\gn : \sigma_{\Fp} : B_{\Fp\gn\gn\Fp}
        + s_1
 \Big]\, \vo{\d T},
    \end{split}
  \end{align}
 where 
  $s_1 = \ip{ (\nabla \sigma)_{\Fp \gn \Fp}, B_{\Fp \gn \gn \Fp} } - \ip{
 (\nabla \sigma)_{\Fp}, B_{\gn \Fp\Fp\Fp}}$.

 Next, let us fix a facet $F \in \triangle_{-1}T$ of an element $T$
 and examine the two terms whose difference is $s_1$.  
 The first term equals
  \[
  \ip{ (\nabla \sigma)_{\Fp \gn \Fp}, B_{\Fp \gn \gn \Fp} } = \ip{ (\nablaF
    \sigma)_{\Fp\Fp \gn}, B_{\Fp \gn \gn \Fp} } =\ip{ \nablaF \sigma, B_{\cdot \gn
      \gn \cdot} \otimes \gn^\flat}.
      \] 
 Its second term can be rewritten, 
 noting that arguments indicated by subscript $\Fp$ are  tangentially projected using  $\Proj=\idop-\gn\otimes\gn^\flat$, and noting that $(\nablaF\sigma)_{\gn\cdot\cdot}=0$, as follows:
      \begin{align*}
        -\langle (\nabla \sigma)_{\Fp}, B_{\gn \Fp\Fp\Fp}\rangle
        &= -\langle \nablaF\sigma, B_{\gn\cdot \Fp\Fp}-\gn^\flat\otimes B_{\gn\gn \Fp\Fp}\rangle
 = -\langle \nablaF\sigma, B_{\gn\cdot \Fp\Fp}\rangle
        \\
        &= -\langle \nablaF \sigma,B_{\gn\cdot \cdot \Fp} \rangle
 = 
          -\langle \nablaF \sigma,{B_{\gn\cdot\cdot\cdot}}
          -B_{\gn\cdot\cdot\gn}\otimes \gn^\flat\rangle
        \\
        &=\langle \nablaF\sigma, B_{\cdot\gn\gn\cdot}\otimes \gn^\flat-B_{\gn\cdot\cdot\cdot}\rangle, 
      \end{align*}
 where we have used \eqref{eq:sym_B} twice, once
 in the second line for 
      $B_{\gn \cdot \Fp\Fp} =B_{\gn \cdot \cdot \Fp}$, and again in the last line for   $B_{\gn\cdot\cdot\gn}=B_{\cdot\gn\gn\cdot}$.
 Therefore, $s_1$ in \eqref{eq:div-div-thm-intermediate1} can be expressed using
  $\nablaF \sigma$, which then permits the use of  the surface integration
 by parts formula~\eqref{eq:ibp_surface} on the facet $F$. Namely, 
  \begin{align*}
    &\int_F s_1
      \, \vo{F}
 = \int_F \ip{ \nablaF \sigma,
      2B_{\cdot \gn \gn \cdot} \otimes \gn^\flat - B_{\gn \cdot\cdot\cdot}
 }\, \vo{F}
    \\
    & =
      \int_F
      \ip{ \sigma, -\divFR (2B_{\cdot \gn \gn \cdot} \otimes \gn^\flat - B_{\gn \cdot\cdot\cdot})}\,\vo{F}
      -\int_F H^{\gn} \ip{ \gn^\flat \otimes \sigma ,
      2B_{\cdot \gn \gn \cdot} \otimes \gn^\flat - B_{\gn \cdot\cdot\cdot}} \,\vo{F}
    \\
    & \quad-\int_{\d F}
      \ip{  \sigma,      2B_{\gcn \gn \gn \cdot} \otimes \gn^\flat - B_{\gn\gcn \cdot\cdot}}
      \,\vo{\d F},
  \end{align*}
 where $\gcn$ is the inward-pointing $g$-conormal on $\d F$.
 We use Lemma~\ref{lem:auxiliary_codim1} to rewrite the first term on the right-hand side above to get 
  \begin{align*}
    &\ip{ \sigma, \divFR (B_{\gn \cdot\cdot\cdot}-2B_{\cdot \gn \gn \cdot} \otimes \gn^\flat )}\,\vo{F} = \ip{\sigma, (\divFR B)_{\gn\cdot\cdot } }- \ip{ \sff^\gn \otimes \sigma, B_{\Fp\Fp\cdot\cdot}}\\
    &\qquad
    -2\ip{ \sigma_{\gn\cdot}, (\divFR B)_{\gn\gn \cdot} }
    -2\ip{ \sff^\gn \otimes \sigma_{\gn\cdot}, B_{\gn\Fp\Fp \cdot }}
    + 2\sigma_{\Fp}: \sff^\gn : B_{\Fp \gn\gn \Fp}\,.
  \end{align*}
 The term multiplying $H^{\gn}$ in the integral of $s_1$ can be simplified to
  \[
    \ip{ \gn^\flat \otimes \sigma,
      2B_{\cdot \gn \gn \cdot} \otimes \gn^\flat - B_{\gn \cdot\cdot\cdot}}
 = \ip{ \sigma_{\gn\cdot}, 2B_{\gn \gn \gn \cdot}}
    - \ip{ \sigma,   B_{\gn \gn \cdot\cdot}} = -\ip{\sigma_{\Fp}, B_{\gn\gn\cdot\cdot}}.
  \]
 Hence, inserting all the work just done on the terms involved in $s_1$, we get from~\eqref{eq:div-div-thm-intermediate1} that 
  \begin{align}
 (\widetilde{\div\div B})&(\sigma)
 = \tilde{s}_{\Eint}
       + \sum_{T \in \T}  \int_T \ip{ \sigma, \div\div B}\, \vo T
        + \int_{\d T} \ip{ \sigma, (\div B)_{\gn \cdot\cdot} }\, \vo{\d T}
      \nonumber\\
      & + \sum_{T \in \T} \int_{\d T}
 \Big[
        \ip{ \sigma_{\gn\gn} \sff^\gn, B_{\Fp \gn\gn \Fp} }
        - \sff^\gn : \sigma_{\Fp} : B_{\Fp\gn\gn \Fp}
 \Big]\, \vo{\d T}
        \nonumber\\
      & + \sum_{T \in \T} 
        \int_{\d T} \Big[
        \ip{ \sigma, (\divFR B)_{\gn\cdot\cdot} - 2(\divFR B)_{\gn\gn\cdot } \otimes \gn^\flat}
        + H^{\gn}  \ip{ \sigma_{\Fp}, B_{\gn\gn \Fp\Fp}}
 \Big]\, \vo{\d T}
        \nonumber\\
      \label{eq:div-div-thm-intermediate2}& +
        \sum_{T \in \T} 
        \int_{\d T} \Big[2\sigma_{\Fp} : \sff^\gn : B_{\Fp\gn\gn \Fp}
        -\ip{ \sff^\gn \otimes \sigma, 2B_{\gn\Fp\Fp\cdot} \otimes \gn^\flat + B_{\Fp\Fp\cdot\cdot}}
 \Big]\, \vo{\d T},
  \end{align}
 where
  $\tilde s_{\Eint} = s_{\Eint} + \sum_{T \in \T} \sum_{F \subset \d
 T} \int_{\d F} \ip{ \sigma, B_{\gn\gcn \cdot\cdot} - 2B_{\gcn \gn
      \gn \cdot } \otimes \gn^\flat} \, \vo{\d F} $ collects all
 integrals over codimension-2 mesh entities.

 Let us collect all integrands on the right-hand side
 of~\eqref{eq:div-div-thm-intermediate2} involving the second
 fundamental form
 into $s_2$, and all terms involving {only one} divergence operator into
  $s_3$, i.e.,
  \begin{align*}
 s_2
    & =
      \ip{ \sigma_{\gn\gn} \sff^\gn, B_{\Fp \gn\gn \Fp} }
      -\ip{ \sff^\gn \otimes \sigma, 2B_{\gn\Fp\Fp\cdot} \otimes \gn^\flat + B_{\Fp\Fp\cdot\cdot}}
      - \sff^\gn : \sigma_{\Fp} : B_{\Fp\gn\gn \Fp}\\
      &\quad+ 2\sigma_{\Fp} : \sff^\gn :
 B_{\Fp\gn\gn \Fp},\\
 s_3
    & =  \ip{ \sigma, (\div B)_{\gn \cdot\cdot}+(\divFR
 B)_{\gn\cdot\cdot}} \!-\!\ip{ \sigma_{\gn\cdot}, 2(\divFR B)_{\gn\gn\cdot }}.
  \end{align*}
 To simplify $s_2$, first note that using \eqref{eq:B_nnFF} {and the symmetry of $\sigma$, $\sff^{\gn}$, and $B_{\cdot\gn\gn \cdot}$,} the last
 two terms simplify to $\sigma_{\Fp} : \sff^\gn : B_{\Fp\gn\gn \Fp}$. Hence,
  \begin{align}
    \label{eq:0500}
 s_2
    & =  \sigma_{\Fp} : \sff^\gn : B_{\Fp \gn \gn \Fp}
      + \ip{ \sigma_{\gn\gn} \sff^\gn, B_{\Fp \gn\gn \Fp} }
      -\ip{ \sff^\gn \otimes \sigma,
      2B_{\gn\Fp\Fp\cdot} \otimes \gn^\flat + B_{\Fp\Fp\cdot\cdot}}.      
  \end{align}
 In the last term, splitting $\sigma$ by \eqref{eq:sigma-split},
  {
  \begin{align*}
    &\ip{ \sff^\gn \otimes \sigma, \,
    2B_{\gn\Fp\Fp\cdot} \otimes \gn^\flat + B_{\Fp\Fp\cdot\cdot}}
 =
      \ip{ \sff^\gn \otimes \sigma_{\Fp}, \,
    2B_{\gn\Fp\Fp\Fp} \otimes \gn^\flat_{\Fp} + B_{\Fp\Fp\Fp\Fp}}\\
      &+\ip{ \sff^\gn \otimes \gn^\flat \otimes \sigma_{\gn \Fp}+\sff^\gn \otimes \sigma_{\Fp\gn}\otimes \gn^\flat,\,
      2B_{\gn\Fp\Fp\cdot} \otimes \gn^\flat + B_{\Fp\Fp\cdot\cdot}}+ \ip{ \sigma_{\gn\gn}\sff^\gn, 2B_{\gn\Fp\Fp\gn} + B_{\Fp\Fp\gn\gn}},
  \end{align*}
 we observe that in the first term on the right-hand side $\ip{ \sff^\gn \otimes \sigma_{\Fp}, \,
    2B_{\gn\Fp\Fp\Fp} \otimes \gn^\flat_{\Fp}}$ cancels, the second term cancels by the skew-symmetries \eqref{eq:sym_B} of $B$ and by the symmetry of $\sigma$ and $\sff^\gn$, while the last simplifies, giving
  \begin{align*}
    \ip{ \sff^\gn \otimes \sigma, \,
    2B_{\gn\Fp\Fp\cdot} \otimes \gn^\flat + B_{\Fp\Fp\cdot\cdot}} = \ip{ \sff^\gn \otimes \sigma_{\Fp}, B_{\Fp}} +
      \ip{ \sigma_{\gn\gn}\sff^\gn, B_{\gn\Fp\Fp\gn}}.
  \end{align*}}
 Hence,~\eqref{eq:0500} simplifies, after using \eqref{eq:B_nnFF},  to
  \begin{equation}
    \label{eq:s2}
 s_2 = \sigma_{\Fp} : \sff^\gn : B_{\Fp \gn \gn \Fp}
    -\ip{ \sff^\gn \otimes \sigma_{\Fp}, B_{\Fp}}.
  \end{equation}

 To simplify $s_3$, splitting $\sigma$ by \eqref{eq:sigma-split} again, 
 and using the corresponding splits for the divergence terms, the
 first inner product in $s_3$ splits into four terms, and the second
 into two terms.  Terms with $(\divFR B)_{\gn\cdot\gn}$ and
  $(\div B)_{\gn\cdot \gn}$ vanish due to the skew-symmetries
  \eqref{eq:sym_B} of $B$. What remains gives
  \begin{align*}
 s_3
    & \! =\! \ip{ \sigma_{\Fp}, (\div B + \divFR B)_{\gn \Fp\Fp}} 
      \!+\! \ip{\sigma_{\gn \Fp}, (\div B  + \divFR B)_{\gn\gn \Fp} }
      \!-\! 2 \ip{\sigma_{\gn \Fp},  (\divFR B)_{\gn\gn \Fp} }.     
  \end{align*}
 Since $(\nabla_X B)_{\gn\gn\gn\cdot}$ and
  $(\nabla_X B)_{\cdot\gn\gn\gn}$ vanish, it is easy to see
 from~\eqref{eq:divF-def} that there holds
  $(\divFR B)_{\gn\gn \Fp} = (\div B)_{\gn\gn\cdot}$. Thus, $s_3$ further
 simplifies to
  \begin{equation}
    \label{eq:s3}
 s_3 =\ip{ \sigma_{\Fp}, (\div B)_{\gn \Fp\Fp}+(\divFR B)_{\gn \Fp\Fp}}. 
  \end{equation}
 Finally, we simplify $\tilde{s}_{\Eint}$ by 
 Lemma~\ref{lem:auxiliary_codim2},
  \begin{align}
    \nonumber 
    \tilde{s}_{\Eint}
    & = s_{\Eint}
      +  \sum_{T \in \T} \sum_{F \subset \d T}
      \int_{\d F}  \ip{ \sigma,
 B_{\gn\gcn \cdot\cdot} - 2B_{\gcn \gn \gn \cdot } \otimes \gn^\flat}
      \, \vo{\d F}
    \\     \label{eq:sEint}
    & = 
      \sum_{E \in \Eint} \sum_{F \in \F_E} \int_E \ip{\sigma_E, \jmp{{B_{\gn\gcn EE}}}_F^E}\,
      \vo{E}.
  \end{align}
  {Inserting}~\eqref{eq:sEint}, \eqref{eq:s3}, and~\eqref{eq:s2} in
  \eqref{eq:div-div-thm-intermediate2} {proves the theorem with \eqref{eq:restr_proj_relation}}.
\end{proof}

\section{Numerical analysis}
\label{sec:num_ana}

In this section, we prove \emph{a priori} convergence estimates for the densitized distributional Riemann curvature tensor in the following setting. Let $\Omega\subset\R^N$, $N\geq 2$, be a domain equipped with a smooth ``exact'' metric tensor $\gex$. We assume that a family of shape-regular triangulations $\{\T_h\}_{h>0}$ of $\Omega$ consisting of affine simplices with meshsize  $h:=\max_{T\in\T_h} h_T$, $h_T:=\mathrm{diam}(T)$, are given together with a family of Regge metrics $\{\gappr\}_{h>0}$. They approximate $\gex$ in a sense made precise shortly.
Here, shape-regularity means that there exists a constant $C_0>0$ independent of $h$ such that for all $h>0$
\begin{align}
  \label{eq:shape-regularity}
\sup\limits_{T\in\T_h}\frac{h_T}{\rho_T}\leq C_0,
\end{align}
where $\rho_T$ is the inradius of $T$. Our convergence estimates will have constants that may depend on $C_0$. {We use standard notation and results from numerical analysis throughout this section, cf. Appendix~\ref{sec:numerical-analysis-prelims} for an overview.}

In finite element computations, we use a reference element $\hat T$,
the unit {Euclidean} $N$-simplex, and the space $\Pol^k(\hat T)$ of polynomials of
degree at most~$k$ on $\hat T$. 
Let $T\in\T_h$ be an affine $N$-simplex of the triangulation, and let $\Phi_T: \hat{T}\to T$ be the affine bijection from the reference simplex to $T$. 
Define the {\em Regge finite element space} of degree $k$ by
\begin{equation}
  \label{eq:regge_fem_space}
  \begin{aligned}
    \Regge_h^k\! =\! \{ \sigma \in \Regge(\T_h):
    \text{ for all } T \in \T_h,\;
    \sigma|_T = & \sigma_{ij} dx^i \otimes dx^j,
    \\
    & \sigma_{ij}\circ\Phi_T \in \Pol^k(\hat T ) \},
  \end{aligned}
\end{equation}
where $\Regge(\T_h)$ is as in \eqref{eq:ttspace}.

\begin{remark}
For finite element computations on general manifolds $M$, we would need charts so that {each element $T\subset M$ in $\T_h(M)$} is covered by a single chart giving the coordinates $x^i$ on $T$. The chart identifies the parameter domain of $T$ as the Euclidean $N$-simplex $\Teuc\subset\Omega\subset \R^N$ diffeomorphic to  $T$. Let ${\tilde{\Phi}}: T \to \Teuc $ denote this diffeomorphism. {Then $\Phi_T = \tilde{\Phi}^{-1} \circ \Phi : \hat T \to T$ maps the reference element, where $\Pol^k(\hat T)$ is defined, diffeomorphically to $T$,} where $\Phi: \hat T \to \Teuc$ is the standard affine bijection from the reference element to the Euclidean simplex $\Teuc$. 
The present Euclidean analysis corresponds to $M=\Omega$, with {$\tilde{\Phi}$} set globally to the identity, and therefore uses affine simplices.
\end{remark}

Throughout, we use standard Sobolev{--Slobodeckij} spaces $\Wsp[\Omega]$ and their norms and seminorms for any $s\geq 0$ and $p\in [1,\infty]$. When the domain is $\Omega$, we omit it from the
norm notation if there is no chance of confusion.  We also use the
elementwise norms $\|u\|_{\Wsph}^p =\sum_{T\in \T_h}\|u\|^p_{\Wsp[T]},$
with the usual adaptation for $p=\infty$. When $p=2$, we put
$\|\cdot\|_{\Hsh}=\|\cdot\|_{W^{s,2}_h}$.  Furthermore, {define for any elementwise $H^1$ function $\sigma$ and elementwise $H^2$ function $\rho$ the norms}
\[
\nrm{\sigma}_2^2 = \| \sigma \|_{L^2}^2 + h^2 \| \sigma \|^2_{H_h^1}, \qquad
\nrmt{\rho}_2^2 = \| \rho \|_{L^2}^2 + h^2 \| \rho \|^2_{H_h^1} + h^4 \| \rho \|^2_{H_h^2}.
\]
Note that both norms scale like the $\Ltwo$-norm due to the added powers of the meshsize.

\subsection{Statements of the convergence results}

In~\eqref{eq:def_curv_riemann}, the generalized densitized curvature
operator $\RogA$ was defined as a linear functional on the
metric-independent mesh-dependent test space $\Utest$ of
\eqref{eq:testspace_U}.  To highlight the dependence on the mesh
$\T_h$ and the metric $g_h$ on it, we now refer to $\Utest$ as
$\Utesth$ and to $\RogA$ as $\RogA(g_h)$.  Obviously, for the smooth
metric $\gex$, we have $\RogA(\gex) = (\Curvature\,\volform) (\gex)$
with the exact smooth curvature $\Curvature$ defined in
\eqref{eq:Q-defn}.  We prove convergence of $\RogA(g_h)$ to the exact
densitized curvature $(\Curvature\,\volform) (\gex)$ in the
$H^{-2}$-norm. For this, we only need {to consider} the action of $\RogA(g_h)$ on an
elementwise $H^2$-version of the metric-independent test space
$\Utesth$. After identifying $\W^{N-2}(\om)^{\odot2}$ with symmetric
$\tilde{N} \times \tilde{N}$ matrix fields, where
\[
  \tilde N = \binom{N}{N-2} =\frac{N(N-1)}{2}, 
\]
we define
\begin{align*}
  \begin{split}
  \Utesth^2
  :=&\Big\{U\in \prod_{T\in\T_h}
  H^2(T,\R^{\tilde N\times\tilde N}_{\mathrm{sym}}):\,U_{\Fr}(X_1,\dots, X_{N-2},Y_1,\dots,Y_{N-2})\\
  &\qquad\text{ is single-valued for any } X_i, Y_i  \in\Xm{F}, F\in\Fint_h\},\\
  \Utesth^{2,0}
  :=&\Big\{U\in\Utesth^2:\,
  U_{\Fr}(X_1,\dots, X_{N-2},Y_1,\dots,Y_{N-2})\! =\!0,
          \\
          & \qquad\qquad\qquad\qquad\qquad\qquad \text{ for any }
 X_i, Y_i \in \Xm{F},\,F\in{\Fbnd}_{,h}\}.
  \end{split}
\end{align*}
The functional $\RogA(g_h)$ is well defined for $U\in\Utesth^{2,0}$.
Moreover,
$H^{2}_0(\Omega,\R^{\tilde{N}\times \tilde{N}}_{\mathrm{sym}})
=
H^{2}_0(\Omega) \otimes \R^{\tilde{N}\times \tilde{N}}_{\mathrm{sym}}
= \{ u: \om \to
  \R^{\tilde{N}\times \tilde{N}} \,:\, u_{ij} = u_{ji} \in
  H_0^2(\om) \} \subset \Utesth^{2,0}$. Let its dual space (see  \eqref{eq:hmknorm})
be denoted by $H^{-2}(\om,\R^{\tilde{N}\times \tilde{N}}_{\mathrm{sym}})$ and the corresponding
dual norm by  $\|\cdot\|_{H^{-2}}$.

\begin{theorem}
  \label{thm:conv_Riemann}
 Let $\Omega\subset\R^N$, $N\ge 2$, be a domain equipped with a
 smooth Riemannian metric $\gex$. Assume that $\{\gappr\}_{h>0}$ is a
 family of Regge metrics on a shape-regular family of triangulations
  $\{\T_h\}_{h>0}$ of $\Omega$ with
  $\lim_{h \to 0}\|\gappr-\gex\|_{\Linf}=0$ and
  $C_1:=\sup_{h>0}\max_{T \in
    \T_h}\|\gappr\|_{W^{2,\infty}(T)}<\infty$. Let
  \[
 C_g \!=\!
    \begin{cases}
      1+\max_{T \in \T_h}(h_T^{-1}\|\gex-\gappr\|_{\Linf[T]})+\|\gex-\gappr\|_{\Winfh},
      & N=2,
      \\
      1+\max_{T \in \T_h}(h_T^{-2}\|\gex-\gappr\|_{\Linf[T]})+\max_{T \in \T_h}(h_T^{-1}\|\gex-\gappr\|_{\Winf[T]}),
      & N\ge 3.
    \end{cases}
  \]
 Then there exists $h_0>0$ and a $C>0$ depending on $N$, $C_0$,
  $C_1$, $\|\gex\|_{W^{2,\infty}(\Omega)}$, and
  $\|\gex^{-1}\|_{\Linf[\Omega]}$ such that for all $h\leq h_0$,
  \[
    \|\RoggA{\gappr}-(\Curvature\,\volform)(\gex)\|_{H^{-2}}
    \leq C C_g \nrmt{\gappr-\gex}_2.
  \]
\end{theorem}
When the metric approximation is sufficiently good to have $C_g$
bounded independently of meshsize, then rates of convergence can be
quantified from the above result.
As an example, let $\OptInt[k]:\Cinf[\Omega,\Sc]\to\Regge_h^k$ be an interpolation operator into the Regge finite element space satisfying the following {best-approximation} conditions: there  exists a $p\in [2,\infty]$ such that the interpolant can be continuously extended to symmetric tensor fields in $W^{k+1,p}(\Omega)$ and
\begin{subequations}
  \label{eq:interp_cond}
  \begin{align}
 |\OptInt[k] g - g|_{W^{t,p}(T)} \le C_2 h_T^{k+1-t} |g|_{W^{k+1,p}(T)},\qquad\forall t\in [0,k+1],
  \end{align}
 where $C_2=C_2(N,k,h_T/\rho_T,t)$. {Further,} suppose  that for all $
  {0 < s \le \min\{k+1,2\}},$
  \begin{align}
 |\OptInt[k] g - g|_{W^{t,\infty}(T)} \le C_2 h_T^{s-t} |g|_{W^{s,\infty}(T)},\qquad \forall t\in [0,s]. 
  \end{align}
\end{subequations}
The canonical Regge interpolant{, cf.~\eqref{eq:RegInt},} from~\cite{li18} is an example of an interpolant fulfilling \eqref{eq:interp_cond}. It is well-defined for $g$ in $W^{s,p}(\Omega)$ with $s>(N-1)/p$, a sufficient condition for defining traces on boundaries of codimension up to $N-1$. Then, setting, e.g., $p=N$ meets the requirements.
An Oswald-type interpolant ({see~\cite{oswaldBPXpreconditionerP1Elements1993}, \cite[Section 22.2]{ernFiniteElementsApproximation2021}, or  Appendix~\ref{sec:numerical-analysis-prelims})} performing local elementwise $L^2$-projections and averaging the degrees of freedom shared by different elements, also leads to a valid choice if the requirements are weakened to hold on element patches, see e.g. \cite[Appendix A]{GN2023} for an example.

\begin{corollary}
  \label{cor:conv_Riemann_int}
 Suppose the assumptions of Theorem~\ref{thm:conv_Riemann} hold.
 Assume further that $\gappr=\OptInt[k] {\gex}\in\Regge_h^k$,
 for some integer $k\geq 0$ for $N=2$ and $k\geq 1$ for $N\geq 3$, 
 and $\OptInt[k]$ is an interpolation operator fulfilling~\eqref{eq:interp_cond}
 for some $p\in [2,\infty]$. Then there exists $h_0>0$ and $C>0$
 depending on $N$, $\Omega$, $k$, $C_0$, $C_1$, $C_2$, $\|\gex\|_{W^{2,\infty}(\Omega)}$, and $\|\gex^{-1}\|_{\Linf[\Omega]}$ 
 such that for all
$h\leq h_0$
\begin{align}
  \label{eq:cgce-rates} &\|\RoggA{\gappr}-(\Curvature\,\volform)(\gex)\|_{H^{-2}}\leq C\Big(\sum_{T\in\T_h}h_T^{p(k+1)}|\gex|^{{p}}_{W^{k+1,p}(T)}\Big)^{1/p}.
\end{align}
For $p=\infty$ the right-hand side is $C\max_{T \in \T_h}h_T^{k+1}|{\gex}|_{W^{k+1,\infty}(T)}$.
\end{corollary}

\begin{remark}[Lowest-order cases {of piecewise constant Regge metrics}]
 Note that in {two dimensions,} $N=2$, Corollary~\ref{cor:conv_Riemann_int} gives convergence already for the lowest-order case of piecewise constant Regge metrics, {$\gappr\in\Regge_h^0$.} In {higher dimensions, $N\geq 3$,} however, at least {a linear approximation of the metric is} needed, {$\gappr\in\Regge_h^k$ with $k\ge 1$,} to obtain norm convergence from \eqref{eq:cgce-rates}{, because the constant $C_g$ is $\mathcal{O}(h^{-1})$ for $k=0$.} In \S\ref{sec:num_examples} we will confirm this requirement through {numerical examples} in the $N=3$ case. It is consistent with the results in \cite[Theorem 4.1]{GN2023} and \cite[Corollary 4.3]{GN2023b} where, for dimensions greater than two, at least linear Regge elements were needed to obtain convergence for the densitized distributional scalar curvature and Einstein tensor, respectively. {In other words, the curvatures of Regge's original approximation do not generally converge in the $H^{-2}$-norm for dimensions $N\ge 3$, but converge in the sense of measures~\cite{Cheeger84}.} See also Remark~\ref{rem:lowest-order-a} below for further discussion of the $k=0$ case. For $N\le 3$, where the Sobolev embedding  $H^2_0(\Omega)\hookrightarrow C^0(\bar\Omega)$ holds, convergence in the sense of measures implies weak convergence in $H^{-2}$, but does not generally imply strong convergence in $H^{-2}$.
\end{remark}

  \begin{remark}[Piecewise flat approximations and Gromov's question]
  \label{rem:gromov-piecewise-flat}
  In the setting of Theorem~\ref{thm:conv_Riemann}, assume in addition that each $\gappr$ is a piecewise flat Regge metric, that $C_g$ is bounded independently of $h$, and that $\nrmt{\gappr-\gex}_2\to0$. Then
  \[
    \|\RoggA{\gappr}-(\Curvature\,\volform)(\gex)\|_{H^{-2}}\to0.
  \]
  If, moreover, the angle defects of $\gappr$ are positive in the sense of Remark~\ref{rem:pos_curv_op}, then $\RoggA{\gappr}$ is a positive distributional curvature operator. For piecewise flat metrics there are no element-interior curvature contributions and the second fundamental form vanishes on every facet. Thus such a family, in the topology of Theorem~\ref{thm:conv_Riemann}, confirms Gromov's suggestion quoted in Section~\ref{sec:intro} that positive angle defects are naturally interpreted as positive curvature operator. For $N\ge 3$, a sufficient way to ensure the boundedness of $C_g$ is
  \[
    \max_{T\in\T_h} h_T^{-2}\|\gappr-\gex\|_{L^\infty(T)}
    +
    \max_{T\in\T_h} h_T^{-1}\|\gappr-\gex\|_{W^{1,\infty}(T)}
    \le C.
  \]
  The strong metric approximations of Meyer--Toubiana~\cite{MeyerToubiana22}, the second-order local metric accuracy implicit in Cheeger--M\"uller--Schrader~\cite[pp.~414--415]{Cheeger84},
  and the metric approximation on Karcher simplices in   \cite[Theorem~2]{deylenDistortionEstimatesBarycentric2016}
  appear to provide natural candidates for such piecewise flat families. 
  Verifying the hypotheses above for these or other constructions
  is (likely nontrivial) additional work.
  Corollary~\ref{cor:conv_Riemann_int} only verifies the hypotheses for the
  specific Regge metric construction. 
\end{remark}

\begin{remark}[Applicability to embedded hypersurfaces]
 Let $M_h$ be a piecewise smooth $N$-dimensional hypersurface
 embedded in $\R^{N+1}$ approximating a smooth hypersurface. {Denote by}
  $\VV_h^k$ the degree~$k$ Lagrange finite element space {\eqref{eq:lag_fe}} on a
 mesh $\T_h$ of $\om$ {for $k\ge 1$}. Letting $\om\subset\R^{N}$ take the role
 of a parameter domain for the hypersurface, assume that
  $M_h = \Phi_h(\om)$ for some embedding
  $\Phi_h\in [\VV_h^k]^{N+1}$, {i.e., $N+1$ copies of $\VV_h^k$.  
 Then, the metric tensor induced by the embedding is $\gappr= (\grad \Phi_h)^T\grad \Phi_h$ with $\grad$ denoting the Euclidean Jacobian matrix. Lagrange finite elements are elementwise smooth and continuous over interfaces. Thus, their tangential derivatives are also continuous across interfaces. 
 Hence, $\gappr$} is $tt$-continuous, and thus defines a Regge metric in
  $\Regge_h^{2(k-1)}$. Therefore, Theorem~\ref{thm:conv_Riemann} and
 Corollary~\ref{cor:conv_Riemann_int} can be applied.
\end{remark}

\subsection{Roadmap of the proof}

To prove Theorem~\ref{thm:conv_Riemann} and Corollary~\ref{cor:conv_Riemann_int}, we use \eqref{eq:evolution_distr_Riemann} to rewrite the error as an integral representation. Namely, letting $\gpar(t):= \gex+(\gappr-\gex)t$ and $\sigma:=\gpar^\prime(t)=\gappr-\gex$, by the fundamental theorem of calculus and Theorem~\ref{thm:distr_Riem_evol}, we obtain the following integral representation of the error 
\begin{align*}
 \Big( \RoggA{\gappr}-(\Curvature\,\volform)(\gex)\Big)(U)&= \int_{0}^1\frac{d}{dt}\RoggA{\gpar(t)}(U)\,dt\\
  &= \int_0^1\left(a(\gpar(t);\sigma,U) + b(\gpar(t);\sigma, U)\right)\,dt,
\end{align*}
where $a(\cdot;\cdot,\cdot)$ and $b(\cdot;\cdot, \cdot)$ are defined as in \eqref{eq:def_ah} and \eqref{eq:def_bh}, respectively.

We estimate the bilinear form $a(\cdot;\cdot,\cdot)$ directly in Proposition~\ref{prop:conv_a_term} below, whereas for the bilinear form $b(\cdot;\cdot, \cdot)$, we use the relation of the distributional covariant incompatibility operator, \eqref{eq:rel_b_inc}, and estimate it via its adjoint, \eqref{eq:distr_adjoint_inc} of Theorem~\ref{thm:distr_covariant_adjoint_inc},
\begin{align*}
 b(\gpar(t);\sigma, U)=-2\,\widetilde{\INC\sigma}(B) &= 2\,(\widetilde{\div\div B})(\sigma),
\end{align*} 
with $B = S\mapUA_{\gpar(t)}(U)$, as done in Proposition~\ref{prop:conv_b_term} below. 
{The adjoint \eqref{eq:distr_adjoint_inc} puts all derivatives from $\sigma$ to $B$ (or equivalently $U$, which is in $H^2$). As $\sigma=\gdot(t)=\gappr-\gex$, we can then extract the best possible convergence rates from the estimates.} The estimates rely on the codimension-1 and codimension-2 trace inequalities~\eqref{eq:trace_inequ_standard}--\eqref{eq:trace_inequ_codim2} to handle the facet and bone contributions in~$a$ and~$b$.

\subsection{Basic estimates}
{We use $a \lesssim b$ to indicate that there is
an $h$-independent generic constant $C>0$, depending on $\om$ and the
shape-regularity {\eqref{eq:shape-regularity}} of the mesh $\T_h$, such that $a \le C b$.
The $C$ may additionally depend on
$\{\|\gex\|_{\Wtinf},\|\gex^{-1}\|_{\Linf}, N\}$.} We assume that the  approximation property $\lim_{h \to 0} \|\gappr-\gex\|_{L^\infty(\Omega)} = 0$ and stability estimate $\sup_{h>0} \max_{T \in \T_h} \|\gappr\|_{W^{2,\infty}(T)} < \infty$, both assumptions of Theorem~\ref{thm:conv_Riemann}, tacitly hold in the remainder of this section.  These assumptions have some elementary consequences {for $\gpar(t)=\gex+t(\gappr-\gex)$ that we quickly state.} 

 For every $h$ sufficiently small, every $t \in [0,1]$, and every vector $w$ with unit Euclidean length there {directly} holds
\begin{subequations}
\begin{align}
  \|\gpar(t)\|_{L^\infty} + \|\gpar(t)^{-1}\|_{L^\infty} &\lesssim 1,\qquad\quad \max_{T\in\T_h} |\gpar(t)|_{W^{2,\infty}(T)} \lesssim 1, \label{eq:gtbound} \\
  1 \lesssim \inf_{\Omega} (w^T \gpar(t) w) \le \sup_{\Omega} (w^T \gpar(t) w) &\lesssim 1, \label{eq:eigbound}
\end{align}
where we interpret $\gpar(t)$ as matrix and $w$ as  column vector in \eqref{eq:eigbound}.  Note that \eqref{eq:eigbound} implies the existence of positive lower and upper bounds on the inverse
\begin{equation*}
  1 \lesssim \inf_{\Omega} (w^T \gpar(t)^{-1} w) \le \sup_{\Omega} (w^T \gpar^{-1}(t) w) \lesssim 1.
\end{equation*}
In addition, the inequalities $\|\gpar(t)\|_{L^\infty} \lesssim 1$ and $\|\gpar^{-1}(t)\|_{L^\infty} \lesssim 1$ imply that
\begin{align} \label{eq:Lpequiv}
  \begin{split}
  & \|\rho\|_{L^p(D,\gpar(t_2))} \lesssim \|\rho\|_{L^p(D,\gpar(t_1))} \lesssim \|\rho\|_{L^p(D,\gpar(t_2))},\\
  & \|\rho\|_{L^p(D)} \lesssim \|\rho\|_{L^p(D,\gpar(t_1))} \lesssim  \|\rho\|_{L^p(D)}
  \end{split}
\end{align}
\end{subequations}
for every $t_1,t_2 \in [0,1]$, every submanifold $D$ {on which the induced metric $g|_D$ is well-defined,} every $p \in [1,\infty]$, every tensor field $\rho$ having finite $L^p(D)$-norm {\eqref{eq:lp-norm},} and every $h$ sufficiently small.  We select $h_0>0$ so that~(\eqref{eq:gtbound}--\eqref{eq:Lpequiv}) hold for all $h \le h_0$, and we tacitly use these inequalities throughout our analysis.

The following lemma shows how the difference of covariant derivatives is related to the difference of the underlying metric tensor. As the covariant derivative depends on the first derivative of the metric, it is not surprising that the $k$-th covariant derivative can be estimated by the $k$-th derivative of the metric, which can be proved via Christoffel symbols or the Koszul formulas. For completeness, we present a rigorous proof. We recall that all fields are defined on $\Omega\subset\R^N$.
\begin{lemma}
  \label{lem:est_T}
  Let $D\in \{T,F,E\}$ and $S_g\in\TT^r_s(D)$ depend on a Regge metric $g$. Assume that $S_g=\mathbb{G}_gV$, where $V\in\TT^{i}_j(D)$ is independent of $g$ and sufficiently smooth, and $\mathbb{G}_g:=\mathbb{G}(g)$ with $\mathbb{G}:\Sc^+(\T)\to \mathrm{Hom}(\TT^i_{j}(D),\TT^r_s(D))$ is a piecewise smooth function depending pointwise on $g$ but not on its derivatives, i.e.,
    $\mathbb{G}_g(x)$ depends only on $g(x)$ at each $x \in \om$.
 Let $\gpar(t)=\gex+(\gappr-\gex)t$. Then for $t\in [0,1]$ there holds for $p\in [1,\infty]$
  \begin{subequations}
    \begin{align}
      \|\nabla^{\ell}_{\gpar(t)} S_{\gpar(t)}\|_{L^p(D)} &\lesssim \|V\|_{W^{\ell,p}(D)},\quad &&\ell\in \{0,1,2\}, \label{eq:est_grad_S}\\
      \|\nabla^{\ell}_{\gpar(t)} S_{\gpar(t)}-\nabla^{\ell}_{\gex} S_{\gex}\|_{L^p(D)} &\lesssim \|\gappr-\gex\|_{W^{\ell,\infty}(D)}\|V\|_{W^{\ell,p}(D)}, \quad &&\ell\in \{0,1,2\}. \label{eq:est_grad_S_diff}
      \end{align} 
 If $S_{g}$ allows for a well-defined jump across a facet $F\in\Fint$, cf. \eqref{eq:jmp-tensor},
      \begin{align}
      \|\jmp{\nabla^{\ell}_{\gpar(t)} S_{\gpar(t)}}\|_{L^p(F)} &\lesssim \|\jmp{\gappr-\gex}\|_{W^{\ell,\infty}(F)}\|V\|_{W^{\ell,p}(F)}, \quad &&\ell\in \{0,1\}. \label{eq:est_grad_S_diff_F}
    \end{align}  
  \end{subequations}
 When $\gappr\in \Regge_h^0$, i.e. $\gappr$ is piecewise constant, then $\|\jmp{\gappr-\gex}\|_{W^{\ell,\infty}(F)}$ can be replaced by $\|\jmp{\gappr-\gex}\|_{\Linf[F]}$ in the estimate \eqref{eq:est_grad_S_diff_F}.
\end{lemma}
\begin{proof}
  The estimate of~\eqref{eq:est_grad_S} follows from \eqref{eq:est_grad_S_diff} by the triangle inequality
  \[
    \|\nabla^\ell_{\gpar(t)}S_{\gpar(t)}\|_{L^p(D)}
    \le \|\nabla^\ell_{\gpar(t)}S_{\gpar(t)}-\nabla^\ell_{\gex}S_{\gex}\|_{L^p(D)}
      + \|\nabla^\ell_{\gex}S_{\gex}\|_{L^p(D)}.
  \]
  The first term is bounded by \eqref{eq:est_grad_S_diff} and $\|\gappr-\gex\|_{W^{\ell,\infty}(D)}\lesssim 1$, which holds by the stability assumption~$C_1<\infty$.
  The second term satisfies $\|\nabla^\ell_{\gex}S_{\gex}\|_{L^p(D)}\lesssim\|V\|_{W^{\ell,p}(D)}$ with a constant depending on $\|\gex\|_{W^{2,\infty}(\Omega)}$ and $\|\gex^{-1}\|_{\Linf[\Omega]}$, both absorbed into the generic constant~$C$. For \eqref{eq:est_grad_S_diff}, we start with $\ell=0$. By the multiplicative decomposition of $S_g$, we have
  \begin{align*}
    \|S_{\gpar(t)}-S_{\gex}\|_{L^p(D)} \lesssim \|\mathbb{G}(\gpar(t))-\mathbb{G}(\gex)\|_{\Linf[D]}\|V\|_{L^p(D)}.
  \end{align*} 
 By Taylor's theorem and the definition of $\gpar(t)$, there exists an $s\in [0,t]$ such that
  \begin{align*}
    \mathbb{G}(\gpar(t))-\mathbb{G}(\gex) = t\,D\mathbb{G}(\gpar(s))\cdot(\gappr-\gex),
  \end{align*} 
 where $D$ denotes the Fr\'echet derivative. By our assumptions on $\gappr$ and $\gex$ the Fr\'echet derivative is uniformly bounded, i.e. $\|D\mathbb{G}(\gpar(s))\|_{W^{2,\infty}(D)} \lesssim 1$. Thus,
  \begin{align*}
    \|S_{\gpar(t)}-S_{\gex}\|_{L^p(D)} \lesssim \|\gappr-\gex\|_{\Linf[D]}\|V\|_{L^p(D)}.
  \end{align*} 
 Next, we consider $\ell=1$. Assume for the moment that $S_{\gpar(t)}\in \W^1(D)$. Then
  \begin{align}
    \label{eq:est_grad_S_diff_1}
 (\nabla_{\gpar(t)} S_{\gpar(t)}-\nabla_{\gex} S_{\gex})_{ij} &= \d_i(S_{\gpar(t)}-S_{\gex})_j -\Gamma_{ij}^k(\gpar(t)) S_{\gpar(t),k}+\Gamma_{ij}^k(\gex)S_{\gex,k}.
  \end{align}
 By the splitting of $S_g$, the product rule, and the boundedness of the Fr\'echet derivative, we can bound the first term in \eqref{eq:est_grad_S_diff_1} as follows
  \begin{align*}
    \|\d_i(S_{\gpar(t)}-S_{\gex})_j\|_{L^p(D)} &\lesssim \|D\mathbb{G}(\gpar(s))\|_{W^{1,\infty}(D)}\|\gappr-\gex\|_{\Winf[D]}\|V\|_{W^{1,p}(D)}.
  \end{align*}
 Using the definition of $\gpar(t)$, we regroup the second term in \eqref{eq:est_grad_S_diff_1}
  \begin{align*}
    -\Gamma_{ij}^k(\gpar(t)) S_{\gpar(t),k}+\Gamma_{ij}^k(\gex)S_{\gex,k} &= -\left(\gpar(t)-\gex\right)^{k\ell}\Gamma_{ij\ell}(\gpar(t)) S_{\gpar(t),k} \\
    &-\gex^{k\ell}\Gamma_{ij\ell}(\gex)(S_{\gpar(t)}-S_{\gex})_k -\gex^{k\ell}(\Gamma_{ij\ell}(\gpar(t))-\Gamma_{ij\ell}(\gex))S_{\gpar(t),k}.
  \end{align*}
 The first two terms on the right-hand side can be estimated directly. As the Christoffel symbols of the first kind depend linearly on the derivative of the metric, we have
  \begin{align*}
    \|\gex^{k\ell}(\Gamma_{ij\ell}(\gpar(t))-\Gamma_{ij\ell}(\gex))S_{\gpar(t),k}\|_{L^p(D)}
    &\lesssim \|\gappr-\gex\|_{\Winf[D]}\|V\|_{L^p(D)}.
  \end{align*}
 Combining all the estimates above, we obtain
  \begin{align*}
    \|\nabla_{\gpar(t)} S_{\gpar(t)}-\nabla_{\gex} S_{\gex}\|_{L^p(D)} &\lesssim \|\gappr-\gex\|_{\Winf[D]}\|V\|_{W^{1,p}(D)}.
  \end{align*}
 The calculations are analogous for general tensors $S_g\in \TT^r_s(D)$, only more Christoffel symbols appear. The case $\ell=2$ is similar, and therefore omitted.
  
 For \eqref{eq:est_grad_S_diff_F}, we can add the tensor depending on the smooth metric into the jump, i.e., we can write
  \begin{align*}
    \|\jmp{\nabla^{\ell}_{\gpar(t)} S_{\gpar(t)}}\|_{L^p(F)} = \|\jmp{\nabla^{\ell}_{\gpar(t)} S_{\gpar(t)}-\nabla^{\ell}_{\gex}S_{\gex}}\|_{L^p(F)}
  \end{align*}
 Then using (iteratively) that $\jmp{ab}=\jmp{a}\{\!\{b\}\!\}+\jmp{b}\{\!\{a\}\!\}$, where $\{\!\{c\}\!\}=\tfrac12(c|_{T^+}+c|_{T^-})$ denotes the mean value across $F$, we can follow the same steps as above.
\end{proof}

\begin{lemma} \label{lemma:ndiff}
 Let $g_1$ and $g_2$ be two symmetric positive-definite matrices, and let $\nv$ be a Euclidean unit vector.  {Writing $g^{\nv\nv}_i=g_i^{ab}\nv_a\nv_b$, set}
  \[
  \gn_{g_{i}} = \frac{1}{\sqrt{g_{i}^{\nv\nv}}} g_{i}^{-1} \nv, \quad i=1,2.
  \]
 Then there exists a constant $c>0$ depending on the Euclidean norms $|g_1|$, $|g_2|$, $|g_1^{-1}|$, $|g_2^{-1}|$ such that
  \[
 |\gn_{g_1} - \gn_{g_2}| \le c|g_1-g_2|.
  \]
\end{lemma}
\begin{proof}
{Follows from \eqref{eq:est_grad_S_diff} with $\mathbb{G} = \frac{1}{\sqrt{g^{\nv\nv}}} g^{-1}$ and $V=\nv$ and $\ell$$=0$.
 }
\end{proof}

As a preparation, we estimate some geometric quantities that frequently arise.
\begin{lemma}
  \label{lem:est_curv_quant}
 Let $D\in \{T,F,E\}$ be a volume, codimension-1, or codimension-2 {subsimplex} of $\T_h$, $\mapUA_{\gpar(t)}$ the mapping defined in \eqref{eq:mapping_U_A}, and $\gpar(t)= \gex+(\gappr-\gex)t$. There holds for all $t\in [0,1]$ and $p\in [1,\infty]$
  \begin{gather*}
    \|\mapUA_{\gpar(t)}(U)\|_{L^p(D)}\lesssim \|U\|_{L^p(D)},\\
    \|\mapUA_{\gpar(t)}(U)-\mapUA_{\gex}(U)\|_{L^p(D)}\lesssim \|\gex-\gappr\|_{\Linf[D]}\|U\|_{L^p(D)},\\
    {\|\vol[\gpar(t)]{D}\|_{\Linf[D]}\lesssim 1,\qquad\|\Riemann_{\gpar(t)}\|_{\Linf[T]}\lesssim 1,\qquad\|\gn_{\gpar(t)}\|_{\Linf[F]}\lesssim 1.}
  \end{gather*}
\end{lemma}
\begin{proof}
 Follows directly by the assumptions on $\gex$ and $\gappr$ {and Lemma~\ref{lem:est_T} with $\mathbb{G}= \mapUA$.}
\end{proof}
The following estimates on boundary facets are crucial for the analysis.
\begin{lemma}
  \label{lem:appr_Weingarten_jmp}
 Let $U\in H^2_0(\Omega,\R^{\tilde{N}\times\tilde{N}}_{\mathrm{sym}})$, $\gpar(t)= \gex+(\gappr-\gex)t$, and $A_{\gpar(t)}=\mapUA_{\gpar(t)}U$. There holds for all $F\in\Fint_h$ and $t\in [0,1]$
  \begin{align*}
    \|\jmp{\sff_{\gpar(t)}}\|_{\Linf[F]}&\lesssim \|\jmp{\gappr-\gex}\|_{\Winf[F]},\\
    \|\jmp{(\div_{\gpar(t)} {S}A_{\gpar(t)}+\divFR_{\gpar(t)}SA_{\gpar(t)})_{\gn_{\gpar(t)}}}\|_{\Ltwo[F]}&\lesssim {\|\jmp{\gex-\gappr}\|_{\Winf[F]}\|U\|_{\Hone[F]}.}
  \end{align*}
If $\gappr$ is piecewise constant, the $\Winf[F]$-norm in both inequalities can be replaced by the $\Linf[F]$-norm.
\end{lemma}
\begin{proof}
  {The first claim follows from \eqref{eq:est_grad_S_diff_F} with $\mathbb{G} = \frac{1}{\sqrt{g^{\nv\nv}}} g^{-1}$ and $V=\nv$ and $\ell$$=1$, noting that $\|\nv\|_{W^{1,\infty}(F)}=1$ for the Euclidean unit normal $\nv$.}
  
{To prove the second claim, we first note that 
\begin{align*}
  \|\jmp{(\div_{\gpar(t)} {S}A_{\gpar(t)}+\divFR_{\gpar(t)}SA_{\gpar(t)})_{\gn_{\gpar(t)}}}\|_{\Ltwo[F]} \lesssim \| \jmp{\nabla_{\gpar(t)} SA_{\gpar(t)}}\|_{\Ltwo[F]}.
\end{align*}
Using \eqref{eq:est_grad_S_diff_F} with $\mathbb{G} = S\circ \mapUA$ and $V=U$ and $\ell$$=1$ the proof is completed.
}

\end{proof}

\subsection{Estimating bilinear form $a$}
We start by estimating the terms involved in \eqref{eq:def_ah}. Note that the appearing inner products have to be understood with respect to the metric $\gpar(t)= \gex+(\gappr-\gex)t$ and recall that $\sigma = \gappr-\gex$.
\begin{lemma}
  \label{lem:est_a_vol}
 There holds for the volume term of \eqref{eq:def_ah} for all $t\in [0,1]$
  \begin{align*}
 \Big|\sum_{T\in\T_h}\int_T\big(
    \langle \Lsigma^{(1)}\Riemann_{\gpar(t)},A_{\gpar(t)}\rangle-\frac{1}{2}\tr[\gpar(t)]{\sigma}\langle\Riemann_{\gpar(t)}&,A_{\gpar(t)}\rangle\big)\vol[\gpar(t)]{T}\Big|\!\lesssim\! \|\gex-\gappr\|_{\Ltwo[]}\|U\|_{\Ltwo[]},
  \end{align*}
where $\Lsigma^{(1)}$, cf. \eqref{eq:Lsig-defn}, has to be understood with respect to $\gpar(t)$.
\end{lemma}
\begin{proof}
 Follows directly by H\"older's inequality and Lemma~\ref{lem:est_curv_quant}, e.g.,
  \begin{align*}
 \left|\int_T
  \langle \Lsigma^{(1)}\Riemann_{\gpar(t)},A_{\gpar(t)}\rangle\,\vol[\gpar(t)]{T}\right|&\lesssim \|\Riemann_{\gpar(t)}\|_{\Linf[T]}\|\sigma\|_{\Ltwo[T]}\|A_{\gpar(t)}\|_{\Ltwo[T]}\\
  &\lesssim \|\sigma\|_{\Ltwo[T]}\|U\|_{\Ltwo[T]}.
  \end{align*}
\end{proof}
\begin{lemma}
  \label{lem:est_a_bnd}
 There holds for the codimension-1 term of \eqref{eq:def_ah} for all $t\in [0,1]$
  \begin{align*}
\begin{split}
 \left|\sum_{F \in \Fint_h}\int_F \jmp{\sff_{\gpar(t)}}:\mathbb{S}_{\Fr}(\sigma):A_{\gpar(t),F\gn_{\gpar(t)}\gn_{\gpar(t)}F}\,\vol[\gpar(t)]{F}\right|&\lesssim  \max_{T \in \T_h}\left(h_T^{-1}\|\gappr-\gex\|_{\Winf[T]}\right)\\
      &\qquad \times \nrm{\gappr-\gex}_2\nrm{U}_2.
\end{split}
  \end{align*}
 If $\gappr$ is piecewise constant, the $\Winf[T]$-norm can be replaced by the $\Linf[T]$-norm.
\end{lemma}
\begin{proof}
 Using Lemma~\ref{lem:est_curv_quant} and Lemma~\ref{lem:appr_Weingarten_jmp} there holds noting that $\|\mathbb{S}_{\Fr}(\sigma)\|_{\Ltwo[F]}\lesssim \|\sigma_{\Fr}\|_{\Ltwo[F]}$ 
  \begin{align*}
 \left|\! \int_F\jmp{\sff_{\gpar(t)}}\!:\!\mathbb{S}_{\Fr}(\sigma)\!:\!A_{\gpar(t),F\gn_{\gpar(t)}\gn_{\gpar(t)}F}\,\vol[\gpar(t)]{F}\right|&\!\lesssim \!\|\mathbb{S}_{\Fr}(\sigma)\|_{\Ltwo[F]}\|\|\jmp{\sff_{\gpar(t)}}\|_{\Linf[F]}\|U\|_{\Ltwo[F]}\\
    &\!\lesssim\! \|\sigma_{\Fr}\|_{\Ltwo[F]}\|\jmp{\gappr-\gex}\|_{\Winf[F]}\|U\|_{\Ltwo[F]}.
  \end{align*}
With the trace inequality {\eqref{eq:trace_inequ_standard}}
we get
\begin{align*}
&\|\sigma_{\Fr}\|_{\Ltwo[F]}^2\|\jmp{\gappr-\gex}\|_{\Winf[F]}^2\|U\|_{\Ltwo[F]}^2\\
&\lesssim\!{\Big(\sum_{i=1}^2 \|\gappr-\gex\|_{W^{1,\infty}(T_i)}^2\Big)  h_{T_1}^{-2}( \|\sigma\|_{L^2(T_1)}^2 \!+\! h^2_{T_1} |\sigma|_{H^1(T_1)}^2 ) ( \|U\|_{L^2(T_1)}^2 \!+\! h^2_{T_1} |U|_{H^1(T_1)}^2 ).}
\end{align*}
Here $T_1$ and $T_2$ are the two elements sharing $F$. The sum over $i=1,2$ in the first factor arises because $\|\jmp{\gappr-\gex}\|_{W^{1,\infty}(F)}$ involves traces from both sides. The trace inequality~\eqref{eq:trace_inequ_standard} for $\|\sigma_{\Fr}\|_{L^2(F)}$ and $\|U\|_{L^2(F)}$ is applied using $T_1$ (since the use of analogous bound from $T_2$ gives the same estimate up to the shape-regularity).
Due to the shape-regularity of $\T_h$, we have $C^{-1} \le h_{T_1}/h_{T_2} \le C$ for some constant $C$ independent of $h$ and $F$, and thus
\begin{align*}
 \left| \int_F\jmp{\sff_{\gpar(t)}}:\mathbb{S}_{\Fr}(\sigma):A_{\gpar(t),F\gn_{\gpar(t)}\gn_{\gpar(t)}F}\,\vol[\gpar(t)]{F}\right|&\lesssim \max_{T \in \T_h}\left(h^{-1}_T\|\gappr-{\gex}\|_{\Winf[T]}\right)\\
  &\quad\times\nrm{\gappr-\gex}_2\nrm{U}_2.
\end{align*}
If $\gappr$ is piecewise constant, there holds by the smoothness of the exact metric $\gex$ that \\ $\|\jmp{\gappr-\gex}\|_{\Winf[F]}=\|\jmp{\gappr-\gex}\|_{\Linf[F]}$.
\end{proof}
\begin{lemma}
  \label{lem:est_a_bbnd}
 There holds for the codimension-2 term of \eqref{eq:def_ah} for all $t\in [0,1]$
  \begin{align*}
    &\left|\sum_{E \in \Eint_h}\int_E\tr[\gpar(t)]{\sigma_E}\Theta_E(\gpar(t))(A_{\gpar(t)})_{\gcn_{\gpar(t)}\gn_{\gpar(t)}\gn_{\gpar(t)}\gcn_{\gpar(t)}}\,\vol[\gpar(t)]{E}\right|\\
    &\qquad\lesssim \left(\max_{T \in \T_h}h_T^{-2} \|\gappr-\gex\|_{\Linf[T]}\right) \nrmt{\gappr-\gex}_2\nrmt{U}_2.
  \end{align*}
\end{lemma}
\begin{proof}
  {We follow the idea of \cite[Lemma 4.13]{GN2023}. First, we rewrite the difference of the angle defects
  \begin{align*}
    \cos\sphericalangle_{E}^T(\gex)-\cos\sphericalangle_{E}^T(\gpar(t))
    &= \gpar(t)(\gn_{\gpar(t)}^{(1)}, \gn_{\gpar(t)}^{(2)}) - \gex(\gn_{\gex}^{(1)}, \gn_{\gex}^{(2)}) \\
    &= \gpar(t)( \gn_{\gpar(t)}^{(1)} - \gn_{\gex}^{(1)}, \gn_{\gpar(t)}^{(2)} - \gn_{\gex}^{(2)} ) + \gpar(t)( \gn_{\gpar(t)}^{(1)} - \gn_{\gex}^{(1)}, \gn_{\gex}^{(2)} ) \\
    & + \gpar(t)( \gn_{\gex}^{(1)}, \gn_{\gpar(t)}^{(2)} - \gn_{\gex}^{(2)} ) + \gpar(t)(\gn_{\gex}^{(1)},\gn_{\gex}^{(2)}) - \gex(\gn_{\gex}^{(1)},\gn_{\gex}^{(2)}),
  \end{align*}
 to obtain with Lemma~\ref{lemma:ndiff} 
  \begin{align*}
    \|\sphericalangle_{E}^T(\gex)-\sphericalangle_{E}^T(\gpar(t))\|_{\Linf[E]}\lesssim \|\gpar(t)-\gex\|_{\Linf[E]}\lesssim \|\gappr-\gex\|_{\Linf[E]}.
  \end{align*}
 Thus, using the fact that the angle defect is zero for the exact smooth metric $\gex$, we have
  \begin{align*}
    \|\Theta_E(\gpar(t))\|_{\Linf[E]} &= \|\Theta_E(\gpar(t))-\Theta_E(\gex)\|_{\Linf[E]} \\
    &\lesssim \sum_{T\supset E}\|\sphericalangle_{E}^T(\gex)-\sphericalangle_{E}^T(\gpar(t))\|_{\Linf[E]} \lesssim \|\gappr-\gex\|_{\Linf[T]}.
  \end{align*}
 Using H\"older's inequality and the codimension-2 trace inequality \eqref{eq:trace_inequ_codim2} we obtain for each $E \in \Eint_h$ that
  \begin{align*}
    &\left|\int_E\tr[\gpar(t)]{\sigma_E}\Theta_E(\gpar(t))(A_{\gpar(t)})_{\gcn_{\gpar(t)}\gn_{\gpar(t)}\gn_{\gpar(t)}\gcn_{\gpar(t)}}\,\vol[\gpar(t)]{E}\right|\\
    &
    \lesssim \|\Theta_E(\gpar(t))\|_{\Linf[E]}\|\sigma_E\|_{\Ltwo[E]}\|U\|_{\Ltwo[E]}\\
    &\lesssim \sum_{i=1}^m\|\gappr-\gex\|_{\Linf[T_i]}h_{T_1}^{-2}\nrmt{\sigma}_{2,{T_1}}\nrmt{U}_{2,{T_1}}.
  \end{align*}
 Due to the shape-regularity of $\T_h$, the number $m$ of adjacent elements $T$ to an edge $E$ is uniformly bounded, and the claim follows by summing over all edges $E\in\Eint_h$.
 }
\end{proof}
\begin{proposition}
  \label{prop:conv_a_term}
 Let $\gpar(t)= \gex+(\gappr-\gex)t$, $\sigma = \gappr-\gex$, and $U\in\Htwoz[\Omega,\R^{\tilde{N}\times\tilde{N}}_{\mathrm{sym}}]$. There holds for all $t\in[0,1]$
  \begin{align*}
 \left|a(\gpar(t);\sigma,U)\right|&\lesssim \left(1 + \max_{T \in \T_h}h_T^{-1} \|\gappr-\gex\|_{\Winf[T]}+\max_{T \in \T_h}h_T^{-2} \|\gappr-\gex\|_{\Linf[T]} \right)\\
          &\qquad\times\nrmt{\gappr-\gex}_2\|U\|_{\Htwo}.
  \end{align*}
 Assume that $\gappr=\OptInt[k]\gex$ is an interpolant fulfilling assumptions \eqref{eq:interp_cond} with $p\in [2,\infty]$. Then for $k\geq 1$
  \begin{align*}
 \big|a(\gpar(t);\sigma,U)\big|\lesssim \left(\sum_{T\in\T_h}h_T^{p(k+1)}|\gex|^{{p}}_{W^{k+1,p}(T)}\right)^{1/p}\|U\|_{\Htwo}.
  \end{align*}
\end{proposition}
\begin{proof}
Combine Lemmas~\ref{lem:est_a_vol}--\ref{lem:est_a_bbnd} and the approximation assumptions \eqref{eq:interp_cond} on the interpolation operator together with
\begin{align}
  \label{eq:trafo_2_p}
\begin{split} 
    &\|\gappr-\gex\|_{\Ltwo[\Omega]}\leq |\Omega|^{\frac{1}{2}-\frac{1}{p}}\|\gappr-\gex\|_{L^p(\Omega)},\\
    &\left(\sum_{T\in \T_h}h^2_T|\gappr-\gex|^2_{\Hone[T]}\right)^{\frac{1}{2}}\leq |\Omega|^{\frac{1}{2}-\frac{1}{p}}\left(\sum_{T\in \T_h}h_T^p|\gappr-\gex|^p_{W^{1,p}(\Omega)}\right)^{\frac{1}{p}},\\
    &\left(\sum_{T\in \T_h}h^4_T|\gappr-\gex|^2_{\Htwo[T]}\right)^{\frac{1}{2}}\leq |\Omega|^{\frac{1}{2}-\frac{1}{p}}\left(\sum_{T\in \T_h}h_T^{2p}|\gappr-\gex|^p_{W^{2,p}(\Omega)}\right)^{\frac{1}{p}},
\end{split} 
\end{align}
for $p\in [2,\infty)$ and standard modifications for $p=\infty$.
\end{proof}

\begin{remark}
  \label{rem:lowest-order-a}
  {Proposition~\ref{prop:conv_a_term} shows optimal $\Hmtwo$-convergence for degrees $k\geq 1$. In contrast, for the lowest-order $k=0$ (piecewise constant metrics), the estimate of the codimension-2 contribution in Lemma~\ref{lem:est_a_bbnd}—which relies on a codimension-2 trace inequality—becomes the rate-limiting term and prevents optimal convergence. The same behavior occurs for the scalar curvature and Einstein tensor in \cite{GN2023, GN2023b} for $N\geq3$, where the analogous codimension-2 term leads to suboptimal rates. Our numerical computations in \S\ref{sec:num_examples} confirm that Lemma~\ref{lem:est_a_bbnd}, and hence Proposition~\ref{prop:conv_a_term}, are sharp. When $N=2$, $a(\cdot;\cdot,\cdot)$ vanishes; see \S\ref{sec:spec_2d}.}
  For $N\ge 3$, the obstruction in Corollary~\ref{cor:conv_Riemann_int} is specific to the generic lowest-order finite element interpolation estimates encoded in~\eqref{eq:interp_cond}. With $k=0$ these estimates provide only first-order $L^\infty$ metric accuracy, $\|\OptInt[0]\gex - \gex\|_{L^\infty(T)} = \mathcal{O}(h)$, and therefore $h_T^{-2}\|\gappr - \gex\|_{L^\infty(T)} = \mathcal{O}(h^{-1})$. This should not be read as a negative statement about every possible piecewise flat approximation. Special constructions, such as the ``strong metric approximation'' of \cite[Condition~(b) of Definition~6.4 and Theorem~6.6]{MeyerToubiana22} using Riemannian barycentric simplices, may achieve second-order $L^\infty$ accuracy even though the resulting metric is piecewise flat. The Karcher-simplex estimates of von Deylen--Glickenstein--Wardetzky~\cite[Theorem~2]{deylenDistortionEstimatesBarycentric2016} likewise give second-order metric accuracy, together with the corresponding first-derivative estimate. Such an approximation is not covered by Corollary~\ref{cor:conv_Riemann_int}, whose purpose is to record what follows from standard finite element interpolation assumptions. Instead, one should apply Theorem~\ref{thm:conv_Riemann} directly, after verifying its hypotheses, in particular the boundedness of $C_g$. When Theorem~\ref{thm:conv_Riemann} is applied to a piecewise flat Regge metric, the second fundamental form vanishes on every facet (each element being flat), so the distributional Riemann curvature has no codimension-1 contribution.
\end{remark}

\subsection{Estimating bilinear form $b$}
Next, we estimate the terms involved in \eqref{eq:def_bh}, or, more precisely, its adjoint \eqref{eq:distr_adjoint_inc} of Theorem~\ref{thm:distr_covariant_adjoint_inc} 
\begin{align*}
 b(\gpar(t);\sigma, U)=-2\,\widetilde{\INC\sigma}(B) &= 2\,(\widetilde{\div\div B})(\sigma),\qquad B = S\mapUA_{\gpar(t)}(U),
\end{align*} 
starting with the volume terms.
\begin{lemma}
  \label{lem:est_b_vol}
 There holds for the volume term of \eqref{eq:distr_adjoint_inc} for all $t\in [0,1]$
  \begin{align*}
 \left|\sum_{T\in\T_h}\int_T\langle \sigma,\div_{\gpar(t)}\div_{\gpar(t)} SA_{\gpar(t)}\rangle \,\vol[\gpar(t)]{T}\right|\lesssim \|\sigma\|_{\Ltwo[\Omega]}\|U\|_{\Htwo[\Omega]}.
  \end{align*}
\end{lemma}
\begin{proof}
 Follows by H\"older's inequality and Lemma~\ref{lem:est_curv_quant}.
\end{proof}

\begin{lemma}
  \label{lem:est_b_bnd}
 There holds for the codimension-1 term of \eqref{eq:distr_adjoint_inc} for all $t\in [0,1]$
  \begin{align}
 \Big|\sum_{F\in\Fint_h}&\int_{F}\big\llbracket\langle\sigma_{\Fr},(\div_{\gpar(t)} SA_{\gpar(t)}+\divFR_{\gpar(t)}SA_{\gpar(t)})_{\gn_{\gpar(t)}}\rangle\nonumber\\
      &\qquad + \sigma_{\Fr}:\overline{\sff}_{\gpar(t)}:(SA)_{{\gpar(t)},\gn_{\gpar(t)}\gn_{\gpar(t)}}- \langle\sff_{\gpar(t)}\otimes\sigma_{\Fr},SA_{\gpar(t)}\rangle\big\rrbracket\,\vol[\gpar(t)]{F}\Big|\label{eq:est_b_bnd}\\
    \lesssim
      \nrm{\sigma}_2&{ \max_{T \in \T_h}(h_T^{-1}\|\gex-\gappr\|_{\Winf[T]})\big(\nrm{U}_2+\nrm{\nabla_{\Eucl} U}_2\big)}\nonumber,
  \end{align}
where $\nabla_{\Eucl} U$ denotes the Euclidean gradient instead of the covariant one.
If $\gappr$ is piecewise constant, the $\Winf[T]$-norm can be replaced by the $\Linf[T]$-norm.
\end{lemma}
\begin{proof}
 The last two terms of \eqref{eq:est_b_bnd} involving the second fundamental form can be estimated as in the proof of Lemma~\ref{lem:est_a_bnd}. For the first term we use H\"older's inequality and Lemma~\ref{lem:appr_Weingarten_jmp}
\begin{align*}
  &\left|\int_F\langle\sigma_{\Fr},\jmp{(\div_{\gpar(t)} SA_{\gpar(t)}+\divFR_{\gpar(t)}SA_{\gpar(t)})_{\gn_{\gpar(t)}}}\rangle\,\vol[\gpar(t)]{F}\right| \\
  &\qquad\lesssim\|\sigma_{\Fr}\|_{\Ltwo[F]}\|\jmp{(\div_{\gpar(t)} SA_{\gpar(t)}+\divFR_{\gpar(t)}SA_{\gpar(t)})_{\gn_{\gpar(t)}}}\|_{\Ltwo[F]}\\
  &\qquad\lesssim \|\sigma_{\Fr}\|_{\Ltwo[F]}{\|\jmp{\gex-\gappr}\|_{\Winf[F]}\|U\|_{\Hone[F]}}.
\end{align*}
Applying trace inequality {\eqref{eq:trace_inequ_standard}} and using the shape-regularity, as in the proof of Lemma~\ref{lem:est_a_bnd}, gives the desired result
\begin{align*}
&\left|\sum_{F \in \Fint_h}\int_F\langle\sigma_{\Fr},\jmp{(\div_{\gpar(t)} SA_{\gpar(t)}+\divFR_{\gpar(t)}SA_{\gpar(t)})_{\gn_{\gpar(t)}}}\rangle\,\vol[\gpar(t)]{F}\right| \\
&\qquad\lesssim \nrm{\sigma}_2{ \max_{T \in \T_h}(h_T^{-1}\|\gex-\gappr\|_{\Winf[T]})\big(\nrm{U}_2+\nrm{\nabla_{\Eucl} U}_2\big)}.
\end{align*}
\end{proof}

The codimension-2 term of \eqref{eq:distr_adjoint_inc} is zero for dimension $N=2$, since $\sigma_E=0$ for 0-dimensional $E$. In higher spatial dimensions, it has to be considered, leading to a lower convergence rate than Lemma~\ref{lem:est_b_vol} and Lemma~\ref{lem:est_b_bnd}, comparable to the result of Lemma~\ref{lem:est_a_bbnd} for bilinear form $a(\cdot;\cdot,\cdot)$.
\begin{lemma}
  \label{lem:est_b_bbnd}
 There holds for the codimension-2 term of \eqref{eq:distr_adjoint_inc} for all $t\in [0,1]$
  \begin{align*}
      &\left|\sum_{E\in\Eint_h}\sum_{F\supset E}\int_E\langle\sigma_E,\jmp{(SA_{\gpar(t)})_{\gn_{\gpar(t)}\gcn_{\gpar(t)}EE}}^E_F\rangle\,\vol[\gpar(t)]{E}\right|\\
      &\hspace*{2cm}\lesssim \max_{T\in\T_h}\left(h_{T}^{-2}\|g_h-{\gex}\|_{\Linf[T]}\right) \nrmt{\sigma}_2\nrmt{U}_2.
  \end{align*}
\end{lemma}

\begin{proof}
 By the shape-regularity of $\T_h$, the number of facets attached to $E\in\Eint$ is bounded by a constant $C$ independent of $h$. Using that $\sum_{F\supset E}\jmp{ (SA_{\gex})_{\gn_{\gex}\gcn_{\gex}EE}}^E_F=0$ for smooth $\gex$ and $U\in H^2_0(\Omega,\R^{\tilde{N}\times \tilde{N}}_{\mathrm{sym}})$ there holds
\begin{align*}
&\left\| \sum_{F \supset E} \langle\sigma_E,\jmp{(SA_{\gpar(t)})_{\gn_{\gpar(t)}\gcn_{\gpar(t)}EE}}^E_F\rangle\right\|_{L^1(E)} \\
&= \left\| \sum_{F \supset E} \langle\sigma_E,\jmp{(SA_{\gpar(t)})_{\gn_{\gpar(t)}\gcn_{\gpar(t)}EE}-(SA_{\gex})_{\gn_{\gex}\gcn_{\gex}EE}}^E_F\rangle\right\|_{L^1(E)}\\
&\lesssim \|\gappr-\gex\|_{\Linf[E]}\|\sigma_E\|_{\Ltwo[E]}\|U\|_{\Ltwo[E]}
\end{align*}
and {thus}
\begin{align*}
 \left|\int_E \sum_{F \supset E} \langle\sigma_E,\jmp{(SA_{\gpar(t)})_{\gn_{\gpar(t)}\gcn_{\gpar(t)}EE}}^E_F\rangle\,\vol[\gpar(t)]{E}\right|\lesssim \|g_h-{\gex}\|_{\Linf[E]}\|U\|_{\Ltwo[E]}\|\sigma_E\|_{\Ltwo[E]}.
\end{align*}
With the codimension-2 trace inequality {\eqref{eq:trace_inequ_codim2}}
the claim follows.
\end{proof}

\begin{proposition}
  \label{prop:conv_b_term}
 Let $\gpar(t)= \gex + (\gappr-\gex)t$, $\sigma=\gappr-\gex$, and $U\in\Htwoz[\Omega,\R^{\tilde{N}\times\tilde{N}}_{\mathrm{sym}}]$. There holds for all $t\in[0,1]$ for dimension $N\geq3$
\begin{align*}
 \big|b(\gpar(t);\sigma,U)\big|&\lesssim \left(1 +\max_{T \in \T_h}h_T^{-2} \|\gappr-\gex\|_{\Linf[T]}+ \max_{T \in \T_h}h_T^{-1} \|\gappr-\gex\|_{\Winf[T]} \right)\\
    &\quad\times\nrmt{\gappr-\gex}_2\|U\|_{\Htwo}
\end{align*}
and for $N=2$
\begin{align*}
 \left|b(\gpar(t);\sigma,U)\right|\lesssim \left(1 + {\max_{T \in \T_h}h_T^{-1} \|\gappr-\gex\|_{W^{1,\infty}(T)}} \right)\nrmt{\gappr-\gex}_2\|U\|_{\Htwo}.
\end{align*}
Assume that $\gappr=\OptInt[k]\gex$ is an interpolant fulfilling assumptions \eqref{eq:interp_cond} with a $p\in [2,\infty]$. Then for $k\geq 0$ for $N=2$ and $k\geq 1$ for $N=3$
\begin{align*}
\big|b(\gpar(t);\sigma,U)\big|\lesssim \left(\sum_{T\in\T_h}h_T^{p(k+1)}|\gex|^{{p}}_{W^{k+1,p}(T)}\right)^{1/p}\|U\|_{\Htwo}.
\end{align*}
\end{proposition}
\begin{proof}
 Analogously to Proposition~\ref{prop:conv_a_term}, combine Lemmas~\ref{lem:est_b_vol}--\ref{lem:est_b_bbnd} and the approximation properties \eqref{eq:interp_cond} of the interpolant together with \eqref{eq:trafo_2_p}. In the case of $N=2$, Lemma~\ref{lem:est_b_bbnd} does not need to be considered as the codimension 2 terms are zero. 
 Together with the replacement of the $W^{1,\infty}$ by the $\Linf[]$ norm in Lemma~\ref{lem:est_b_bnd}, we obtain convergence also for piecewise constant metrics $k=0$.
\end{proof}

\subsection{Completing the convergence proofs}

\begin{proof}[{\bf{Proof of Theorem~\ref{thm:conv_Riemann}}}]
 For two dimensions $N=2$ we use Proposition~\ref{prop:conv_b_term} noting that $a(\cdot;\cdot,\cdot)=0$ in this case, see {Proposition}~\ref{prop:spec_2d}. In the case $N\geq 3$ combine Proposition~\ref{prop:conv_a_term} and Proposition~\ref{prop:conv_b_term}.
\end{proof}

\begin{proof}[{\bf{Proof of Corollary~\ref{cor:conv_Riemann_int}}}]
 This also follows from Proposition~\ref{prop:conv_a_term} and Proposition~\ref{prop:conv_b_term}.
\end{proof}

\section{The two- and three-dimensional cases}
\label{sec:spec}
\subsection{Specialization to 2D}
\label{sec:spec_2d}

In two dimensions, test space \eqref{eq:testspace_U} consists of piecewise 0-forms, which are globally continuous, $\Utest = \W^0(\T)\cap C^0(\Omega)$, and the elements of \eqref{eq:testspace_A} can be characterized, noting that $\star(X^\flat\wedge Y^\flat) = \volform(X,Y)$, via \eqref{eq:mapping_U_A}. Namely, they are all of the form in~\eqref{eq:Achoice2D} for some $v \in \Utest$. 
In \cite{GNSW2023} the covariant incompatibility operator $\inc:\W^1(T)\odot\W^1(T)\to \R$ has been defined in coordinates
\begin{align*}
  {\inc(\sigma)}
	& =   \gveps^{qi} \gveps^{jk}\Big( \pder{\sigma_{ik}}{jq}
	-  \pder{( \Gamma_{ji}^m \sigma_{mk})}{q}- \Gamma_{lq}^l
	(\pder{\sigma_{ik}}{j} - \Gamma_{ji}^m \sigma_{mk})
	\Big).
\end{align*}
Additionally define the covariant curl of a 2-tensor and 1-form, {$\curl:\TT^2_0(T)\to \W^1(T)$, $\curl:\W^1(T)\to\R$} in the right-handed {$g$-orthonormal} frame $\{\gt,\gn\}$ 
by 
{\begin{align*}
&\curl(\sigma)(X)=
(\nabla_{\gt}\sigma)(\gn,X)-(\nabla_{\gn}\sigma)(\gt,X),\quad &&\sigma\in \TT^2_0(T), X\in \Xm{T}.\\
&\curl(\alpha) = (\nabla_{\gt}\alpha)(\gn)-(\nabla_{\gn}\alpha)(\gt), \quad&& \alpha\in \W^1(T).
\end{align*}}
A coordinate expression can be found in~\cite{GNSW2023}. {Then there holds $\inc=\curl\curl$.}

\begin{proposition}
	\label{prop:spec_2d}
	The distributional densitized Riemann curvature tensor simplifies in 2D to the densitized distributional Gauss curvature (after rescaling by a factor $4$)
	\begin{align*}
\widetilde{\Gauss\volform}(u) = \sum_{T \in \T} \int_T \Gauss u\, \vo{T} + \sum_{F \in \Fint} \int_F \llbracket \GeodCurv \rrbracket u\, \vo{F} + \sum_{E \in \Eint} \Theta_E u(E),\qquad \forall u\in \Utesto(\T),
\end{align*}
	and for the bilinear forms \eqref{eq:def_ah}--\eqref{eq:def_bh} there holds for all $\sigma\in\Regge(\T)$ and $u\in\Utest$
	\begin{align*}
		a(g;\sigma,u)&=0,\\
		b(g;\sigma,u)&= -2\sum_{T\in\T}\int_T\inc\sigma\,u\,\vo{T}+2\sum_{F \in \Fint}\int_F\jmp{\curl(\sigma)(\gt)+\nabla_{\gt}(\sigma_{\gn\gt})}u\,\vo{F}\nonumber\\
		&\qquad- 2\sum_{E \in \Eint}\sum_{F\supset E}\jmp{\sigma_{\gn\gcn}}^E_Fu(E).
	\end{align*}
In particular, $b(g;\cdot,\cdot)$ coincides with the distributional covariant incompatibility operator defined in \cite{GNSW2023} up to a factor $-2$.
\end{proposition}
\begin{proof}
  Let $A = -u\,\og \otimes \og$.
  The reduction to the stated expression of the distributional Gauss curvature was already shown in \S\ref{subsec:spec_gauss_curv}.
	
	To prove that $a(g;\sigma,u)=0$, we start with the first volume term of \eqref{eq:def_ah}
	\begin{align*}
		\langle \Lsigma^{(1)}\Riemann,A\rangle &= \Riemann_{1212}\sigma^1_{\phantom{1}1}A^{1212}+\Riemann_{1221}\sigma^1_{\phantom{1}1}A^{1221}+\Riemann_{2112}\sigma^2_{\phantom{2}2}A^{2112}+\Riemann_{2121}\sigma^2_{\phantom{2}2}A^{2121}\\
		&= 2(\Riemann_{1212}\sigma^1_{\phantom{1}1}A^{1212}+\Riemann_{2112}\sigma^2_{\phantom{2}2}A^{2112})\\
		&=2\Riemann_{1212}A^{1212}\tr{\sigma}=\frac{1}{2}\tr{\sigma}\,\langle\Riemann,A\rangle,
	\end{align*}
which cancels with the second volume term of \eqref{eq:def_ah}. Next, we consider the boundary terms
recalling that $\mathbb{S}_{\Fr}(\sigma)=\sigma_{\Fr}-\tr{\sigma_{\Fr}}g_{\Fr}$
\begin{align*}
\jmp{\sff}:\mathbb{S}_{\Fr}(\sigma):\Atnnt& = (\sigma_{\gt\gt}\jmp{\GeodCurv}-\jmp{\GeodCurv}\sigma_{\gt\gt})A_{\gt\gn\gn\gt}=0.
\end{align*}
The claim therefore follows together with $\tr{\sigma_E}=0$ on 0-dimensional vertices.

For proving the stated  expression for $b(g;\sigma,u)$, we can verify as in~\cite{GNSW2023} that 
\begin{align*}
\langle \nabla^2\sigma,SA\rangle = -(\inc \sigma)\, u.
\end{align*}
Further, on each facet $F$ we have with $A_{\gt\gn\gn\gt}=u$ and $\GeodCurv^{\gn}=\sff^\gn(\gt,\gt)$
\begin{align*}
	&\langle\jmp{\sigma_{\gn\gn}\sff+(\nabla\sigma)_{\Fr\gn \Fr}+\nabla(\gn\lrcorner \sigma)-(\nabla\sigma)_{\gn \Fr\Fr}},\Atnnt\rangle \\
	&\qquad\qquad= \jmp{\sigma_{\gn\gn}\sff_{\gt\gt}+(\nabla_{\gt}\sigma)_{\gn\gt}+\nabla_{\gt}(\gn\lrcorner \sigma)_\gt-(\nabla_{\gn}\sigma)_{\gt\gt}}u\\
	&\qquad\qquad=\jmp{\sigma_{\gn\gn}\,\GeodCurv+(\nabla_{\gt}\sigma)_{\gn\gt}-(\nabla_{\gn}\sigma)_{\gt\gt}+\nabla_{\gt}(\sigma_{\gn\gt})-\sigma_{\gn\gn}\GeodCurv}u\\
	&\qquad\qquad = \jmp{(\nabla_{\gt}\sigma)_{\gn\gt}-(\nabla_{\gn}\sigma)_{\gt\gt}+\nabla_{\gt}(\sigma_{\gn\gt})}u=\jmp{\curl(\sigma)(\gt)+\nabla_{\gt}(\sigma_{\gn\gt})}u,
\end{align*}
where we used {$g(\nabla_{\gt}\gt,\gt)=0$ since $\|\gt\|=1$ along $F$ and also Leibniz' rule for $\nabla$ and \eqref{eq:contraction}, thus} $\nabla_{\gt}\gt = g(\nabla_{\gt}\gt,\gn)\gn=\GeodCurv^\gn\gn$ there holds
\begin{align*}
\nabla_{\gt}(\gn\lrcorner \sigma)_\gt &= \nabla_{\gt}(\sigma_{\gn\gt})-\sigma(\gn,\nabla_{\gt}\gt)=\nabla_{\gt}(\sigma_{\gn\gt})-\sigma_{\gn\gn}g(\nabla_{\gt}\gt,\gn)=\nabla_{\gt}(\sigma_{\gn\gt})-\sigma_{\gn\gn}\GeodCurv^\gn.
\end{align*}
The claim follows now with $A_{\gcn\gn\gn\gcn}=A_{\gt\gn\gn\gt}=u$.
\end{proof}

\subsection{Specialization to 3D}
\label{sec:spec_3d}

On a 3-dimensional manifold, using a natural notion of the curl of a 2-tensor
field, one can define a covariant incompatibility operator. In this
subsection, we show that our $N$-dimensional incompatibility operator
coincides with it when $N=3$. We then show simplifications of our
$N$-dimensional generalized curvature formula in 3D, and display
coordinate formulas, which are also useful for 3D numerical
experimentation (in Section~\ref{sec:num_examples}).

Define the metric-dependent cross product of vector fields yielding 1-forms by \begin{equation}
  \label{eq:3d-cross}
 g( X \times Y, Z) = \og(X, Y, Z), \qquad X, Y, Z \in \Xm \T.  
\end{equation}
where $\omega$ denotes the volume form as usual.
{It is easily seen to satisfy $X\times Y = (\star({X^\flat\wedge Y^\flat}))^\sharp$, and it} reads in coordinates as
$(X\times Y)^i = \gveps^{ijk}(X^\flat)_j(Y^\flat)_k$.  Similar
formulas define the tensor cross product \cite{BGO16} of
a matrix $A$ and vector fields $u,v\in\Xm{\T}$ as follows:
\begin{align}
  \label{eq:tensor-cross}
  \begin{split}
    &{(A\times u)(X,Y):=A(X,u\times Y),\quad (u\times A)(X,Y):=A(X\times u,Y),} \\
    &{(u\times A\times v)(X,Y):=A(X\times u,v\times Y).}
  \end{split}
\end{align}

We start by defining the covariant curl of 2-tensors using the metric-dependent cross product,
$\curl:\TT^2_0(T)\to \TT^2_0(T)$, by
  {\begin{align}
    \label{eq:cov_curl_def}
    (\curl \sigma)(X,Y\times Z) &= (\nabla_Y\sigma)(X,Z)-(\nabla_Z\sigma)(X,Y),
  \end{align}
  for any $\sigma \in \TT^2_0(T)$ and $X,Y,Z \in \Xm T$.
 In Appendix~\ref{sec:geometrical-prelims} we relate the covariant curl to the \emph{exterior covariant derivative}, which will be useful to prove integration by parts formulas later. A relationship between the covariant curl and the exterior covariant derivative in 2D can be found in \cite[Eq.~(4.3) and Remark~4.1]{GNSW2023}.}
 Combining two curl operations in succession, with an intervening transpose,
 we define the
 {\em three-dimensional  covariant incompatibility} operator
  $ \inc:\Sc(T)\to \Sc(T) $ by
  \begin{align*}
 (\inc \sigma)(X,Y) & =  (\curl(\curl \sigma)^T)(X,Y),
    && \sigma \in \Sc(T),\;
 X,Y\in \Xm T,
  \end{align*}
where $(\curl \sigma)^T(X,Y):=(\curl \sigma)(Y,X)$.
{For completeness,} the covariant curl and incompatibility operator read {in coordinates with the covariant Levi-Civita permuting symbol $\gveps^{ijk}$ \eqref{eq:Levi-Civita-symbol} and the Christoffel symbols of the second kind $\Gamma_{ij}^k$}
    \begin{align*}
      \mt{\curl \sigma}_{ij}
      & =
      \gveps^{pql}g_{lj}
 (\pder{\sigma_{iq}}{p} - \Gamma_{pi}^m \sigma_{mq}),\quad\mt{\curl \sigma}_{i}^{\phantom{i}j}
 =
      \gveps^{pqj}
 (\pder{\sigma_{iq}}{p} - \Gamma_{pi}^m \sigma_{mq}),\\
      \mt{\inc\sigma}_{ij}
      & = \gveps^{pql}\gveps^{rst}g_{lj}
 \Big((\d_pg_{ti}-g_{ti}\Gamma_{vp}^v)(\d_r\sigma_{qs}-\Gamma_{rq}^u\sigma_{us})
       \\ 
      & \hspace{2.5cm}
      +g_{ti}\d_p(\d_r\sigma_{qs}-\Gamma_{rq}^u\sigma_{us})-\Gamma_{pi}^mg_{tm}(\d_r\sigma_{qs}-\Gamma_{rq}^u\sigma_{us})\Big),\\
      \mt{\inc\sigma}^{ij}
      & = \gveps^{pqj}
 \big(\gveps^{rsi}(\d_p(\d_r\sigma_{qs}-\Gamma_{rq}^u\sigma_{us})  -\Gamma_{lp}^l(\d_r\sigma_{qs}-\Gamma_{rq}^u\sigma_{us}))\\
                          &\hspace{2.5cm} +\gveps^{rst}\Gamma_{pt}^i(\d_r\sigma_{qs}-\Gamma_{rq}^u\sigma_{us})\big).
    \end{align*}
The Euclidean version of this 3D incompatibility operator (which can be
obtained from the above by setting the metric $g$ to the
identity) has appeared extensively in the elasticity
literature---see e.g.,
\cite{AG20,ArnolAwanoWinth08,ChrisGopalGuzma24}.

{
Next, we proceed to examine the terms of the generalized Riemann curvature
in 3D.  We use the mapping \eqref{eq:mapping_U_A} to identify the test functions $A=\mapUA U\in\Atest$ with $U\in\Utest$ via
\begin{align}
   \label{eq:id_testfunction_3d}
  \begin{split}
 (\mapUA U)(X,Y,Z,W) &= \langle U, \star(X^\flat\wedge Y^\flat)\odot\star(W^\flat\wedge Z^\flat)\rangle
    \\  
    & = \langle U, \star(X^\flat\wedge Y^\flat)\otimes\star(W^\flat\wedge Z^\flat)\rangle \\  
    & = U\big((\star(X^\flat\wedge Y^\flat))^\sharp,(\star(W^\flat\wedge Z^\flat))^\sharp\big)
    \\
      &= - U(X\times Y, Z\times W),
  \end{split}
\end{align}
where we used \eqref{eq:innerprod-extended} and the 3D cross product of \eqref{eq:3d-cross}. It is useful to relate the covariant curl operator \eqref{eq:cov_curl_def} to the  \emph{exterior covariant derivative} $d^{\nabla}$, cf. \eqref{eq:exterior-covariant-derivative}, namely  $(\curl \sigma)(X,Y) = \ip{(\star d^{\nabla}\Lsigma)(X),Y}$ (see Lemma~\ref{lem:ext_cov_der_curl} below), for $\sigma \in \TT^2_0(T), \; X,Y \in \Xm T$. Here $\Lsigma \in\TT^1_1(T)$ is the result of conversion of
$\sigma \in \TT^2_0(T) $ into an endomorphism per \eqref{eq:Ldefn}. The exterior covariant derivative is acting on
  $\Lsigma$ considered as a vector-valued 1-form in $\W^1(T, \Xm
 T)$. Then $d^{\nabla}\Lsigma$ is a vector-valued 2-form in
  $\W^2(T, \Xm T)$.  Using the Hodge dual operator in 3D, we convert it to
 the vector-valued 1-form $\star d^{\nabla}\Lsigma$.}

  \begin{lemma}
  \label{lem:ext_cov_der_curl}
 There holds for vector fields $X,Y,Z,V,W\in\Xm T$ and $\sigma\in\TT^2_0(T)$
  \begin{subequations}
    \begin{align}
    & \ip{(d^{\nabla}\Lsigma)(X,Y),Z} = (\curl\sigma^T)(Z,X\times Y),\\
    & (\curl \sigma)(X,Y) = \ip{(\star d^\nabla \Lsigma)(X),Y},\quad \text{i.e. }L_{\curl \sigma} = \star d^\nabla \Lsigma.\label{eq:curl_ext_cov_der_rel}
    \end{align}
  \end{subequations}
 Furthermore, the following integration by parts formula holds for $\sigma,U\in\TT^2_0(T)$:
  \begin{align}
    \label{eq:ibp_curl}
    \int_T \ip{\curl \sigma, U} \,\vo{T} = \int_T  \ip{ \sigma, \curl U}\,\vo{T} + \int_{\d T}  \ip{ \sigma, U \times\gn} \,\vo{\d T}.
  \end{align}
\end{lemma}
\begin{proof}
Using the definition of $d^{\nabla}$, \eqref{eq:exterior-covariant-derivative-alt}, and the Leibniz rule we have
\begin{align*}
 (d^{\nabla}\Lsigma)(X,Y) &= (\nabla_X \Lsigma)(Y) - (\nabla_Y \Lsigma)(X)\\
  &= \nabla_X(\Lsigma(Y)) - \Lsigma(\nabla_X Y) - \nabla_Y(\Lsigma(X)) + \Lsigma(\nabla_Y X).
\end{align*}
Taking the derivative in direction $Z$ of \eqref{eq:Ldefn}
\begin{align*}
 Z(\sigma(X,Y)) = \ip{\nabla_Z(\Lsigma X), Y} + \ip{\Lsigma X, \nabla_Z Y} = \ip{\nabla_Z(\Lsigma X), Y} + \sigma(X, \nabla_Z Y)
\end{align*}
we can express $\nabla_Z(\Lsigma X)$ yielding the first claim
\begin{align*}
  &\ip{(d^{\nabla}\Lsigma)(X,Y),Z} = \ip{\nabla_X(\Lsigma(Y)) - \Lsigma(\nabla_X Y) - \nabla_Y(\Lsigma(X)) + \Lsigma(\nabla_Y X),Z} \\
  &=  X(\sigma(Y,Z))-\sigma(Y, \nabla_X Z) - \sigma(\nabla_X Y, Z) - Y(\sigma(X,Z)) + \sigma(X, \nabla_Y Z) + \sigma(\nabla_Y X, Z) \\ 
  &= (\nabla_X\sigma)(Y,Z)-(\nabla_Y\sigma)(X,Z)\\
  &= (\nabla_X\sigma^T)(Z,Y)-(\nabla_Y\sigma^T)(Z,X) = (\curl\sigma^T)(Z,X\times Y).
\end{align*}
For the second we use that for 2-forms $\star\star=\idop$ \eqref{eq:hodge_star_twice} and the isometry \eqref{eq:hodge_star_iso}
\begin{align*}
 d^{\nabla}\Lsigma(X,Y)&=\star\star d^{\nabla}\Lsigma(X,Y) = \frac{1}{2}\ip{\star\star d^{\nabla}\Lsigma, X^\flat\wedge Y^\flat} \\
  &= \ip{\star d^{\nabla}\Lsigma, \star(X^\flat\wedge Y^\flat)} = \star d^{\nabla}\Lsigma(X\times Y). 
\end{align*}
Thus, there holds
\begin{align*}
 (\curl\sigma^T)(Z,X\times Y) = \ip{(d^\nabla \Lsigma)(X,Y),Z} = \ip{\star d^{\nabla}\Lsigma(X\times Y),Z}
\end{align*}
and renaming variables yields the claim \eqref{eq:curl_ext_cov_der_rel}.

To prove the integration by parts formula \eqref{eq:ibp_curl} we use \eqref{eq:ibp_ext_cov_der} with $\sigma=\Lsigma\in\W^1(T,\Xm{T})$ and $\rho=\star L_U\in \W^2(T,\Xm{T})$, and $\delta^\nabla$ denoting the exterior covariant coderivative \eqref{eq:codiff_ext_cov_der}
\begin{align*}
  \int_T\frac{1}{2}\ip{d^\nabla\Lsigma,\star L_U}\,\omega_T = \int_T\ip{\Lsigma,\delta^\nabla(\star L_U)}\,\omega_T + \int_{\partial T}\ip{(\Lsigma)_{\Fr},(\gn\lrcorner (\star L_U))_{\Fr}}\,\omega_{\partial T}.
\end{align*}
For $N=3$ and $k=2$ there holds $(-1)^{N(k-1)+1}=1$ and for $k=2$ by \eqref{eq:hodge_star_twice} $\star\star=1$ such that by \eqref{eq:codiff_ext_cov_der} $\delta^\nabla\star L_U = \star d^\nabla L_U$. Further, with \eqref{eq:hodge_star_iso} we have
\begin{align*}
  \int_T\frac{1}{2}\ip{d^\nabla\Lsigma,\star L_U}\,\omega_T = \int_T\ip{\star d^\nabla\Lsigma, L_U}\,\omega_T= \int_T\ip{L_{\curl\sigma}, L_U}\,\omega_T= \int_T\ip{\curl\sigma, U}\,\omega_T
\end{align*}
for the left-hand side and for the first term on the right-hand side
\begin{align*}
\int_T\ip{\Lsigma,\delta^\nabla(\star L_U)} &= \int_T\ip{\Lsigma,\star d^\nabla L_U}\,\omega_T = \int_T\ip{\Lsigma, L_{\curl U}}\,\omega_T = \int_T\ip{\sigma, \curl U}\,\omega_T.  
\end{align*}
The boundary term can be rewritten by using a $g$-orthonormal basis $\{E_i\}_{i=1}^2$ of the tangent space of $\partial T$ 
\begin{align*}
  \int_{\partial T}\ip{(\Lsigma)_{\Fr},(\gn\lrcorner (\star L_U))_{\Fr}}\,\omega_{\partial T} &= \int_{\partial T}\sum_{i=1}^2\ip{\Lsigma(E_i), (\star L_U)(\gn,E_i)}\,\omega_{\partial T}.
\end{align*}
We compute with \eqref{eq:4}, \eqref{eq:hodge_star_iso}, and the definition of the cross product
\begin{align*}
 (\star L_U)(\gn,E_i) = \ip{\star L_U, \gn^\flat\otimes E_i^\flat} =\frac{1}{2}\ip{\star L_U, \gn^\flat\wedge E_i^\flat} = \ip{L_U,\star(\gn^\flat\wedge E_i^\flat)} =  L_U(\gn\times E_i)
\end{align*}
In the first step, since $\star L_U\in\W^2(T,\Xm{T})$ is a vector-valued $2$-form, the bracket $\ip{\star L_U,\gn^\flat\otimes E_i^\flat}$ denotes the inner product taken only over the form part of $\star L_U$, yielding a vector in $\Xm{T}$.
Thus, by \eqref{eq:tensor-cross},
\begin{align*}
  \int_{\partial T}\ip{(\Lsigma)_{\Fr},(\gn\lrcorner (\star L_U))_{\Fr}}\,\omega_{\partial T} &= \int_{\partial T}\sum_{i=1}^2\ip{\Lsigma(E_i), L_U(\gn\times E_i)}\,\omega_{\partial T} \\
  &= \int_{\partial T}\ip{\sigma, U\times \gn}\,\omega_{\partial T},
\end{align*}
finishing the proof.
\end{proof}

The next lemma shows the relationship between the above-defined 3D
$\inc \sigma$ and the previously defined $N$-dimensional $\INC\sigma$
of \eqref{eq:1}. Further, we simplify parts of the boundary terms.
\begin{lemma}
    \label{lem:inc_curl_3d_id}
 Let $T \in \T$, $U\in\Utest$, $A=\mapUA U\in \Atest$ and $F\in\triangle_{-1}T$ with $g$-normal $\gn$. Then, at a point $p$, 
        \begin{align}
          \label{eq:17}
          &\langle\INC\sigma,A\rangle= \langle\inc\sigma,U\rangle,&\quad p\in T,\\
          \label{eq:18}
      &\langle(\nabla\sigma)_{\Fr\gn \Fr}-(\nabla\sigma)_{\gn \Fr\Fr},{\Atnnt}\rangle=\langle ((\curl\sigma)^T\times \gn)_{\Fr},U_{\Fr}\rangle,&\quad p\in F.
        \end{align}
\end{lemma}

\begin{proof}
Let $U$, and therefore $A=\mapUA U$, have compact support. We integrate by parts and use a $g$-orthonormal frame $\{E_i\}_{i=1}^3$ to get
\begin{align*}
  \int_{T}&\ip{\INC \sigma, A}\,\vo{T} = - \ip{\nabla^2 \sigma, SA}\,\vo{T} =  \int_{T}\ip{\nabla\sigma, \div (SA)}\,\vo{T}\\
  &= \int_{T}\sum_{i,j,k,l=1}^3(\nabla_{E_i}\sigma)(E_j,E_k)(\nabla_{E_l}(SA))(E_l,E_i,E_j,E_k) \,\vo{T}\\
  &= \int_{T}\sum_{i,j,k,l=1}^3(\nabla_{E_i}\sigma)(E_j,E_k)(\nabla_{E_l}A)(E_l,E_j,E_i,E_k)\,\vo{T}\\
  &=\int_{T}\sum_{i,j,k,l=1}^3\frac{1}{2}((\nabla_{E_i}\sigma)(E_j,E_k)-(\nabla_{E_k}\sigma)(E_j,E_i))(\nabla_{E_l}U)(E_l\times E_j, E_k\times E_i)\,\vo{T}.
\end{align*}
From \eqref{eq:curl_ext_cov_der_rel} and the symmetry of $\sigma$ and $U$, we have
\begin{align*}
  &(\nabla_{E_i}\sigma)(E_j,E_k)-(\nabla_{E_k}\sigma)(E_j,E_i) = (\curl\sigma)(E_j,E_i\times E_k)),\\
  &\sum_{l=1}^3(\nabla_{E_l}U)(E_l\times E_j, E_k\times E_i) = -(\curl U)(E_k\times E_i,E_j).
\end{align*}
Inserting these into the above expression yields, with integration by parts \eqref{eq:ibp_curl},
\begin{align*}
  \int_{T}\ip{\INC \sigma, A}\,\vo{T} &= -\int_{T}\sum_{i,j,k=1}^3\frac{1}{2}(\curl\sigma)(E_j, E_i\times E_k))(\curl U)(E_k\times E_i,E_j)\,\vo{T} \\
  &= \int_{T}\ip{\curl\sigma, (\curl U)^T}\,\vo{T} = \int_{T}\ip{\inc \sigma, U}\,\vo{T}.
\end{align*}
As this identity holds for all compactly supported $U$, the first claim follows.

For the second claim, we consider a $g$-orthonormal frame $\{E_i\}_{i=1}^2$ of $\Xm{F}$
\begin{align*}
  \ip{(\nabla\sigma)_{\Fr\gn\Fr}-(\nabla\sigma)_{\gn\Fr\Fr},A_{\Fr\gn\gn\Fr}} &= \sum_{i,j=1}^2\left((\nabla\sigma)_{E_i\gn E_j}-(\nabla\sigma)_{\gn E_iE_j}\right)A_{E_i\gn\gn E_j}\\
  &=\sum_{i,j=1}^2(\curl \sigma)(E_j,E_i\times \gn)A_{E_i\gn\gn E_j}
\end{align*}
and by relation \eqref{eq:mapping_U_A}
\begin{align*}
 A_{E_i\gn\gn E_j} = \langle U, (E_i\times \gn)\odot (E_j\times \gn)\rangle = U(E_i\times \gn, E_j\times \gn).
\end{align*}
Defining $E_a = E_i\times \gn$ and $E_b = E_j\times \gn$, 
\begin{align*}
  \ip{(\nabla\sigma)_{\Fr\gn\Fr}-(\nabla\sigma)_{\gn\Fr\Fr},A_{\Fr\gn\gn\Fr}} &= (\curl\sigma)( \gn \times E_b,E_a)U(E_a,E_b) \\
  &= (\curl\sigma)^T( E_a, \gn \times E_b)U(E_a,E_b)\\
  &= \ip{((\curl\sigma)^T\times \gn)_{\Fr}, U_{\Fr}}. 
\end{align*}
Here we have used  \eqref{eq:tensor-cross}.
This yields the claim.
\end{proof}

\begin{lemma}
 When $N=3$, the metric-independent test space equals the Regge space, i.e., 
  \[
    \Utesto = \oRegge(\T).
  \]
\end{lemma}
\begin{proof}
 By the definition of $\Utest$ in~\eqref{eq:testspace_U}, any
  $U \in \Utest = \W^1(\T) \odot \W^1(\T)$ is a linear combination of
 symmetric dyadic products of coordinate 1-forms $dx^i$
 and any element of $\oRegge(\T)$ takes the same form. Since the
 interface continuity conditions and boundary conditions of $ \Utesto$
 and $\oRegge(\T)$ match, they must be in the same space. 
\end{proof}

The curvature operator $\Curvature$ defined in Remark~\ref{rem:curv-oper} simplifies to a symmetric $2$-tensor in three dimensions. In coordinates, it reads, by  \eqref{eq:mapping_U_A_inv_coo}, 
\begin{equation}
  \label{eq:5-Curv-3d}
\Curvature^{ij} = (\mapUA^{-1}\Riemann)^{ij}= -\frac{1}{4}\gveps^{i\,kl}\gveps^{j\,mn}\Riemann_{klmn}.  
\end{equation}
Motivated by~\eqref{eq:17}, let us define a generalized 3D
incompatibility operator $\widetilde{\inc}$ as a linear functional on
the Regge space by
\begin{equation}
  \label{eq:gen-inc-3d}
  \widetilde{\inc \sigma}(U) = \widetilde{\INC \sigma}( \mapUA U),
  \qquad U \in \oRegge(\T),
\end{equation}
where $\widetilde{\INC \sigma}$ is as in Definition~\ref{def:distr-incomp-oper}.

\begin{proposition}
  \label{prop:3d}
 The distributional densitized Riemann curvature tensor  in 3D  yields
 (after rescaling by a factor of $4$) the following densitized distributional curvature $\widetilde{\Curvature \og}$ as a functional on the Regge space:
\begin{align}
  \label{eq:distr_Curvature_op_3d}
  \begin{split}
          \widetilde{{\Curvature}\volform}(U)
 =\sum_{T\in\T}\int_T\langle {\Curvature},U\rangle\,\vo{T}
          & -\sum_{F\in\Fint}\int_F\langle\jmp{\overline{\sff}}, U_{\Fr}\rangle\,\vo{F}
          +\sum_{E\in\Eint}\int_E \Theta_E\, U_{\gt\gt}\,\vo{E},
  \end{split}
\end{align}
for all $U\in\oRegge(\T)$, where $\overline{\sff}^\gn=\mathbb{S}_{\Fr}\sff^\gn=\sff^\gn-H^\gn g_{\Fr}$ is the trace-reversed second fundamental form with $\mathbb{S}_{\Fr}V = V_{\Fr} - \tr{V_{\Fr}}g_{\Fr}$ the trace-reversed part of a 2-tensor $V$ restricted to the facet $F$,
and $\gt = \gn \times \gcn$ is a tangent vector along the edge $E$. 
The   bilinear forms \eqref{eq:def_ah}--\eqref{eq:def_bh} read,  for all $\sigma\in\Regge(\T)$ and $U\in\oRegge(\T)$, as follows:
  \begin{align*}
 a(g;\sigma,U)&=-2\sum_{T\in\T}\int_T{\Curvature}:\sigma:U\,\vo{T}\nonumber\\
    &\quad-2\sum_{F \in \Fint}\int_F \jmp{\sff}:\mathbb{S}_{\Fr}(\sigma):\mathbb{S}_{\Fr}(U)\,\vo{F}\nonumber\\
    &\quad- 2\sum_{E \in \Eint}\int_E\Theta_E\,\sigma_{\gt\gt}\,U_{\gt\gt}\,\vo{E},\\
 b(g;\sigma,U)&= -2\sum_{T\in\T}\int_T\langle \inc \sigma ,U\rangle\,\vo{T} \\
    &\quad+ 2\sum_{F \in \Fint}\int_{F}\langle \llbracket((\curl \sigma)^T \times \gn)_{\Fr}-\sigma_{\gn\gn} \overline{\sff}-\mathbb{S}_{\Fr}\nabla (\gn \lrcorner \sigma) \rrbracket ,U_{\Fr}\rangle \,\vo{F}\\
    &\quad -2\sum_{E\in\Eint}\int_{E}\sum_{F\supset E}\jmp{\sigma_{\gn\gcn}}^E_F U_{\gt\gt}\,\vo{E}.
  \end{align*}
        
 Moreover, the latter
 expression is related to  the generalized 3D covariant incompatibility operator of~\eqref{eq:gen-inc-3d} by 
        \begin{equation}
          \label{eq:b-inc-3d}
 b(g;\sigma,U)=-2\,\widetilde{\inc \sigma}(  U ).
        \end{equation}
\end{proposition}
\begin{proof}
 To prove \eqref{eq:distr_Curvature_op_3d}, we start with the
 identity of Theorem~\ref{thm:distr_Riemann_U}. Note that within each
 element, using \eqref{eq:2} and the definition of $\Curvature$
 in~\eqref{eq:Q-defn}
  \begin{align*}
    \ip{ \Riemann, \mapUA U}
    & = \ip{ \mapUA \Curvature, \mapUA U} =
      \ip{ \star^{\odot 2} \Curvature, \star^{\odot 2} U} =
      4 \ip{ \Curvature,  U},
  \end{align*}
 where the last identity followed from~\eqref{eq:hodge_star_iso} with $N=3$.
 This produces the first term on the right-hand side
 of~\eqref{eq:distr_Curvature_op_3d}.
 To obtain the facet term in \eqref{eq:distr_Curvature_op_3d},
 we use
 Proposition~\ref{prop:mapA_coo}, \eqref{eq:tensor-cross}, and the identities ${\gn\times V\times\gn= \mathbb{S}_{\Fr}V^T}$ and $\langle V_F,\mathbb{S}_{\Fr}W\rangle=\langle \mathbb{S}_{\Fr}V,W_{\Fr}\rangle$ for any 2-tensors $V,W$, which can be {simply shown by computing in an orthonormal basis $E_1$, $E_2$, $E_3$, such that $\gn=E_3$,}
  \begin{align}
    \label{eq:simpl_facet_3d}
    \begin{split}
      \langle\jmp{\sff},(\mapUA U)_{\cdot\gn\gn\cdot}\rangle &= {-\sum_{i,j=1}^3\jmp{\sff}(E_i,E_j)U(E_i\times \gn,\gn\times E_j) = -\ip{\jmp{\sff}, \gn\times U\times \gn}}\\
      &=-\langle\jmp{\sff},\mathbb{S}_{\Fr}U\rangle={-}\langle\jmp{\overline{\sff}},U_{\Fr}\rangle.
    \end{split}
  \end{align}
 Finally, the codimension-2 term in \eqref{eq:distr_Curvature_op_3d} also
 follows {with \eqref{eq:id_testfunction_3d}}
  \[
 (\mapUA U)_{\gcn\gn\gn\gcn} = {-U(\gcn\times\gn,\gn\times\gcn)=U(\gt,\gt),}
  \]
 where $\gt=\gn\times\gcn$ is the tangent vector of the edge
  $E$.

 Next, let us prove the stated expressions for $a$ and $b$.
 By \eqref{eq:id_testfunction_3d} and \eqref{eq:5-Curv-3d},
  \begin{align*}
    \langle \Lsigma^{(1)}\Riemann,\mapUA U\rangle &= \gveps_{ijo}\gveps_{klp}{\Curvature}^{op}\sigma^i_{\phantom{i}a}\gveps^{ajm}\gveps^{kln}U_{mn} \\
    &=2\big(\delta_i^{\phantom{i}a}\delta_{o}^{\phantom{o}m}-\delta_i^{\phantom{i}m}\delta_{o}^{\phantom{o}a}\big){\Curvature}^{on}\sigma^{i}_{\phantom{i}a}U_{mn}\\
    &= 2\tr{\sigma}\langle {\Curvature},U\rangle -2\,{\Curvature}:\sigma:U,
  \end{align*}
where we used the identities $\eveps_{klp}\eveps^{kln}=2\delta_{p}^{\phantom{p}n}$ and $\eveps_{ijo}\eveps^{ajm}=\delta_i^{\phantom{i}a}\delta_{o}^{\phantom{o}m}-\delta_i^{\phantom{i}m}\delta_{o}^{\phantom{o}a}$. The first term cancels with
\[
-\frac{1}{2}\tr{\sigma}\langle\Riemann,\mapUA U\rangle = -2\tr{\sigma}\langle {\Curvature},U\rangle.
\]
For the codimension-1 term we have {by a computation very close} to \eqref{eq:simpl_facet_3d}
\begin{align*}
  &\jmp{\sff}:\mathbb{S}_{\Fr}(\sigma):\Atnnt=-\jmp{\sff}:\mathbb{S}_{\Fr}(\sigma):\mathbb{S}_{\Fr}(U).
\end{align*}
 There holds $(\mapUA U)_{\gcn\gn\gn\gcn}= U_{\gt\gt}$ and $\tr{\sigma_E}=\sigma_{\gt\gt}$, proving the expression of $a(g;\sigma,U)$.
The stated expression for  $b(g;\sigma,U)$ follows from Lemma~\ref{lem:inc_curl_3d_id} and 
\begin{align*}
  \langle{\jmp{\sigma_{\gn\gn}\sff+\nabla (\gn \lrcorner \sigma)}},\Atnnt\rangle=-\langle\jmp{\sigma_{\gn\gn}\overline{\sff}+\mathbb{S}_{\Fr}\nabla (\gn \lrcorner \sigma)}, U_{\Fr}\rangle.
\end{align*}
Finally, \eqref{eq:b-inc-3d} follows from \eqref{eq:rel_b_inc}.
\end{proof}

  \begin{remark}[Relation between  Einstein tensor and curvature operator in 3D]
    In 3D, the curvature operator $\Curvature$ and the Einstein tensor
    $\Einstein$ from \S\ref{ssec:spec-einst-tens} contain the same
    curvature information and are related by
    \begin{align}
      \label{eq:rel_curv_einstein_3d}
      \Curvature = -\Einstein.
    \end{align}
     Hence the $a$ and $b$ bilinear forms from Proposition~\ref{prop:3d} can be rewritten to fit the linearization of the Einstein tensor in \cite[Proposition 4]{GN2023b}.
  \end{remark}

\begin{remark}[Distributional incompatibility in the 3D Euclidean case]
 Assume that the triangulation $\T$ consists of non-curved simplices. In the Euclidean case, when the metric is the identity, 
 Proposition~\ref{prop:3d} and \eqref{eq:b-inc-3d}, after simplifications,
 yield 
  \begin{align}
    \nonumber 
    \widetilde{\inc \sigma}(U) &= \sum_{T\in\T}\int_T \langle \inc \sigma, U\rangle\,dx \\  \nonumber 
                               &\qquad- \sum_{F \in \Fint}\int_{F}\langle \llbracket((\curl \sigma)^T \times \nv)_{\Fr}-\mathbb{S}_{\Fr}(\grad^{\Fr}\sigma_{\nv})\rrbracket,U_{\Fr}\rangle \,ds
    \\ \label{eq:19}
    &\qquad+\sum_{E\in\Eint}\int_{E}\sum_{F\supset E}\jmp{\sigma_{\nv\cnv}}^E_F U_{\tv\tv}\,dl,
  \end{align}
 where $dx$, $ds$, and $dl$ are the volume, surface, and line elements, respectively, and all involved differential operators and tangent, normal, and conormal vectors are the standard Euclidean ones in 3D. 
 Since $\inc$ is a constant-coefficient linear differential operator in the
 Euclidean case, it has a classical generalization as a
 distributional derivative when applied to a  $\sigma$ that is only piecewise smooth, given by
  \[
 (\inc \sigma)^{\mathrm{{distr}}} (\varphi) = \int_\om \sigma : \inc \varphi \; dx
  \]
 for all $\varphi$ in $\D(\om)^{3 \times 3}$ with components in the
 Schwartz space of smooth compactly supported test functions. {Here, $\sigma : \inc \varphi$ denotes the Euclidean Frobenius inner product to distinguish from the $g$-inner product.} Observe that 
  $(\inc \sigma)^{\mathrm{{distr}}} (\varphi)$ equals
  $\widetilde{ \inc \sigma }(\varphi) = \widetilde{\INC \sigma}(
  \mapUA \varphi)$ by Theorem~\ref{thm:distr_inc_curved}. Hence, the linear
 functional in~\eqref{eq:19} gives the distributional inc when
 applied with  $U = \varphi$.
 When $\sigma$ is in the lowest-order (piecewise constant) Regge
 space, equation~\eqref{eq:19} reduces to the formula for (Euclidean)
 distributional incompatibility of $\sigma$ derived in
  \cite{Christiansen11}.
\end{remark}

\section{Numerical examples}
\label{sec:num_examples}

In this section we show that the theoretical convergence rates from
Corollary~\ref{cor:conv_Riemann_int} are sharp, insofar as can be
confirmed by numerical experiments.  All experiments were performed in
the open-source finite element software
NGSolve\footnote{\href{www.ngsolve.org}{www.ngsolve.org}}
\cite{Sch97,Sch14}, where the Regge elements are available. {For reproducibility, the code and computational results are publicly available in~\cite{GNSW2026}.}

\begin{figure}
	\centering
	\resizebox{0.32\textwidth}{!}{
		\begin{tikzpicture}
			\begin{loglogaxis}[
				legend style={at={(0,0)}, anchor=south west},
				xlabel={ndof},
				ylabel={error},
				ymajorgrids=true,
				grid style=dotted,
				]
				
				\addlegendentry{$k=0$}
				\addplot[color=red, mark=*, style=solid] coordinates {
          ( 19,0.1735945878157066 )
          ( 98,0.1056338764706522 )
          ( 604,0.041566277769064144 )
          ( 4184,0.017331318735084915 )
          ( 31024,0.007968506297828625 )
          ( 238688,0.003840425598163186 )
          ( 1872064,0.0018772505771210096 )
				};
			
				\addlegendentry{$k=1$}
				\addplot[color=blue, mark=square, style=solid] coordinates {
          ( 92,0.10411183998473542 )
          ( 556,0.026680712573369153 )
          ( 3800,0.012171551590059986 )
          ( 27952,0.0038088797898044677 )
          ( 214112,0.0010196778606744962 )
          ( 1675456,0.00026512340491882306 )
          ( 13255040,6.76940799655966e-05 )
				};
			
				\addlegendentry{$k=2$}
				\addplot[color=olive, mark=asterisk, style=solid] coordinates {
          ( 255,0.013642216690081433 )
          ( 1662,0.01810793201311883 )
          ( 11892,0.002398290997098561 )
          ( 89736,0.00032783517957240627 )
          ( 696720,4.3564231097657597e-05 )
          ( 5489952,5.631758766755594e-06 )
				};
			
				\addplot[color=black, mark=none, style=dashed] coordinates {
					( 1500, {0.5*1500^(-1/3)} )
					( 1500000, {0.5*1500000^(-1/3)} )
				};
				
				\addplot[color=black, mark=none, style=dashed] coordinates {
					( 20000, {4.8*20000^(-2/3)} )
					( 6000000, {4.8*6000000^(-2/3)} )
				};
				\addplot[color=black, mark=none, style=dashed] coordinates {
					( 2000, {10*2000^(-3/3)} )
					( 1000000, {10*1000000^(-3/3)} )
				};
				
			\end{loglogaxis}
			
			\node (A) at (4.1, 4.3) [] {$\mathcal{O}(h)$};
			\node (B) at (6., 2.5) [] {$\mathcal{O}(h^2)$};
			\node (C) at (3.3, 1.9) [] {$\mathcal{O}(h^3)$};
	\end{tikzpicture}}
  \resizebox{0.32\textwidth}{!}{
		\begin{tikzpicture}
			\begin{loglogaxis}[
				legend style={at={(0,0)}, anchor=south west},
				xlabel={ndof},
				ylabel={error},
				ymajorgrids=true,
				grid style=dotted,
				]
				
				\addlegendentry{$k=0$}
				\addplot[color=red, mark=*, style=solid] coordinates {
          ( 19,0.15675687375128403 )
          ( 98,0.19708572593311852 )
          ( 604,0.07692285724282293 )
          ( 4184,0.030933448662126092 )
          ( 31024,0.01389985175909611 )
          ( 238688,0.006436391986807926 )
          ( 1872064,0.0030392063087964715 )
				};
			
				\addlegendentry{$k=1$}
				\addplot[color=blue, mark=square, style=solid] coordinates {
          ( 92,0.15524643189045076 )
          ( 556,0.0580620689593139 )
          ( 3800,0.024501279985712 )
          ( 27952,0.007513270149880071 )
          ( 214112,0.001949348144236988 )
          ( 1675456,0.000498750440839081 )
          ( 13255040,0.0001260940929733197 )
				};
			
				\addlegendentry{$k=2$}
				\addplot[color=olive, mark=asterisk, style=solid] coordinates {
          ( 255,0.018738480196020094 )
          ( 1662,0.02648419388659905 )
          ( 11892,0.003651546247308392 )
          ( 89736,0.000502765202136504 )
          ( 696720,6.686208978839148e-05 )
          ( 5489952,8.668883837873378e-06 )
				};
			
				\addplot[color=black, mark=none, style=dashed] coordinates {
					( 1500, {0.5*1500^(-1/3)} )
					( 1500000, {0.5*1500000^(-1/3)} )
				};
				
				\addplot[color=black, mark=none, style=dashed] coordinates {
					( 20000, {3*20000^(-2/3)} )
					( 5000000, {3*5000000^(-2/3)} )
				};
				\addplot[color=black, mark=none, style=dashed] coordinates {
					( 2000, {10*2000^(-3/3)} )
					( 1000000, {10*1000000^(-3/3)} )
				};
				
			\end{loglogaxis}
			
			\node (A) at (6.1, 3.5) [] {$\mathcal{O}(h)$};
			\node (B) at (6.3, 1.4) [] {$\mathcal{O}(h^2)$};
			\node (C) at (3.4, 1.9) [] {$\mathcal{O}(h^3)$};
	\end{tikzpicture}}
  \resizebox{0.32\textwidth}{!}{
		\begin{tikzpicture}
			\begin{loglogaxis}[
				legend style={at={(0,0)}, anchor=south west},
				xlabel={ndof},
				ylabel={error},
				ymajorgrids=true,
				grid style=dotted,
				]
				
				\addlegendentry{$k=0$}
				\addplot[color=red, mark=*, style=solid] coordinates {
          ( 19,0.08681582022700111 )
          ( 98,0.15150999301006415 )
          ( 604,0.07920239305749883 )
          ( 4184,0.0339898038868941 )
          ( 31024,0.014746393023173874 )
          ( 238688,0.006425835570241218 )
          ( 1872064,0.0028596004509526584 )
				};
			
				\addlegendentry{$k=1$}
				\addplot[color=blue, mark=square, style=solid] coordinates {
          ( 92,0.16135411037843114 )
          ( 556,0.11520261404791875 )
          ( 3800,0.035052777354636404 )
          ( 27952,0.008946751093607548 )
          ( 214112,0.0023452729651180157 )
          ( 1675456,0.0006019269106392641 )
          ( 13255040,0.00015263309030639677 )
				};
			
				\addlegendentry{$k=2$}
				\addplot[color=olive, mark=asterisk, style=solid] coordinates {
          ( 255,0.03158042255873671 )
          ( 1662,0.024910140458075186 )
          ( 11892,0.004917377736299004 )
          ( 89736,0.0006341032420584942 )
          ( 696720,8.662024609649966e-05 )
          ( 5489952,1.1367445970430913e-05 )
				};
			
				\addplot[color=black, mark=none, style=dashed] coordinates {
					( 1500, {0.5*1500^(-1/3)} )
					( 1500000, {0.5*1500000^(-1/3)} )
				};
				
				\addplot[color=black, mark=none, style=dashed] coordinates {
					( 20000, {3*20000^(-2/3)} )
					( 5000000, {3*5000000^(-2/3)} )
				};
				\addplot[color=black, mark=none, style=dashed] coordinates {
					( 2000, {10*2000^(-3/3)} )
					( 1000000, {10*1000000^(-3/3)} )
				};
				
			\end{loglogaxis}
			
			\node (A) at (6.1, 3.5) [] {$\mathcal{O}(h)$};
			\node (B) at (6.3, 1.4) [] {$\mathcal{O}(h^2)$};
			\node (C) at (3.4, 1.9) [] {$\mathcal{O}(h^3)$};
	\end{tikzpicture}}
	
	\caption{Convergence of the distributional curvature operator ${\Curvature}$ {for the first (left), second (middle), and third (right) setting} in the $H^{-2}(\Omega)$-norm for $N=3$ with respect to the number of degrees of freedom (ndof) of $\gappr\in\Regge_h^k$ for $k=0,1,2$.}
	\label{fig:conv_plot}
\end{figure}

\begin{table}
	\centering
  {
    \begin{tabular}{cccc}
                & $k=0$ & $k=1$ & $k=2$\\
                \hline
                $h$ & \begin{tabular}{@{}ll@{}}
                        Error &\hspace{0.2in} Order \\
                \end{tabular} & \begin{tabular}{@{}ll@{}}
                        Error &\hspace{0.2in} Order \\
                \end{tabular} & \begin{tabular}{@{}ll@{}}
                        Error &\hspace{0.2in} Order \\
                \end{tabular} \\
                \hline
                \begin{tabular}{@{}l@{}}
                        $3.46\cdot 10^{-0}$\\
                        $1.77\cdot 10^{-0}$\\
                        $9.99\cdot 10^{-1}$\\
                        $5.00\cdot 10^{-1}$\\
                        $2.69\cdot 10^{-1}$\\
                        $1.36\cdot 10^{-1}$\\
                        $6.85\cdot 10^{-2}$\\
                \end{tabular} & \begin{tabular}{@{}ll@{}}
                $1.74\cdot 10^{-1}$ & \\
                $1.06\cdot 10^{-1}$ & 0.74\\
                $4.16\cdot 10^{-2}$ & 1.64\\
                $1.73\cdot 10^{-2}$ & 1.26\\
                $7.97\cdot 10^{-3}$ & 1.25\\
                $3.84\cdot 10^{-3}$ & 1.07\\
                $1.88\cdot 10^{-3}$ & 1.04\\
                \end{tabular} & \begin{tabular}{@{}ll@{}}
                $1.04\cdot 10^{-1}$ & \\
                $2.67\cdot 10^{-2}$ & 2.02\\
                $1.22\cdot 10^{-2}$ & 1.38\\
                $3.81\cdot 10^{-3}$ & 1.68\\
                $1.02\cdot 10^{-3}$ & 2.13\\
                $2.65\cdot 10^{-4}$ & 1.97\\
                $6.77\cdot 10^{-5}$ & 1.99\\
                \end{tabular} & \begin{tabular}{@{}ll@{}}
                $1.36\cdot 10^{-2}$ & \\
                $1.81\cdot 10^{-2}$ & -0.42\\
                $2.40\cdot 10^{-3}$ & 3.54\\
                $3.28\cdot 10^{-4}$ & 2.88\\
                $4.36\cdot 10^{-5}$ & 3.26\\
                $5.63\cdot 10^{-6}$ & 3.00\\\\
                \end{tabular} \\\\
\end{tabular}
\begin{tabular}{cccc}
                & $k=0$ & $k=1$ & $k=2$\\
                \hline
                $h$ & \begin{tabular}{@{}ll@{}}
                        Error &\hspace{0.2in} Order \\
                \end{tabular} & \begin{tabular}{@{}ll@{}}
                        Error &\hspace{0.2in} Order \\
                \end{tabular} & \begin{tabular}{@{}ll@{}}
                        Error &\hspace{0.2in} Order \\
                \end{tabular} \\
                \hline
                \begin{tabular}{@{}l@{}}
                        $3.46\cdot 10^{-0}$\\
                        $1.77\cdot 10^{-0}$\\
                        $9.99\cdot 10^{-1}$\\
                        $5.00\cdot 10^{-1}$\\
                        $2.69\cdot 10^{-1}$\\
                        $1.36\cdot 10^{-1}$\\
                        $6.85\cdot 10^{-2}$\\
                \end{tabular} & \begin{tabular}{@{}ll@{}}
                $1.57\cdot 10^{-1}$ & \\
                $1.97\cdot 10^{-1}$ & -0.34\\
                $7.69\cdot 10^{-2}$ & 1.65\\
                $3.09\cdot 10^{-2}$ & 1.32\\
                $1.39\cdot 10^{-2}$ & 1.29\\
                $6.44\cdot 10^{-3}$ & 1.13\\
                $3.04\cdot 10^{-3}$ & 1.10\\
                \end{tabular} & \begin{tabular}{@{}ll@{}}
                $1.55\cdot 10^{-1}$ & \\
                $5.81\cdot 10^{-2}$ & 1.46\\
                $2.45\cdot 10^{-2}$ & 1.51\\
                $7.51\cdot 10^{-3}$ & 1.71\\
                $1.95\cdot 10^{-3}$ & 2.18\\
                $4.99\cdot 10^{-4}$ & 2.00\\
                $1.26\cdot 10^{-4}$ & 2.01\\
                \end{tabular} & \begin{tabular}{@{}ll@{}}
                $1.87\cdot 10^{-2}$ & \\
                $2.65\cdot 10^{-2}$ & -0.51\\
                $3.65\cdot 10^{-3}$ & 3.47\\
                $5.03\cdot 10^{-4}$ & 2.87\\
                $6.69\cdot 10^{-5}$ & 3.25\\
                $8.67\cdot 10^{-6}$ & 2.99\\\\
                \end{tabular} \\\\
\end{tabular}
\begin{tabular}{cccc}
                & $k=0$ & $k=1$ & $k=2$\\
                \hline
                $h$ & \begin{tabular}{@{}ll@{}}
                        Error &\hspace{0.2in} Order \\
                \end{tabular} & \begin{tabular}{@{}ll@{}}
                        Error &\hspace{0.2in} Order \\
                \end{tabular} & \begin{tabular}{@{}ll@{}}
                        Error &\hspace{0.2in} Order \\
                \end{tabular} \\
                \hline
                \begin{tabular}{@{}l@{}}
                        $3.46\cdot 10^{-0}$\\
                        $1.77\cdot 10^{-0}$\\
                        $9.99\cdot 10^{-1}$\\
                        $5.00\cdot 10^{-1}$\\
                        $2.69\cdot 10^{-1}$\\
                        $1.36\cdot 10^{-1}$\\
                        $6.85\cdot 10^{-2}$\\
                \end{tabular} & \begin{tabular}{@{}ll@{}}
                $8.68\cdot 10^{-2}$ & \\
                $1.52\cdot 10^{-1}$ & -0.83\\
                $7.92\cdot 10^{-2}$ & 1.14\\
                $3.40\cdot 10^{-2}$ & 1.22\\
                $1.47\cdot 10^{-2}$ & 1.35\\
                $6.43\cdot 10^{-3}$ & 1.22\\
                $2.86\cdot 10^{-3}$ & 1.18\\
                \end{tabular} & \begin{tabular}{@{}ll@{}}
                $1.61\cdot 10^{-1}$ & \\
                $1.15\cdot 10^{-1}$ & 0.50\\
                $3.51\cdot 10^{-2}$ & 2.09\\
                $8.95\cdot 10^{-3}$ & 1.97\\
                $2.35\cdot 10^{-3}$ & 2.16\\
                $6.02\cdot 10^{-4}$ & 1.99\\
                $1.53\cdot 10^{-4}$ & 2.00\\
                \end{tabular} & \begin{tabular}{@{}ll@{}}
                $3.16\cdot 10^{-2}$ & \\
                $2.49\cdot 10^{-2}$ & 0.35\\
                $4.92\cdot 10^{-3}$ & 2.84\\
                $6.34\cdot 10^{-4}$ & 2.96\\
                $8.66\cdot 10^{-5}$ & 3.21\\
                $1.14\cdot 10^{-5}$ & 2.98\\\\
                \end{tabular} 
\end{tabular}
  }
	\caption{Convergence of the distributional curvature operator $\Curvature$ {in the first (top), second (middle), and third (bottom) setting.} Same as Figure~\ref{fig:conv_plot}, but in tabular form.}
	\label{tab:error_N3}
\end{table}

{We consider three examples in dimension $N=3$ on the unit cube $\Omega=(-1,1)^3$. In the first two cases, the Riemannian metric $g$ is induced by an embedding into $\R^4$. More precisely, we consider embeddings of the form 
$$(x,y,z)\mapsto (x,y,z,f_i(x,y,z)).$$ 
For the first example, we take 
$$f_1(x,y,z):= \frac{1}{2}(x^2+y^2+z^2)-\frac{1}{12}(x^4+y^4+z^4)$$ 
as proposed in \cite{GN2023}. The second embedded manifold is obtained by perturbing this embedding function and is given by 
$$f_2(x,y,z):= f_1(x,y,z) + \alpha(xy+yz+xz)+\beta xyz$$ 
with $\alpha=0.2$ and $\beta=0.12$. In contrast to these two examples, the third manifold is not given via an embedding but is specified directly through its Riemannian metric. Here, we consider the conformally flat metric 
$$g = \exp(2u)I_{3\times 3},\qquad u:=\epsilon(xy+yz+zx+yxz),$$ 
with $\epsilon=0.23$. In all three cases, we employ the equivalent formulation of the curvature operator \eqref{eq:distr_Curvature_op_3d}. We emphasize that the test function $A$ and formulation \eqref{eq:distr_Riemann} can also be used.} The components of the curvature operator in the first setting read
\begin{align*}
	{\Curvature^{xx}} &= \frac{9(z^2-1)(y^2-1)}{\det(\gex)\big(q(x)+q(y)+q(z)+9\big)},\\
	{\Curvature^{yy}} &= \frac{9(z^2-1)(x^2-1)}{\det(\gex)\big(q(x)+q(y)+q(z)+9\big)},\\
	{\Curvature^{zz}} &= \frac{9(x^2-1)(y^2-1)}{\det(\gex)\big(q(x)+q(y)+q(z)+9\big)},\\
	{\Curvature^{xy}}&={\Curvature^{xz}}={\Curvature^{yz}}=0,
\end{align*}
where $q(x) = x^2 (x^2-3)^2$. {Note that this curvature operator is a diagonal matrix. The second curvature operator, $\Curvature=\Curvature^{ij}\d_i\odot \d_j$, also has off-diagonal entries with
\begin{align*}
  \Curvature = \frac{1}{(1+S)^2}\begin{pmatrix}BC-r^2 & qr - Cp & pr-Bq\\
    qr - Cp&AC-q^2 & pq-Ar\\
    pr-Bq &pq-Ar & AB-p^2
  \end{pmatrix},
\end{align*}
where $A=1-x^2$, $B=1-y^2$, $C=1-z^2$, $p=\alpha +\beta z$, $q=\alpha +\beta y$, $r=\alpha +\beta x$, and $S = \|\nabla f_2\|^2$ with
\begin{align*}
  \nabla f_2 = \begin{pmatrix}
    x\frac{2+A}{3}+\alpha(y+z) +\beta yz \\ y\frac{2+B}{3}+\alpha(x+z) +\beta xz \\ z\frac{2+C}{3}+\alpha(x+y) +\beta xy
  \end{pmatrix}.
\end{align*} 
The curvature operator of the third example reads
\begin{align*}
  \Curvature = e^{-4u}\begin{pmatrix}
    -\epsilon^2 A^2 & \epsilon(1+z) - \epsilon^2AB & \epsilon(1+y) - \epsilon^2AC \\
                     \epsilon(1+z) - \epsilon^2AB& -\epsilon^2 B^2 & \epsilon(1+x) - \epsilon^2 BC \\
                   \epsilon(1+y) - \epsilon^2AC  &  \epsilon(1+x) - \epsilon^2 BC & -\epsilon^2 C^2
  \end{pmatrix},
\end{align*}
with $A = y+z+yz$, $B = x+z+xz$, and $C = x+y+xy$. Note that for $\epsilon<1$ the off-diagonal entries are dominant compared to the diagonal ones.}

We compute the $H^{-2}(\Omega)$-norm of the error $f:= \widetilde{{\Curvature}\omega}(\gappr)-{\Curvature}\omega(\gex)$ by using that $\|f\|_{H^{-2}(\Omega,\R^{3\times 3}_{\mathrm{sym}})}$ is equivalent to $\|V\|_{\Htwo[\Omega,\R^{3\times 3}_{\mathrm{sym}}]}$, where $V\in \Htwoz[\Omega,\R^{3\times 3}_{\mathrm{sym}}]$ solves the biharmonic equation $\Delta^2 V = f$ applied to each component. This equation will be solved numerically using the (Euclidean) Hellan--Herrmann--Johnson (HHJ) method \cite{Com89} for each component of $V$ (Although the HHJ method was originally defined only for two-dimensional domains, it is straightforward to extend it to arbitrary dimensions, see e.g. \cite{li18}.). To avoid that the discretization error spoils the real error, we use for $V_h$ two polynomial orders more than for $\gappr\in\Regge_h^k$.

We consider a structured mesh consisting of $6\cdot 2^{3i}$ tetrahedra, with $\tilde{h}=\max_T h_T=\sqrt{3}\,2^{1-i}$ (and minimal edge-length $2^{1-i}$) for $i=0,1,\dots$. We perturb each component of the inner mesh vertices by a random number drawn from a uniform distribution in the range $[-\tilde{h}\,2^{-3.5},\tilde{h}\,2^{-3.5}]$ to avoid possible superconvergence due to mesh symmetries. As shown in Figure~\ref{fig:conv_plot} and displayed in Table~\ref{tab:error_N3}, we obtain linear convergence when $\gappr\in\Regge_h^k$ has polynomial degree $k=0$. For $k=1$ and $k=2$, higher convergence rates are obtained as expected. Therefore, Corollary~\ref{cor:conv_Riemann_int} is sharp. For $k=0$ we observe numerically linear convergence in all settings, which is better than predicted by Corollary~\ref{cor:conv_Riemann_int}. Further investigations suggest, however, that the observed linear convergence for $k=0$ is only pre-asymptotic, i.e., the final rate of convergence (or in this case no convergence) has not been reached yet. A detailed test of whether Lemma~\ref{lem:est_a_bbnd}, Lemma~\ref{lem:est_b_bbnd}, and the sum of both are sharp has been performed in \cite{GN2023b} for the first setting with the embedded manifold $f_1$ in terms of the Einstein tensor  which is related to the curvature operator in three dimensions as mentioned in~\eqref{eq:rel_curv_einstein_3d} and is therefore omitted here.

\appendix

\section{Summary of geometric notions used}
\label{sec:geometrical-prelims}

Let $(\Omega, g)$ be an oriented Riemannian manifold with $\Omega\subset \R^N$. In this appendix (only), the
metric $g$ is smooth. Here we gather the standard notions in
Riemannian geometry that we have used in previous sections, point out
our choice of conventions or normalizations when multiple
options exist, and provide references.

The value of a $(k, l)$-tensor field $\rho \in \TM{l}{k}$ acting on
$k$ vectors $X_i \in {\XM}$ and $l$ covectors $\mu_j\in \WM{1}$ is
denoted by $\rho(X_1, \ldots, X_k,\mu_1, \ldots, \mu_l)$.  Note that
$\WM{1} = \TM{0}{1}$ and ${\XM} = \TM{1}{0}$. Note also that it is
standard to extend the Levi-Civita connection $\nabla$ from vector
fields to tensor fields (see e.g., \cite{Lee18}) so that
Leibniz rule holds, i.e. for $A\in\TM{l}{k}$
  \begin{align*}
 (\nabla_XA)(Y_1,\dots&,Y_k,\alpha_1,\dots,\alpha_l)= X(A(Y_1,\dots,Y_k,\alpha_1,\dots,\alpha_l))\\
    &-A(\nabla_XY_1,\dots,Y_k,\alpha_1,\dots,\alpha_l)-\dots -A(Y_1,\dots,Y_k,\alpha_1,\dots,\nabla_X\alpha_l).
  \end{align*}
Notice that any symmetries of $A$ are preserved by $\nabla_X A$. We define the $(k+1,l)$-tensor $\nabla A$ by $(\nabla A)(X,\dots) = (\nabla_XA)(\dots)$.
Higher-order operators can be defined inductively via $\nabla^{k+1}A = \nabla(\nabla^{k}A)$. We frequently use the second  covariant derivative
\begin{equation}
  \label{eq:2ndDerivative}
 (\nabla^2_{X,Y}A)(\dots):=(\nabla^2A)(X,Y,\dots)  
\end{equation}
where we have selected the convention of placing the subscripts $X, Y$ as the first two arguments (rather than the last two, as done in \cite[p.~99]{Lee18}).

We use standard operations such as the tensor product $\otimes:\TM{l}{k}\times\TM{q}{p} \to\TM{l+q}{k+p}$,
the tangent to cotangent isomorphism
$\flat: {\XM} \to \WM{1}$, the reverse operation
$\sharp: \WM{1} \to {\XM}$, 
and  the wedge product $\wedge$ with the normalization convention set to the so-called determinant convention,  so that, e.g., 
\begin{equation}
  \label{eq:4}
  \varphi \wedge \eta = \varphi \otimes \eta - \eta \otimes \varphi,
  \qquad \varphi, \eta \in \W^{1}(\om).
\end{equation}
All these are exactly as in
standard texts~\cite{Lee18, Peter16}, where one can also
find the definition of the Hodge dual (Hodge star) $\star:\W^k(\om)\to\W^{N-k}(\om)$, namely, 
\begin{equation}
  \label{eq:Hodgestar}
 g^{-1}( \star \eta, {\zeta} )\, \omega = \eta \wedge {\zeta}, \qquad  \eta \in \W^k(\om),
  {\zeta} \in \W^{N-k}(\om).
\end{equation}
The unique volume form over an oriented Riemannian manifold $D$ is denoted by $\vol{D}$, and when $D=\om$, we simply abbreviate $\omega_\om$ to $\omega$, as in~\eqref{eq:Hodgestar}. There, as usual, the inner product 
$g^{-1}(\varphi^i, \phi^j) = g((\varphi^i)^\sharp, (\phi^j)^\sharp)$ between  co-vectors $\varphi^i$ and $\phi^j$ is extended to $k$-forms
by
$
 g^{-1}(\varphi^1 \wedge \cdots \wedge \varphi^k,
  \phi^1 \wedge \cdots \wedge \phi^k) = \det ( g^{-1}( \varphi^i, \phi^j)).
$
One can prove, using the properties of the wedge product, that 
\eqref{eq:Hodgestar} implies 
\begin{align*}
 g^{-1}(\star \eta, {\zeta}) = (-1)^{k(N-k)} g^{-1}( \eta, \star {\zeta}), 
  \qquad  \eta \in \W^k(\om),
  {\zeta} \in \W^{N-k}(\om).
\end{align*}
Given a $g^{-1}$-orthonormal co-vector basis $e^i$ of matching
orientation, \eqref{eq:Hodgestar} implies
\begin{equation}
  \label{eq:Hodge-on-frame}
  \star ( e^1 \wedge \cdots \wedge e^k) = e^{k+1} \wedge \cdots \wedge e^N.
\end{equation}

{Denote by $\eveps^{i_1\dots i_{N}}=\eveps_{i_1\dots i_{N}}$ the standard permutation
symbol, whose value is $1$, $-1$, or $0$, when $(i_1,\dots, i_{N})$ is
an even, odd, or not a permutation of $(1,\dots,N)$, respectively. Then the covariant version of the Levi-Civita symbol is defined as
\begin{align}
  \label{eq:Levi-Civita-symbol}
  \gveps^{i_1\dots i_{N}} = \sqrt{\det
 g}^{-1}\eveps^{i_1\dots i_{N}} ,\qquad \gveps_{i_1\dots i_{N}} = \sqrt{\det
g}\,\eveps_{i_1\dots i_{N}}.
\end{align} 
}

Given a general $\eta \in \W^k(\om)$, a coordinate frame $\d_i$ and its
coframe $dx^j$,
the Hodge~dual can be computed in terms of
components $\eta_{i_1 \dots i_k} = \eta(\d_{i_1}, \dots, \d_{i_k})$ to get 
\begin{equation}
  \label{eq:Hodgestar-comp}
  \star\eta
 = 
  \frac{\sqrt{\det g}}{ (N-k)! k!}\;
  \eta_{m_1\dots m_k} 
 g^{m_1 i_1} \dots g^{m_k i_k}
  \eveps_{i_1\dots i_k j_1\dots j_{N-k}}\;
 dx^{j_1} \wedge\dots\wedge dx^{j_{N-k}}.
\end{equation}
In particular, {using $\eta=dx^p \wedge dx^q\in \W^2(\Omega)$ and that $(dx^p\wedge dx^q)_{mn}=\delta_m^p\delta_n^q-\delta_m^q\delta_n^p$ yields after contracting the Kronecker deltas}
\begin{equation}
  \label{eq:Hodge-dxpdxq}
  \star( dx^p \wedge dx^q )
 = \frac{\sqrt{\det g}}{(N-2)!} \, g^{rp}g^{sq}\,
    \eveps_{rs j_1 \dots j_{N-2}} dx^{j_1} \wedge \dots \wedge dx^{j_{N-2}}.
\end{equation}
Using \eqref{eq:Hodgestar-comp},
one can prove that for vector fields $X_i, Y_j \in \Xm \Omega$,
\begin{equation}
  \label{eq:Hodgestar-vecs}
  \star ( X_1^\flat \wedge X_2^\flat \wedge \cdots \wedge {X_k^\flat})
 (Y_1, \ldots, Y_{N-k}) =
  \og(X_1, X_2, \ldots, X_k,Y_1, \ldots, Y_{N-k}).
\end{equation}
{Applying the Hodge-star twice yields $\star\circ\star=\pm 1$, more precisely,
\begin{align}
  \label{eq:hodge_star_twice}
  \star\star \eta = (-1)^{k(N-k)},\qquad \eta \in \W^k(\om).
\end{align}
}
The notation $\langle X,Y\rangle:=g(X,Y)$ denotes the $g$-inner product for
two vector fields $X,Y\in{\XM}$, as well as its extension to general
tensors $A,B\in\TM{l}{k}$ through tensor product compatibility. Then, 
using coordinates of $k$-covariant tensors $A$ and $B$,
\begin{equation}
  \label{eq:innerprod-extended}
  \ip{A, B} = A_{i_1 \ldots i_k} g^{i_1 j_1} \,\ldots\, g^{i_k j_k} B_{j_1 \ldots j_k}.
\end{equation}
Note that for $k$-forms $\eta, \varphi \in \W^k(\om)$, this inner
product and previously mentioned $g^{-1}$~inner product are related by~\cite[Exercise~2-17]{Lee18}
\begin{equation}
  \label{eq:factor-ip}
 g^{-1}( \eta, \varphi)= \frac{1}{k!} \ip{\eta, \varphi}.
\end{equation}
It is easy to see from~\eqref{eq:Hodgestar} that 
the Hodge dual is an isometry in the $g^{-1}$~inner
product. Hence,~\eqref{eq:factor-ip} implies that for $k$-forms
$\eta, \vphi$, we have
$$\ip{ \star \eta,\star \vphi} = (N-k)! \; g^{-1}( \star \eta,\star \vphi)
= (N-k)! \; g^{-1}(  \eta, \vphi).$$
Therefore,
\begin{align}
  \label{eq:hodge_star_iso}
  \langle \star \eta,\star \vphi\rangle = \frac{(N-k)!}{k!}\langle \eta,\vphi\rangle,\qquad \eta, \varphi \in \W^k(\om),
\end{align}
showing that the Hodge dual is a quasi-isometry in the tensor inner product.

The \emph{exterior covariant derivative}, see e.g. {\cite[Section 2.2.2.2]{Peter16}},
\begin{align}
  \label{eq:exterior-covariant-derivative}
 d^\nabla:\W^k(\Omega,\XM)\to \W^{k+1}(\Omega,\XM)
\end{align}
extends the exterior derivative $d$ from differential forms to e.g., vector-valued differential forms. For $k=0$ it coincides with the exterior derivative, $d^\nabla=d$, and the exterior covariant derivative fulfills the {Leibniz} rule
\begin{align*}
 d^\nabla (\alpha\wedge \eta) = d\alpha\wedge \eta + (-1)^{k} \alpha\wedge d^\nabla \eta,\qquad\forall \alpha \in \W^k(\Omega),\, \eta \in \W^l(\Omega,\XM),
\end{align*}
allowing also for an inductive definition of $d^\nabla$. The Hodge dual
in \eqref{eq:Hodgestar} can be readily extended to   $\star:\W^k(\Omega,\XM)\to \W^{N-k}(\Omega,\XM)$ by noting that 
\[
  \W^k(\Omega,\XM)\simeq \W^k(\Omega)\otimes \XM \text{ and } \W^{N-k}(\Omega,\XM)\simeq \W^{N-k}(\Omega)\otimes \XM,
\]
and performing the standard Hodge star operation on the alternating part. {
An alternative but equivalent definition of the exterior covariant derivative skew-sym\-metrizes the covariant derivative in all components \cite[Eq.~(2.16)]{Taylor2011}, i.e., for $\eta \in \W^k(\Omega,\XM)$, 
\begin{align*}
  (d^\nabla \eta)(X_0,X_1,\dots,X_k)
  = & \sum_{i=0}^k (-1)^i X_{i}(\eta(X_{0},\dots,\widehat{X}_i,\dots,X_{k}))\\
  + & \sum_{i<j}(-1)^{i+j}\eta([X_i,X_j],X_0,\dots,\widehat{X}_i,\dots,\widehat{X}_j,\dots,X_k),
\end{align*}
where $\widehat{X}_j$ indicates that the $X_j$ argument is omitted. Because $\nabla$ is a torsion-free connection, the exterior covariant derivative reads more compactly \cite[p. 691]{kupfermanDoubleFormsRegular2024}
\begin{align}
  \label{eq:exterior-covariant-derivative-alt}
 (d^\nabla \eta)(X_0,X_1,\dots,X_k) &= \sum_{i=0}^k (-1)^i (\nabla_{X_{i}}\eta)(X_{0},\dots,\widehat{X}_i,\dots,X_{k}).
\end{align}

The inner product on $\W^k(\Omega,\XM)$ is defined by
\begin{align*}
  \frac{1}{k!}\ip{\sigma,\rho}\,\omega=g^{-1}(\sigma,\rho)\,\omega = \sigma\wedge \star \rho,
\end{align*}
(cf.~\eqref{eq:Hodgestar} and \eqref{eq:factor-ip})
where the wedge pairing $\wedge: \W^{k}(\om,\Xm{\om})\times \W^{l}(\om,\Xm{\om})\to \W^{k+l}(\om)$ is defined for $\sigma= X\otimes \alpha\in\W^{k}(\om,\Xm{\om})$ and $\rho = Y\otimes \beta\in\W^{l}(\om,\Xm{\om})$ by 
\begin{align*}
  \sigma\wedge \rho = \ip{X,Y}\, \alpha\wedge \beta.
\end{align*}
For any  $\sigma\in\W^{k-1}(\om,\Xm{\om})$ and $\rho\in\W^{k}(\om,\Xm{\om})$ the following  integration by parts formula (the $1/k!$ factors are due to \eqref{eq:factor-ip}) \cite[eq. (2.6)]{kupfermanDoubleFormsRegular2024} holds, 
\begin{equation}
  \label{eq:ibp_ext_cov_der}
  \begin{aligned}
    \int_{\om} \frac{1}{k!} & \ip{d^\nabla\sigma,\rho}\, \omega_{\om}  \\
                        &= \int_{\om}\frac{1}{(k-1)!}\ip{\sigma,\delta^\nabla\rho}\,\omega_{\om} + \int_{\partial \om}\frac{1}{(k-1)!}\ip{\sigma_{\Fr},(\gn\lrcorner \rho)_{\Fr}}\,\omega_{\partial \om},
  \end{aligned}
\end{equation}
where the \emph{exterior covariant coderivative} $\delta^\nabla$ is the formal adjoint of $d^\nabla$ with respect to the $L^2$ inner product on $\om$ \cite[p. 693]{kupfermanDoubleFormsRegular2024}, i.e. \eqref{eq:ibp_ext_cov_der} with compactly supported functions $\sigma$ and $\rho$ such that the boundary term vanishes. Further, $\sigma_{\Fr}$ denotes the restriction of $\sigma$ on the tangent space of $\partial \om$. There holds for the exterior covariant coderivative the relation \cite[pp. 10--11]{eellsSelectedTopicsHarmonic1983}
\begin{align}
  \label{eq:codiff_ext_cov_der}
  \delta^\nabla = (-1)^{N(k-1)+1}\star d^\nabla \star.
\end{align}
}
The covariant divergence of a tensor $A \in \TT_0^k(\om)$, with $k \ge 1$,  is defined as the trace of the covariant derivative of $A$ in its first two components, $\div A := \tr[12]{\nabla A}$. (Note that with this convention, in components, the divergence is applied to the first index of $A$.) We neglect the subscripts of the trace operator when there is no possibility of confusion. Equivalently, given a $g$-orthonormal basis $E_i$,
\begin{equation}
  \label{eq:div-def}
 (\div A )(X_1, X_2, \ldots) = \sum_{i=1}^N (\nabla A)( E_i, E_i, X_1, X_2, \ldots)
\end{equation}
for any $X_i \in \Xm \om$.

Next, recall the classical Stokes theorem {\cite[Theorem~16.11]{Lee12b}, \cite[Chapter 8]{Spi1999}}
for $(N-1)$-forms {on domains $\Omega$ with a piecewise smooth boundary}. Applying it to the Hodge dual of a 1-form $\alpha$,
we get the divergence theorem for smooth $1$-forms on Riemannian
manifolds:
\begin{align}
  \label{eq:divergence-thm}
  \int_{\Omega} (\div \alpha)\; \og_\om = -\int_{\d \Omega} \alpha(\gn)\,\og_{\d \Omega},\quad \forall \alpha\in\WM{1},
\end{align}
where $\gn$ is the inward $g$-normalized unit normal of the boundary $\d\Omega$
with the induced orientation. (Further standard assumptions needed for the 
existence of the integrals in the Stokes theorem,
such as compactness of $\om$ or the support of $\alpha$,
are tacitly placed throughout.) Now, for arbitrary tensors
$A\in\TT_0^{k+1}(\Omega)$ and $B\in\TT_0^k(\Omega),$ we can produce a
1-form by 
$\theta(X) = \ip{ X \lrcorner A, B}$, using the standard interior product 
\begin{equation}
  \label{eq:contraction}
 (X \lrcorner A)(Y, \dots) = A(X, Y, \dots).
\end{equation}
Then 
applying~\eqref{eq:divergence-thm} to $\theta$, we obtain the
integration by parts formula
\begin{align} \label{eq:ibp_volume}
  \int_{\Omega} \langle A,\nabla B\rangle\, \omega_\om
 =-\int_{\Omega}\langle \div A,B\rangle\, \og_\om
  -\int_{\d \Omega}\langle A,\gn^\flat\otimes B\rangle\,\og_{\d \Omega}.
\end{align}

Next, consider an $(N-1)$-dimensional submanifold $F$ of $\Omega$ with
unit normal vector $\gn$ and let $\{\gt_1,\dots,\gt_{N-1},\gn\}$ be a
smooth oriented $g$-orthonormal frame on $F$.  Then define the
following operations on $F$ for ambient tensors $A \in \TT_0^k(\om)$,
$k \ge 1$:
\begin{align}
  \label{eq:divF-def}
  (\divFR A )(Z_1, Z_2, \ldots)
  & =
    \sum_{i=1}^{N-1} (\nabla A)( \gt_i, \gt_i, Z_1, Z_2, \ldots), 
  \\     \label{eq:11-F}
  (\nablaFR A)(X, Z_1, Z_2 \dots)
  & = (\nabla A)(X, Z_1, Z_2, \dots)
  \\     \label{eq:11}
  (\nablaF A)(Z_1, Z_2, Z_3, \dots)
  & = (\nabla A)(Q Z_1, Z_2, Z_3, \dots)
\end{align}
for any ambient $Z_i \in \Xm \om$ and tangential $X \in \Xm F$
where
$  \Proj Z_1 = Z_1 - g(Z_1, \gn) \gn$.
Clearly, \eqref{eq:divF-def} and \eqref{eq:div-def} imply
\begin{align}
  \label{eq:divF-div_relation}
  \div A = \divFR A + \gn \lrcorner \nabla_{\gn}A, \qquad A \in \TT_0^k(\Omega).
\end{align}
Note that $\nablaFR$ does not generally equal the Levi-Civita
connection $\nabla_F$ on $F$, nor does $\divFR$ equal the intrinsic
surface divergence $\div_F$ defined by applying~\eqref{eq:div-def} on
the $(N-1)$-manifold $F$ with $\nabla_F$ in place of $\nabla$.
We may use $\div^F$ to express the mean curvature: indeed, since the second fundamental form (in~\eqref{eq:sff-def}) equals
$\sff^{\gn} = -\nablaFR \gn^\flat$,
the mean curvature of $F$ is given by
\begin{equation}
  \label{eq:H}
  \begin{aligned}
    &H^{\gn} := \tr{\sff^{\gn}}=-\sum_{i=1}^{N-1}g(\nabla_{\gt_i}\gn,\gt_i)=-\divFR\gn^\flat.
  \end{aligned}          
\end{equation}
Here, we have adopted the convention that makes $H^{\gn}$ positive for a sphere
with an inward-pointing normal vector. 

An integration by parts formula on $F$ using $\divFR$, in
contrast to \eqref{eq:ibp_volume}, contains a term with the mean
curvature $H^{\gn}$ of $F$. First, consider a 1-form
$\alpha \in \W^1(\om)$.  Its projection $\alpha_\Fp \in \W^1(\om)$ and
restriction $\alpha_F \in \W^1(F)$ are respectively
$\alpha_\Fp(Z) = \alpha (Q Z)$ and $\alpha_F(X) =\alpha(X)$ for
ambient $Z \in \Xm \om$ and tangential $X \in \Xm F$.
Then $\alpha = \alpha_\Fp + \alpha(\gn) \gn^\flat$. It is easy to see that
\begin{equation}
  \label{eq:divFR-spilt}
  \divFR \alpha = \divFR \alpha_\Fp  + \alpha(\gn) \, \divFR \gn^\flat.
\end{equation}
The scalar function $\divFR \alpha_\Fp$ is also readily seen to equal
the intrinsic surface divergence of the restriction 
\begin{equation}
  \label{eq:divFR-divF}
  \divFR \alpha_\Fp = \div_F \,\alpha_F.
\end{equation}
Furthermore, by~\eqref{eq:divergence-thm} applied to $\alpha_F$ on the
manifold~$F$, we have
\begin{equation}
  \label{eq:12}
  \int_F (\div_F \,\alpha_F)\; \og_F = {-}\int_{\d F} \alpha_F(\gcn)\; \og_{\d F},
\end{equation}  
where $\gcn$ denotes the {inward}-pointing
$g$-normalized conormal vector on $\d F$. Combining \eqref{eq:H},
\eqref{eq:divFR-spilt}, \eqref{eq:divFR-divF}, and 
\eqref{eq:12}, we obtain
\begin{align}
  \label{eq:surface_stokes}
  \int_F(\divFR\alpha)\,  \og_F = -\int_F H^{\gn}\,\alpha(\gn)\,\og_F
  -\int_{\d F} \alpha(\gcn)\,\og_{\d F}, \quad \forall \alpha\in \WM{1}.
\end{align}

Now, we can perform an argument similar to what we used to go
from~\eqref{eq:divergence-thm} to~\eqref{eq:ibp_volume}. Namely,
  given smooth tensors $A\in\TM{0}{k+1}$ and $B\in\TM{0}{k},$ making
  the 1-form $\theta \in \W^1(F)$ by
  $\theta(X) = \ip{ X \lrcorner A, B}$ for $X \in \Xm F$, one can
  show that $\divFR \theta = \ip{ \divFR A, B} + \ip{ A, \nablaFR
    B}$. Thus, applying~\eqref{eq:surface_stokes} to $\theta$, we obtain the surface integration by parts formula
\begin{equation}
  \label{eq:ibp_surface}
  \begin{aligned}    
    \int_F \langle A,\nablaFR B\rangle\, \og_F
 = &  -\int_F \langle \divFR A,B\rangle\,\og_{F}
      -\int_{\d F}\langle A,\gcn^\flat\otimes B\rangle\,\og_{\d F}
    \\
    & -\int_FH^{\gn}\,\langle A,\gn^\flat\otimes B\rangle\, \og_F, 
  \end{aligned}  
\end{equation}
for any smooth ambient tensors $A\in\TM{0}{k+1}$ and $B\in\TM{0}{k}$.

\section{Summary of numerical analysis notions used}
\label{sec:numerical-analysis-prelims}

Standard references to numerical analysis and finite element methods include \cite{agmonLecturesEllipticBoundary2010,Bra2007, BS2008, Cia1978, ernFiniteElementsApproximation2021}. This appendix fixes the finite element notation and approximation facts used in Section~\ref{sec:num_ana}. It is not meant to replace the standard references.
Throughout this appendix, $\Omega\subset\R^N$ is bounded and equipped with the triangulations used in Section~\ref{sec:num_ana}. The trace, scaling, inverse, and approximation estimates used below are applied elementwise on the simplices, facets, and codimension-two entities of the mesh. Thus the constants depend on the uniform element shape-regularity, not on a global Lipschitz regularity assumption for $\partial\Omega$.
For $s\ge 0$ and $1\le p\le\infty$, $W^{s,p}(\Omega)$ denotes the Sobolev--Slobodeckij space of differentiability index $s\in [0,\infty)$ and integrability index $p\in [1,\infty]$ with norm and seminorm
$\|\cdot\|_{W^{s,p}(\Omega)}$ and $|\cdot|_{W^{s,p}(\Omega)}$ (with respect to the Euclidean metric) \cite{Ada2003,Gri1985}.
For integer $s=m$, the seminorm $|\cdot|_{W^{m,p}}$ contains only derivatives of order $m$. For non-integer $s$ it is the usual Slobodeckij seminorm. Seminorms appear in local approximation estimates, such as Bramble--Hilbert estimates.
We write $L^p(\Omega)=W^{0,p}(\Omega)$ and $H^s(\Omega)=W^{s,2}(\Omega)$.
For $k\in\N$, $H_0^k(\Omega)$ is the closure of $\D(\Omega)$, the space of smooth compactly supported test functions, with respect to the $H^k$-norm and $H^{-k}(\Omega)=(H_0^k(\Omega))'$ denotes the topological dual space with the canonical duality pairing $\langle \cdot,\cdot\rangle_{H^{-k}, H_0^k}$ and norm
\begin{align}
  \label{eq:hmknorm}
  \|f\|_{H^{-k}(\Omega)}
 = \sup_{0\neq u\in H_0^k(\Omega)} \frac{\langle f,u\rangle_{H^{-k},H_0^k}}{\|u\|_{H^k(\Omega)}}.
\end{align}
The definitions extend componentwise to vector/tensor fields, e.g.\ $H^2(\Omega,\R^{N\times N})$.

Let $(\Omega,g)$ be an $N$-dimensional Riemannian manifold with metric tensor $g$. If $D\subset\Omega$ is a submanifold with induced metric $g|_D$, then for an $(l,k)$-tensor $\rho\in\TT^l_k(D)$ we set
\begin{align}
  \label{eq:lp-norm}
  \|\rho\|_{L^p(D,g)}=
  \begin{cases}
 \bigl(\int_D |\rho|_g^p\,\omega_D(g)\bigr)^{1/p}, & 1\le p<\infty,\\[0.25em]
    \operatorname*{ess\,sup}_D |\rho|_g, & p=\infty,
  \end{cases}
\end{align}
where $\omega_D(g)$ is the induced volume form and $|\rho|_g=\langle \rho,\rho\rangle^{1/2}$ the norm induced by the $g$-inner product. When we work with the Euclidean case we abbreviate $\|\cdot\|_{L^p(D)}=\|\cdot\|_{L^p(D,\delta)}$.

Let $\T$ be a shape-regular triangulation of $\Omega$, and we write $T\in\T$ for elements, $F\subset\partial T$ for $(N\!-\!1)$-facets, and $E$ for $(N\!-\!2)$-entities (e.g., edges in 3D, vertices in 2D). Let $h_T=\operatorname{diam}(T)$ and $h=\max_{T\in\T} h_T$ denote the local and global meshsize, respectively. Shape-regularity means there is $\gamma\ge1$ such that $h_T/\rho_T\le \gamma$ for all $T$, where $\rho_T$ is the inradius. For a family $\{\T_h\}_{h>0}$, shape-regularity means that the same $\gamma$ works for every element of every triangulation in the family. This excludes degenerating simplices with arbitrarily small angles and is the only mesh-quality assumption entering our estimates. Consequently, constants hidden in $\lesssim$ may depend on $\gamma$ but not on $h$. Moreover, neighboring elements have comparable diameters and the number of elements meeting at a fixed codimension-2 entity is uniformly bounded.
Standard affine pullbacks to a fixed reference simplex $\hat T$ yield the usual scaling estimates. For example, let $0\le m\le s$ and $1\le p\le\infty$. Then, for $v\circ\Phi =\hat v$ with $v\in W^{m,p}(T)$, $\hat v\in W^{s,p}(\hat T)$, and affine $\Phi:\hat{T}\to T$ there holds, see e.g. \cite[Section 4.3--4.4]{BS2008}
\begin{align}
  \label{eq:scaling}
 |\hat v|_{W^{m,p}(\hat T)} \simeq h_T^{m-\frac{N}{p}}\,|v|_{W^{m,p}(T)}
\end{align}
with constants hidden in $\simeq$ and $\lesssim$ depending only on $N$, $m$, $s$, $p$, and the shape-regularity of $\T$.

On a single element $T\in\T$, the (continuous) trace inequality for $v\in H^1(T)$ and any facet $F\subset\partial T$ is \cite[Theorem 3.10]{agmonLecturesEllipticBoundary2010}
\begin{align}
  \label{eq:trace_inequ_standard}
  \|v\|_{L^2(F)} \le C\Big(h_T^{-\frac12}\|v\|_{L^2(T)} + h_T^{\frac12}|v|_{H^1(T)}\Big),
\end{align}
with a constant $C$ independent of $h_T$. Repeating \eqref{eq:trace_inequ_standard} and assuming $H^2$-regularity of $v$ on $T$ yields the codimension-2 trace inequality for $E\in\triangle_{-2}(T)$,
\begin{align}
  \label{eq:trace_inequ_codim2}
  \|v\|_{L^2(E)}
  \le C\Big(h_T^{-1}\|v\|_{L^2(T)} + |v|_{H^1(T)} + h_T\,|v|_{H^2(T)}\Big).
\end{align}

For $k\in\N$, denote by $\Pol^k(T)$ the scalar polynomials of degree $\le k$ on $T$. If $X$ is a finite-dimensional vector space, such as $\R^N$, $\Sc$, or another tensor space, then $\Pol^k(T,X):=\Pol^k(T)\otimes X$ denotes the $X$-valued polynomial functions on $T$.
Let $v_h\in \Pol^k(T)$.
For $0\le r\le m$ and $1\le q\le p\le\infty$, there holds the inverse estimate, cf. \cite[Lemma 4.5.3]{BS2008},
\begin{align*}
 |v_h|_{W^{m,p}(T)}
  \le C\, h_T^{\,r-m+N\bigl(\frac1p-\frac1q\bigr)}\;|v_h|_{W^{r,q}(T)}
\end{align*}
and the discrete trace inequality
\begin{align*}
  \|v_h\|_{L^2(F)} \le C\, h_T^{-\frac12}\, \|v_h\|_{L^2(T)}, \qquad F\subset\partial T.
\end{align*}
Let $m\in\N$, $1\le p\le\infty$, and $T\in\T$. For $v\in W^{m,p}(T)$, there holds the Bramble--Hilbert lemma \cite{BH1970}
\begin{align*}
  \inf_{q\in \Pol^{m-1}(T)} \|v-q\|_{W^{r,p}(T)}
  \le C\, h_T^{\,m-r}\,|v|_{W^{m,p}(T)}, \qquad 0\le r\le m,
\end{align*}
with $C$ depending only on $N,m,p$ and the shape-regularity of $\T$.
This underpins all local finite element approximation estimates.

For $k\in\N$, the Lagrange finite element space is
defined using the element map $\Phi_T:\hat T\to T$ from the reference simplex to the physical element. The expression $v|_T\circ\Phi_T$ is the pullback of the restriction of $v$ to $T$. It is required to be a polynomial on the fixed reference simplex:
\begin{align}
  \label{eq:lag_fe}
  \VV_h^k := \{\, v \in C^0(\Omega) : \forall T\in\T,\, v|_T\circ\Phi_T\in\Pol^k(\hat{T})\} \subset \Hone[\Omega].
\end{align}
For smooth enough $u$, standard interpolation operators into $\VV_h^k$ satisfy the usual local estimates. Denoting such an operator by $\mathcal J_h$, one has, for $T\in\T$, \cite[Theorem 4.4.4]{BS2008},
\begin{align*}
 |u - \mathcal{J}_h u|_{W^{t,p}(T)} \le C\, h_T^{\,s-t}\, |u|_{W^{s,p}(T)},
  \qquad 0\le t\le s\le k+1,\quad 1\le p\le\infty,
\end{align*}
with real (possibly non-integer) $t$ and $s$, and $C$ depending only on $N,k,p$ and the shape-regularity. This yields the global estimate
\begin{align}
  \label{eq:lagrange-global}
  \|u - \mathcal{J}_h u\|_{W^{t,p}(\Omega)} \le C\, h^{\,s-t}\, |u|_{W^{s,p}(\Omega)}.
\end{align}

If $u$ has low regularity, one can use local, stable quasi-interpolation operators $\Pi_h$ (Cl\'ement \cite{Cle1975}) or $\mathcal{Z}_h$ (Scott--Zhang \cite{SZ1990}) defined via local averaging.
For every element $T$ and its patch $\omega_T=\bigcup\{\,K\in\T:\overline{K}\cap\overline{T}\neq\emptyset\,\}$, \cite[eq. (4.8.10)]{BS2008},
\begin{align}
  \label{eq:clement-sz}
 |u - \mathcal{I}_h u|_{W^{t,p}(T)} \le C\, h_T^{\,s-t}\, |u|_{W^{s,p}(\omega_T)},
  \quad \mathcal{I}_h\in\{\Pi_h,\mathcal{Z}_h\},\quad 0\le t\le s\le k+1,
\end{align}
with stability $\|\mathcal{I}_h u\|_{W^{t,p}(T)}\le C\|u\|_{W^{t,p}(\omega_T)}$ for each $t\in\{0,1\}$.
Combining \eqref{eq:clement-sz} over all $T$ gives the same global order as in \eqref{eq:lagrange-global}.
Let $Q_h^k:=\{v\in L^2(\Omega): \forall T\in\T,\, v|_T\circ\Phi_T\in\Pol^k(\hat{T})\}$ be the discontinuous space. The Oswald operator $O_h: Q_h^k\to \VV_h^k$ \cite{oswaldBPXpreconditionerP1Elements1993} \cite[Section 22.2]{ernFiniteElementsApproximation2021} is defined by averaging degrees of freedom associated with shared mesh entities.
It is locally stable, for $1\le p\le\infty$,
\begin{align*}
\|O_h v_h\|_{L^p(\Omega)} \le C \|v_h\|_{L^p(\Omega)},\qquad
 |O_h v_h|_{H^1(\Omega)} \le C\, |v_h|_{H^1_h} ,
\end{align*}
where $|v_h|_{H^1_h}^2:=\sum_{T\in\T} |v_h|_{H^1(T)}^2$ is the broken $H^1$ seminorm. This idea of averaging degrees of freedom to construct Oswald-type interpolation operators extends to other finite element spaces; see e.g.~\cite[Appendix A]{GN2023} for the Regge finite element space.

Let $\Sc$ denote the space of symmetric $N\times N$ tensors and $\Regge_h^k$ the Regge space of order $k$, cf. \eqref{eq:regge_fem_space}.
The canonical Regge interpolant $\RegInt[k]: W^{s,p}(\Omega,\Sc)\to\Regge_h^k$ (defined for $s>1/p$) in dimension $N=2$ is characterized on each triangle $T$ by the edge and cell moment conditions \cite{li18}
\begin{subequations}
  \label{eq:RegInt}
  \begin{align}
    \int_E (\RegInt[k]\sigma)_{\tv\tv}\, q \, \mathrm{d}l
    &= \int_E \sigma_{\tv\tv}\, q \, \mathrm{d}l
    &&\forall q\in\Pol^k(E),\ \forall E\subset\partial T, \\
    \int_T \RegInt[k]\sigma : \rho \, \mathrm{d}x
    &= \int_T \sigma : \rho \, \mathrm{d}x
    &&\forall \rho\in \Pol^{k-1}(T,\Sc),\, T\in\T,
  \end{align}
\end{subequations}
where $\tv$ is a Euclidean unit tangent vector along edge $E$. The extension to $N\ge 3$ is straightforward. By Bramble--Hilbert and scaling, $\RegInt[k]$ is locally stable and provides the approximation, for $p\in [1,\infty]$, $s\in (1/p,k+1]$, $t\in [0,s]$, and $\sigma\in W^{s,p}(T,\Sc)$, \cite[Theorem 2.5]{li18}
\begin{align*}
  \|\sigma - \RegInt[k]\sigma\|_{W^{t,p}(T)}
  \le C\, h_T^{\,s-t}\, \|\sigma\|_{W^{s,p}(T)},
\end{align*}
with constants depending only on $N, k,s,t$ and the shape-regularity of $T$.

\section*{Acknowledgments}
This work was supported in part by the Austrian Science Fund (FWF) projects \href{https://doi.org/10.55776/F65}{10.55776/F65} and \href{https://doi.org/10.55776/J4824}{10.55776/J4824},
and the National Science Foundation (USA) grant DMS-2409900. For open-access purposes, the authors have applied a CC BY public copyright license to any Author Accepted Manuscript (AAM) version arising from this submission. {We thank the two anonymous reviewers for their careful reading of the manuscript and their valuable comments and suggestions that helped to improve the presentation of the paper.} In particular, references~\cite{cheegerHodgeTheoryRiemannian1980, MeyerToubiana22} were pointed out by one of the referees, whose comments also led to Remark~3.8, and whose encouragement led us to refactor several coordinate-based proofs into intrinsic arguments.

\bibliographystyle{acm}
\bibliography{cites}

\end{document}